\documentclass[11pt]{article}    

\usepackage[margin=2.35cm]{geometry}
\usepackage{graphicx}
\usepackage{amsmath,amsfonts,amsthm,mathrsfs,amssymb}
\usepackage{epstopdf}
\usepackage{hyperref}
\usepackage{soul}
\usepackage{caption}
\usepackage{subcaption}
\usepackage{comment}
\usepackage{tikz}
\newcommand{\x}{\mrm{x}}
\newcommand{\eps}{\varepsilon}
\usetikzlibrary{angles,quotes}

\newtheorem{theorem}{Theorem}[section]
\newtheorem{lemma}[theorem]{Lemma}
\newtheorem{remark}[theorem]{Remark}
\newtheorem{definition}[theorem]{Definition}

\newtheorem{proposition}[theorem]{Proposition}

\theoremstyle{definition}
\newtheorem{example}[theorem]{Example}

\usepackage{xcolor}

\newcommand{\Blu}[1]{{#1}}

\newcommand{\pare}[1]{\left(#1\right)}
\newcommand{\Pare}[1]{\big(#1\big)}
\newcommand{\PAre}[1]{\Big(#1\Big)}
\newcommand{\abs}[1]{\left\lvert #1 \right\rvert}

\newcommand{\AnD}{\quad\text{and}\quad}
\newcommand{\AND}{\qquad\text{and}\qquad}

\newcommand{\RR}{\mathbb{R}}
\newcommand{\NN}{\mathbb{N}}

\newcommand{\CC}{\mathbb{C}}

\newcommand{\In}{\mathrm{in}}

\newcommand{\resp}{\textit{resp. }}

\DeclareMathOperator{\Id}{Id}

\usepackage{hyperref}
\usepackage[numbers]{natbib}
\hypersetup{
	pdfpagemode=UseThumbs,
	pdftoolbar=true,        
	pdfmenubar=true,        
	pdffitwindow=false,     
	pdfstartview={Fit},    
	pdftitle={Examples of non-scattering inhomogeneities},    
	pdfauthor={L. Chesnel, H. Haddar, H. Li, J. Xiao},     
	pdfsubject={},  
	pdfcreator={L. Chesnel, H. Haddar, H. Li, J. Xiao},   
	pdfproducer={L. Chesnel, H. Haddar, H. Li, J. Xiao}, 
	pdfkeywords={}, 
	pdfnewwindow=true,      
colorlinks=true,       
linkcolor=magenta,          
citecolor=red,        
filecolor=cyan,      
urlcolor=blue           
}
\newcommand{\Om}{\Omega}
\newcommand{\om}{\omega}
\newcommand{\mL}{\mrm{L}}
\newcommand{\mrm}[1]{\mathrm{#1}}
\newcommand{\dsp}{\displaystyle}
\renewcommand{\div}{\mrm{div}}
\newcommand{\mH}{\mrm{H}}

\begin{document}
~\vspace{0.0cm}
\begin{center}
{\sc \bf\huge\selectfont Examples of non-scattering inhomogeneities}\\[16pt]
\end{center}
\begin{center}
\textsc{Lucas Chesnel}$^1$, \textsc{Houssem Haddar}$^1$, \textsc{Hongjie Li}$^{2}$, \textsc{Jingni Xiao}$^{3}$\\[16pt]
\begin{minipage}{0.91\textwidth}
		{\small
$^1$ IDEFIX, Ensta Paris, Institut Polytechnique de Paris, Palaiseau, France;\\
$^2$ Yau Mathematical Sciences Center, Tsinghua University, Beijing, China;\\
$^3$ Department of Mathematics, Drexel University, Philadelphia, Pennsylvania, USA.\\[10pt]
			E-mails:
 \texttt{lucas.chesnel@inria.fr}, \texttt{houssem.haddar@inria.fr}, \texttt{hongjieli@tsinghua.edu.cn}, \texttt{jingni.xiao@drexel.edu}.\\[-14pt]
			\begin{center}
				(\today)
			\end{center}
		}
	\end{minipage}
\end{center}
\textbf{Abstract.} We consider the scattering of waves by a penetrable inclusion embedded in some reference medium. We exhibit examples of materials and geometries for which non-scattering frequencies exist, i.e., for which at some frequencies there are incident fields which produce null scattered fields outside of the inhomogeneity. We show in particular that certain domains with corners or even cusps can support non-scattering frequencies. We relate the latter, for some inclusions, to resonance frequencies for Dirichlet or Neumann cavities. We also find situations where incident non-scattering fields solve the Helmholtz equation  in a neighbourhood of the inhomogeneity and not in the whole space. In relation with invisibility, we give examples of inclusions of anisotropic materials which are non-scattering for all real frequencies. We prove that corresponding material indices must have a special structure on the boundary. 
\\[5pt]

\section{Introduction}

The so-called transmission eigenvalues play an important role in the study of inverse scattering from inhomogeneous media. They can be helpful in addressing theoretical questions such as uniqueness of the perturbation or in the justification of some reconstruction methods such as the Linear Sampling Method \cite{CoKi96,CaCoHa23}. These special frequencies can also be exploited in imaging algorithms for highly cluttered media \cite{Audi2021,Audi2024}. In general, transmission eigenvalues are associated with eigenfunctions that cannot be exactly represented as solutions to the Helmholtz equation in the whole space. In that case, exact non-scattering does not occur. Proving that incident fields always scatter in certain configurations has been the subject of several works and is usually associated with the presence of some geometric singularities such as corners in the boundary of the inhomogeneity. The techniques and results  differ according to the considered model for the propagation of waves inside the inclusion, namely 
\begin{equation}\label{CaseN}
\Delta u+k^2qu=0
\end{equation}
or 
\begin{equation}\label{CaseAq}
\div(A\nabla u)+k^2qu=0.
\end{equation}
Here $k$ denotes the frequency\footnote{\Blu{We note that $k$ is generally referred to as the wavenumber or spatial frequency, which is proportional to the (time) frequency in the time-harmonic case.}} and $A$, $q$ are functions characterizing the physical properties of the materials. We refer the reader respectively to \Blu{\cite{BlPS14,PaSV14,ElHu18,CaVo23,SaSh21,KLSS24} and \cite{BlLiXi21,CaXi21,Xiao22,CakoniVX23,KSS24} for studies concerning cases \eqref{CaseN} and \eqref{CaseAq}}. In particular, it is proven for model \eqref{CaseN} that if there is a corner in the boundary of the support of $q-1$, then at any frequency an incident wave will always produce a non zero scattered field outside the inhomogeneity, provided that $q$ is sufficiently regular and different from the background coefficient $1$ near the corner. 
However, things can be different for model \eqref{CaseAq} especially when $A$ is different from the background coefficient $\mrm{Id}$ near the boundary of the support of the inhomogeneity. The main goal of the present article is to provide examples for which non-scattering occur for model \Blu{\eqref{CaseAq} when $A\not\equiv \mrm{Id}$ in the inhomogeneous medium.} We show that this can happen even when the domain has singularities like corners or cusps. This analysis can be helpful in better understanding the optimality of some results concerning the absence of non-scattering frequencies in the literature, and provide some insight on the difference between cases (\ref{CaseN}) and (\ref{CaseAq}). We also give examples of situations where non-scattering incident fields have some singularities outside the inhomogeneity. This motivates the definition of non-scattering frequencies adopted below.\\
\newline
Let us first describe the scattering problem we are considering. It is associated with the scalar acoustic wave equation in $\RR^d$, $d\ge2$. We assume that some inclusion is located in a bounded domain $\Om$ with Lipschitz boundary $\partial\Om$
such that $\RR^d\setminus\overline{\Om}$ is connected. In accordance with (\ref{CaseAq}), the inclusion is characterized by some material properties $A\in\mL^{\infty}(\Om,\RR^{d\times d})$, $q\in\mL^{\infty}(\Om,\RR)$ such that
\begin{equation}\label{AssumCoef}
A=\mrm{Id},\ q=1\mbox{ in }\RR^d\setminus\overline{\Om}; \qquad\underset{\x\in\Om}{\mbox{ess inf }}\underset{\xi\in\mathbb{C}^d,|\xi|=1}{\mbox{inf }}(\xi\cdot A(\x)\overline{\xi})>0;\qquad\underset{\x\in \Om}{\mbox{ess inf }}q(\x)>0.
\end{equation}
\Blu{We also assume that $A\not\equiv\mrm{Id}$ in $\overline{\Omega}$.}
Below, we refer without distinction to $(A,q;\Omega)$ as the inclusion or the inhomogeneity. Consider some  
incident field $u_i$ satisfying 
\begin{equation}\label{HelmholtzHomo}
\Delta u_i+k^2u_i=0\quad\mbox{ in }\tilde{\Om},
\end{equation}
where $\tilde{\Om}$ is \Blu{an open} neighbourhood of $\Om$, more precisely a domain such that $\overline{\Om}\subset\tilde{\Om}$. Note that we do not assume $u_i$ be defined and smooth in the whole $\RR^d$. This allows us to take into account the possibility of illuminating the inclusion for example by incident fields due to point sources. The scattering of $u_i$ due to the inhomogeneity is governed by the following time-harmonic problem where $u_s$ denotes the scattered field, 
\begin{equation}\label{eq:MainGov1}
\begin{array}{|rcll}
\div(A\nabla u_s)+k^2 q u_s & = & -\div((A-\mrm{Id})\nabla u_i)-k^2 (q-1) u_i & \mbox{ in }\RR^d\\[5pt]
\multicolumn{4}{|c}{\dsp{\dsp\lim_{R\to +\infty}\int_{|\x|=R}\left|\frac{\partial u_s}{\partial r}-ik u_s \right|^2 ds(\x)  = 0.}}
\end{array}
\end{equation}
The derivatives in the first line of (\ref{eq:MainGov1}) are to be understood in the weak sense where $u_i$ is extended by zero outside $\tilde{\Om}$. 
\Blu{The last line of \eqref{eq:MainGov1} is the so-called Sommerfeld radiation condition. As a consequence to this radiation condition, $u_s$ admits the following asymptotic expansion
\begin{equation*}
	u_s(\x)=\frac{e^{ik|\x|}}{|\x|^{\frac{d-1}{2}}}\Pare{u_\infty(\hat{\x})+O(\frac{1}{|\x|})},\qquad\mbox{as $|\x|\to\infty$},
\end{equation*}
where $\hat{\x}:=\x/|\x|$. The term $u_\infty$ is referred to as the far-field pattern. By Rellich's lemma, $u_s$ in $\RR^d\setminus\overline{\Om}$ and $u_\infty$ have one-to-one correspondence. In other words, $u_\infty=0$ if and only if $u_s=0$ in $\RR^d\setminus\overline{\Om}$.}
The function $u:=u_i+u_s$ is usually referred to as the total field.
For all $k>0$, Problem (\ref{eq:MainGov1}) has a unique solution $u_s$ belonging to $ \mH^1(\mathcal{O})$ for all bounded domains $\mathcal{O}\subset\RR^d$ if $d=2$. When $d=3$, the existence of solutions can be established under the same assumptions. 
The uniqueness for $d=3$ has been proven with the additional condition that $A$ is Lipschitz continuous in $\overline{\Omega}$, by using the unique continuation principle, see e.g. \cite{CaCoHa23}.\\
We define in the following the so-called non-scattering frequencies.
\begin{definition}\label{DefNS}
We say that $k>0$ is a non-scattering frequency if there is \Blu{a domain} $\tilde{\Om}$ such that $\overline{\Om}\subset\tilde{\Om}$ and $u_i$ solving (\ref{HelmholtzHomo}) for which the corresponding scattered field satisfies $u_s = 0$ in $\RR^d\setminus\Omega$. 
Such $u_i$ is referred to as a non-scattering incident field. We denote by $\mathscr{S}_{\mrm{NS}}$ the set of non-scattering frequencies of (\ref{eq:MainGov1}).
\end{definition}
\noindent Notice if $u_i$ is a non-scattering incident field, then the inhomogeneity $(A,q;\Omega)$ is invisible from the exterior under the probing wave $u_i$. In addition,
by setting $v=u_i$, we see that $(u,v)\in \mH^1(\Omega)\times \mH^1(\Omega)$ solves the interior transmission eigenvalue problem
\begin{equation}\label{eq:ITEP}
\begin{array}{|rclll}
\div(A \nabla u) + k^2 q u&=&0&\mbox{ in }\Om \\[3pt]
\Delta v+k^2 v&=&0&\mbox{ in }\Om\\[3pt]
u-v&=&0&\mbox{ on }\partial\Om\\[3pt]
\nu\cdot A\nabla u-\nu\cdot\nabla v&=&0&\mbox{ on }\partial\Om.
\end{array}
\end{equation}
Here $\nu$ stands for the unit outward normal vector to $\partial\Om$.
\begin{definition}\label{df:TE}
We say that $k>0$ is a transmission eigenvalue if there is a non-zero \Blu{solution $(u,v)\in \mH^1(\Omega)\times \mH^1(\Omega)$ to \eqref{eq:ITEP}.} We denote by $\mathscr{S}_{\mrm{TE}}$ the set of transmission eigenvalues.
\end{definition}
\Blu{\noindent We note that if $A-\mrm{Id}$ is either positive definite or negative definite in $\overline\Omega$ near $\partial\Omega$, the natural regularity of transmission eigenfunction $(u,v)$ is $\mH^1(\Omega)\times \mH^1(\Omega)$. However when $A-\mrm{Id}$ changes sign on $\partial\Omega$, Fredholmness can be lost in $\mH^1(\Omega)\times \mH^1(\Omega)$ due to the appearance of strongly oscillating singularities. In this case, the wellposedness of \eqref{eq:ITEP} can be restored in weighted Sobolev spaces; see \cite{dhia2013strongly}.  On the other hand if $A= \mrm{Id}$ near $\partial\Omega$, a transmission eigenfunction $(u,v)$ is usually defined in $\mrm{L}^2(\Omega)\times \mrm{L}^2(\Omega)$ with $u-v\in \mH^2_0(\Omega)$. For more details on  transmission eigenvalue problems, we refer the readers to \cite{CaCoHa23}. 
}

\noindent The \Blu{discussion before Definition~\ref{df:TE}} shows that $\mathscr{S}_{\mrm{NS}}\subset\mathscr{S}_{\mrm{TE}}$. However the converse does not hold in general. More precisely, we have the following characterization result whose proof is straightforward.
\begin{proposition}\label{EquivDefNS}
A real frequency $k\in\mathscr{S}_{\mrm{TE}}$ is also in  $\mathscr{S}_{\mrm{NS}}$ if and only if there is a non-zero eigenpair \Blu{$(u,v)\in \mH^1(\Omega)\times \mH^1(\Omega)$ of \eqref{eq:ITEP} associated with $k$ such that $v$ can be extended in $\tilde{\Om}$, an open neighbourhood of $\Om$, as a function $\tilde{v} \in \mH^1(\tilde\Omega)$} satisfying
\begin{equation}\label{DefHlmEqn}
\Delta \tilde{v}+k^2 \tilde{v}=0\mbox{ in }\tilde{\Om}.
\end{equation}
\end{proposition}
\noindent Concerning (\ref{eq:ITEP}), it has first been shown that $\mathscr{S}_{\mrm{TE}}$ is discrete for rather mild conditions on the coefficients $A$ and \Blu{$q$} \cite{CaCoHa23}. 
Under a bit more restrictive hypothesis on $A$ and $q$, people have been able to prove that $\mathscr{S}_{\mrm{TE}}$ contains a countably infinite sequence of eigenvalues accumulating only at $+\infty$ (see \cite{PaSy08,Kirs09,CaHa09,CaGH10,Robb13}). In contrast, the set $\mathscr{S}_{\mrm{NS}}$ has been less studied. Certain simple examples of geometries (rectangle, balls) where $\mathscr{S}_{\mrm{NS}}\ne\emptyset$ have been given \Blu{in \cite{LaVaSu,CakoniVX23,KSS24}.} On the other hand, it has been established that for certain scatterers with non-smooth geometries (having corners in 2D or conical tips/edges in 3D), there holds $\mathscr{S}_{\mrm{NS}}=\emptyset$ (see \cite{BlPS14,PaSV14,ElHu18}).\\[5pt]
We would like to note that the non-scattering phenomenon is also related to the free boundary problem. More specifically, we notice that $u_s$ satisfies 
\begin{equation}\label{eq:freebdry}
	\begin{array}{|rclll}
		\div(A \nabla u_s) + k^2 q u_s&=&-f\chi_\Omega&\mbox{ in }\tilde{\Om}\\[3pt]
		\nu\cdot A\nabla u_s&=&-g&\mbox{ on }\partial\Om \\[3pt]
		u_s&=&0&\mbox{ in }\tilde{\Omega}\setminus\overline{\Om},
	\end{array}
\end{equation}
where $f=\div((A-\mrm{Id})\nabla u_i)+k^2 (q-1) u_i$, and $g=\nu\cdot(A-\mrm{Id})\nabla u_i$. If either $f\neq 0$ and $A=\mrm{Id}$ (and hence $g=0$) on $\partial\Omega$, or $g\neq 0$ and $A\neq \mrm{Id}$ on $\partial\Omega$, \Blu{in order for \eqref{eq:freebdry} to have a solution $u^s\in \mH^1(\tilde{\Om})$, one can show that $\partial\Omega$ must satisfy certain regularity conditions, provided $A$ and $q$ are sufficiently regular in $\overline{\Omega}$. The proofs are proceeded} 
by adopting techniques for free boundary problems. For more details, we refer the readers to \Blu{\cite{CaVo23,SaSh21,CakoniVX23,KSS24} and the references therein}.
\\
\newline
As indicated above, the main goal of the present article is to provide examples of configurations where $\mathscr{S}_{\mrm{NS}}\ne\emptyset$. We first consider the case where the material properties have the special form  $A=a\,\mrm{Id}$, $q=a$, with $a>0$ (Section \ref{section1}). This configuration is interesting for several reasons. For instance, one can completely characterize in this case the set $\mathscr{S}_{\mrm{TE}}$ as the union of Dirichlet and Neumann cavity eigenvalues. One can therefore exhibit non-scattering frequencies by  considering either Dirichlet or Neumann eigenvalues. We show in particular that some polygonal domains (convex or not) possess non-scattering frequencies. We also exhibit examples, some inspired by and some extracted from \cite{EckmannP95}, where non-scattering incident fields are not entire solutions to the Helmholtz equation (i.e. $\tilde \Omega \neq \RR^d$). We consider separately Dirichlet and Neumann eigenvalues. While examples for the first case are already present in the literature, the examples for Neumann eigenvalues are harder to find analytically and seem less known. We propose a method to construct such domains based on analyzing the characteristics for the gradient of the solutions to the Helmholtz equation. This allows for instance to prove existence of non-scattering frequencies for certain domains with corners of arbitrary angles and even cusps. We emphasize that on the contrary, non-scattering domains for the Dirichlet eigenvalues can only have corners of particular apertures. 
In Section \ref{section3}, we work with anisotropic materials and exhibit configurations where non-scattering occurs at any frequencies. The first examples are inspired by the concept of invisibility by diffeomorphisms \cite{KohnV84}. For associated anisotropies we prove in particular that the matrix $A$ has a special structure at irregular points of the boundary as it must satisfy $A \nu = \nu$ (see also \cite{CakoniVX23}). We also provide other examples of anisotropies (not associated with diffeomorphism transformations) for which non-scattering occur at any frequency in the case of regular and non-regular domains.\\
\newline
Note: we say that a function is an entire solution of the Helmholtz equation if it solves the homogeneous Helmholtz equation in $\RR^d$.

\section{Non-scattering in the case $A=a\,\mrm{Id}$, $q=a$, with $a>0$}
\label{section1}

In this section, we make the assumption that there holds
\begin{equation}\label{eq:Aqconst}
		A=a\,\mrm{Id}\AnD q=a\qquad\mbox{ in }\Om,
\end{equation}
where $a>0$ is a constant such that $a\ne1$. In this particular situation, we first show that $\mathscr{S}_{\mrm{TE}}$ is formed by the union of the eigenvalues of the Dirichlet and Neumann Laplacians in $\Om$. More precisely, introduce the problems
\begin{equation}\label{eq:eiDN}
\begin{array}{|rl}
			\Delta w_D+k^2w_D=0&\mbox{ in }\Omega\\[3pt]
			w_D=0&\mbox{ on }\partial\Omega
\end{array}\qquad\mbox{ and }\qquad
\begin{array}{|rl}
			\Delta w_N+k^2w_N=0&\mbox{ in }\Omega\\[3pt]
			\partial_{\nu}w_N=0&\mbox{ on }\partial\Omega.
\end{array}
\end{equation}
We denote by $\mathscr{S}_\mrm{D}$ (resp. $\mathscr{S}_\mrm{N}$) the \Blu{set of positive numbers $k$} such that the Dirichlet (resp. Neumann) problem \eqref{eq:eiDN} admits non-zero solutions in $\mH^1(\Om)$. We have the following result.
	\begin{lemma}\label{lem:DN}
When $A$, $q$ satisfy (\ref{eq:Aqconst}), we have $\mathscr{S}_{\mrm{TE}}=\mathscr{S}_\mrm{D}\cup\mathscr{S}_\mrm{N}$.
	\end{lemma}	
\begin{proof}
Suppose that $(u,v)$ is a transmission eigenpair solving \eqref{eq:ITEP}. Define $w_D:=u-v$ and $w_N:=au-v$. Then one readily checks that $w_D$, $w_N$ solve respectively the Dirichlet and Neumann eigenvalue problems \eqref{eq:eiDN}. Moreover, at least one of the functions $w_D$ and $w_N$ must be non-trivial. Hence, we have  $\mathscr{S}_{\mrm{TE}}\subset(\mathscr{S}_\mrm{D}\cup\mathscr{S}_\mrm{N})$.\\
Conversely, assume that $w_D$ is a non-zero solution of the Dirichlet problem appearing in \eqref{eq:eiDN}. Then $(u,v)=(aw_D,w_D)$ is a (non-trivial) eigenpair of \eqref{eq:ITEP}. Similarly, if $w_N$ is a non-zero solution of the Neumann problem appearing in \eqref{eq:eiDN}, then  $(u,v)=(w_N,w_N)$ solves \eqref{eq:ITEP}. This guarantees that $(\mathscr{S}_\mrm{D}\cup\mathscr{S}_\mrm{N})\subset\mathscr{S}_{\mrm{TE}}$.
\end{proof}
\begin{remark}	
Note that the case where $A$, $q$ satisfy (\ref{eq:Aqconst}) is a situation where all the $k^2\in\CC$ such that (\ref{eq:ITEP}) admits a non-zero solution \Blu{must be real numbers}. It is not known if this can happen with other pairs $A$, $q$.
\end{remark}

\noindent From Lemma \ref{lem:DN}, when $A$, $q$ satisfy (\ref{eq:Aqconst}), the question of finding non-scattering eigenvalues can be reformulated as ``are there eigenfunctions of the Dirichlet or Neumann Laplacian in $\Omega$ that can be extended as solutions to the Helmholtz equation in a neighbourhood of $\overline{\Omega}$?''.
Some positive answers can be given in three situations:\\[3pt]
- For certain geometries $\Om$, one can compute analytically the eigenfunctions of the Dirichlet/Neumann Laplacians and observe that they are defined in larger domains than $\Om$;\\
- For certain $\Omega$, one can use reflections to extend the eigenfunctions of the Dirichlet/Neumann Laplacians in $\Om$ to larger domains;\\
- Given a function $u_i$ solving the Helmholtz equation (\ref{HelmholtzHomo}) in some given domain $\tilde{\Om}$, one can look for $\Omega$, with $\overline{\Om}\subset\tilde{\Om}$, for which $u_i$ is an eigenfunction of the Dirichlet or Neumann Laplacian in $\Om$.\\[3pt]
We present corresponding results in the following subsections.

\subsection{Domains where analytic expressions can be obtained for the eigenfunctions}

Assume here that $\Om$ coincides with the unit square and consider $k\in\mathscr{S}_\mrm{D}$ as well as $w_D$ a corresponding eigenfunction of the Dirichlet problem appearing in (\ref{eq:eiDN}). In an appropriate set of coordinates, $w_D$ writes as a linear combination of the functions
\[
\sin(m\pi x)\sin(n\pi y),\qquad m,n\in\NN^\ast:=\{1,2,3,\dots\}.
\]
In the proof of Lemma \ref{lem:DN}, we have seen that $(u,v)=(aw_D,w_D)$ constitutes an eigenpair of \eqref{eq:ITEP}. Clearly $v$ extends as a function solving the homogeneous Helmholtz equation (\ref{DefHlmEqn}) in $\RR^2$. From Proposition \ref{EquivDefNS}, we infer that $k\in\mathscr{S}_{\mrm{NS}}$. Interestingly, as explained in \cite{BlPS14,CakoniVX23}, in that case one can exhibit non-scattering incident fields $u_i$ which are simple combinations of plane waves. More precisely, $u_i$ such that
\[
\begin{array}{rcl}
u_i(x,y)&=&4\sin(m\pi x)\sin(n\pi y)\\[3pt]
 &=& e^{i\pi(mx-ny)}+e^{-i\pi(mx-ny)}-e^{i\pi(mx+ny)}-e^{-i\pi(mx+ny)}
\end{array}
\]
produces a scattered field which is exactly zero outside of $\Om$. Similarly we establish that $\mathscr{S}_\mrm{N}\subset\mathscr{S}_{\mrm{NS}}$.\\
\newline
This reasoning can be adapted to deal with other simple domains $\Om$ where we can use separation of variables. This allows one to state the following result.

\begin{proposition}\label{Propo1}
Assume that $\Omega\subset\RR^2$ is a rectangle, a disk, or an elliptical domain (whose boundary is an ellipse). Then when $A$, $q$ satisfy (\ref{eq:Aqconst}), we have $\mathscr{S}_{\mrm{NS}}=\mathscr{S}_{\mrm{TE}}=\mathscr{S}_\mrm{D}\cup\mathscr{S}_\mrm{N}$.
\end{proposition}
\noindent Note that to consider the case of elliptic domains, we work with the elliptic coordinates $(\mu,\theta)$ such that, after a rigid change of variables,
\[
x=\alpha \cosh\mu \cos\theta, \qquad y=\alpha \sinh\mu \sin\theta,
\]
where $\mu\geq 0$, $\theta\in \RR/(2\pi)$ and $(\pm\alpha,0)$ stand for the foci of the ellipse. In particular, the curves $\mu=\mrm{constant}>0$ are confocal ellipses. Solutions of (\ref{eq:eiDN}) can be decomposed on functions with separate variables in $\mu$, $\theta$. For the latter functions, we find that the dependences in $\mu$, $\theta$ satisfy respectively a modified Mathieu's equation and a Mathieu's equation with periodic boundary conditions, which are both Sturm-Liouville problems. Moreover it is known that the solutions of the Mathieu's equation which are regular at zero can be extended in $\RR$ so that all eigenfunctions of the Dirichlet and Neumann Laplacians in a bounded elliptical domain can be extended as solutions of an Helmholtz equation in $\RR^2$. For more details, we refer the readers to \cite{Day55,McL64}.\\[3pt]
Similar to Proposition~\ref{Propo1}, we have the following statement in $\RR^3$.
\begin{proposition}\label{Propo3}
	Assume that $\Omega\subset\RR^3$ is a ball, an ellipsoid, or, in an appropriate system of coordinates, there holds $\Omega=I\times\om$ where $I$ is a bounded open interval and $\om\subset\RR^2$ is a rectangle, a disk, or an elliptical domain. Then when $A$, $q$ satisfy (\ref{eq:Aqconst}), we have $\mathscr{S}_{\mrm{NS}}=\mathscr{S}_{\mrm{TE}}=\mathscr{S}_\mrm{D}\cup\mathscr{S}_\mrm{N}$.
\end{proposition}
\begin{figure}[ht!]
\centering
\raisebox{0.4cm}{\includegraphics[height=3.2cm]{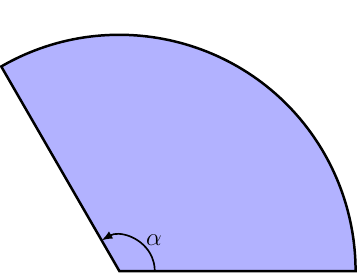}}
\includegraphics[height=4cm]{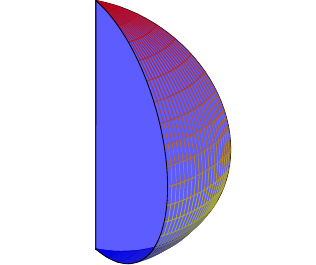}\includegraphics[height=4.2cm]{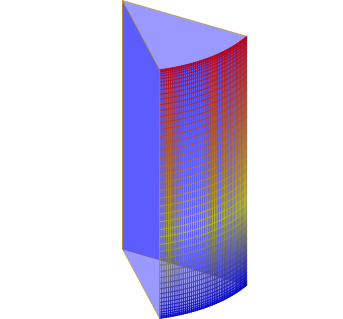}
\caption{Sector (left), spherical wedge (center) and cylindrical wedge (right).\label{Geomcylindrical}} 
\end{figure}

\noindent We now turn our attention to other geometries. Consider some constants $\alpha\in(0,2\pi)$ and $\ell>0$. Define in $\RR^2$ the sector 
\[
\Omega=\{(r\cos\theta,r\sin\theta)\,|\,r\in(0,\ell),\,\theta\in(0,\alpha)\}
\] 
and in $\RR^3$, the spherical wedge
\[
\Omega=\{(r\cos\phi\sin\theta,r\sin\phi\sin\theta,r\cos\theta)\,|\,r\in(0,\ell),\theta\in(0, \pi),\phi\in(0,\alpha)\}
\]
as well as the cylindrical wedge 
\[
\Omega=\{(r\cos\theta,r\sin\theta,z)\,|\,r\in(0,\ell),\,\theta\in(0,\alpha),z\in I\}
\]
(see Figure \ref{Geomcylindrical} for illustrations). Here $I$ is an open bounded interval. For theses domains, depending on the opening angle, there may exist zero or infinitely many non-scattering frequencies.
\begin{proposition}\label{Propo2}
With the notation above, assume that $\Omega$ is either a sector in $\RR^2$, a spherical wedge in $\RR^3$ or a cylindrical wedge in $\RR^3$, of angle $\alpha$. When $A$, $q$ satisfy (\ref{eq:Aqconst}), we have:\\[3pt]
- if $\alpha\in(0,2\pi)\cap\mathbb{Q}\pi$, then $\mathscr{S}_{\mrm{NS}}$ contains an unbounded sequence which accumulates only at $+\infty$;\\[2pt]
- if $\alpha\in(0,2\pi)\cap(\mathbb{R}\pi\setminus\mathbb{Q}\pi)$, then $\mathscr{S}_{\mrm{NS}}=\emptyset$.
\end{proposition}
\begin{remark}
In fact, we show in the proof that $\mathscr{S}_{\mrm{NS}}$ coincides exactly with $\mathscr{S}_{\mrm{TE}}=\mathscr{S}_\mrm{D}\cup\mathscr{S}_\mrm{N}$ if and only if $\alpha\in(0,2\pi)\cap(\pi/\mathbb{N}^\ast)$.
\end{remark}
\begin{remark}
In the second item of Proposition \ref{Propo2}, the fact that $\mathscr{S}_{\mrm{NS}}=\emptyset$ in that domains $\Om$ is not only due to the value of the opening angle $\alpha$ but also to the shape of the whole $\Om$. Indeed in \S\ref{sec:Neumann}, we will exhibit examples of geometries in the plane supporting non-scattering frequencies with corners of arbitrary $\alpha\in(0,2\pi)$.
\end{remark}
\begin{proof}
	Let $\Omega$ be a sector in $\RR^2$. The eigenfunctions of the Dirichlet Laplacian in $\Omega$ coincide with the functions $J_{m\pi/\alpha}(\mu_mr)\sin(m\pi\theta/\alpha)$, with $m\in\NN^\ast$ and $\mu_m>0$ such that $J_{m\pi/\alpha}(\mu_m\ell)=0$. Here $J_{m\pi/\alpha}$ is the Bessel function of the first kind with order $m\pi/\alpha$\Blu{, and  $(r,\theta)$ are the polar coordinates of $\x\in\RR^2$. 
	Notice that the function $v$ given by $v(\x)=J_{m\pi/\alpha}(\mu_mr)\sin(m\pi\theta/\alpha)$ is real-analytic in an open neighbourhood of the origin if and only if $m\pi/\alpha\in\mathbb{Z}$. Moreover, if $m\pi/\alpha=l\in\mathbb{Z}$ then $(x,y)\mapsto J_{l}(\mu_mr)\sin(l\theta)$ is in fact an entire solution of the Helmholtz equation. Hence we infer that an eigenfunction can be extended as a solution of the Helmholtz equation in an open neighbourhood of $\Omega$ if and only if $\alpha\in\mathbb{Q}\pi$.} 
The analysis is similar for eigenfunctions of the Neumann Laplacian.\\
The cases of spherical and cylindrical wedges in $\RR^3$ can be dealt similarly by working, respectively, with spherical and  cylindrical coordinates.
\end{proof}
\noindent With this approach, we could also consider 3D conical tip, i.e. domains of the form 
\[
\Omega=\{(r\cos\theta\cos\phi,r\cos\theta\sin\phi,r\sin\theta)\,|\,0<r<\ell,0<\theta< \theta_0,0\le\phi<2\pi\}
\]
with $\theta_0>0$. In this case, the possibility of extending eigenfunctions of the Dirichlet/Neumann Laplacians in $\Om$ depends on whether $\cos\theta_0$ is a zero of the Legendre polynomial $P^m_n$ for some integers $n$ and $m\in[-n,n]$. We choose not to elaborate much on this direction.\\
\newline
Let us mention that some of the discussions above for 2D cases can also be found in \cite{KS6033}.

\subsection{Domains where eigenfunctions can be extended by reflections}

Now we present other geometries where some or all the Dirichlet/Neumann eigenfunctions of (\ref{eq:eiDN}) can be extended, this time by working with reflections. We start with a definition.
\begin{definition}\label{Defproper paving unit}
Let $\Omega$ be a polygon of $\RR^2$ (resp. a polyhedron of $\RR^3$). We say that $\Omega$ is a \emph{proper paving unit} if the following two conditions hold:\\[5pt]
- One can find a combination of successive reflections of $\overline{\Omega}$ with respect to the edges (resp.  faces) that allows one to cover \Blu{an open set $\tilde{\Om}$ such that $\overline{\Om}\subset\tilde\Om$,}  
with the overlaps only on the skeleton $\mathcal{S}$ consisting of the edges (resp. faces) of $\partial\Om$ and their images by the reflections. \\[5pt]
- If one assigns a different colour to each of the edges (resp. faces) of $\partial\Om$, then each element of $\mathcal{S}$ may inherit two or more colours due to reflections. We require that every element of $\mathcal{S}$ has only one colour under reflections. 
\end{definition}
\noindent In Figure \ref{fig:reflec}, we present an example of domain $\Om$ and covering which does not satisfy the second item of Definition \ref{Defproper paving unit}. \Blu{We also note that Definition~\Ref{Defproper paving unit} concerning polygons and polyhedra 
	 can be easily generalized to polytopes in $\RR^d$ for any $d\ge 2$. In this case, the corresponding reflections are operated with respect to the facets.
 } 

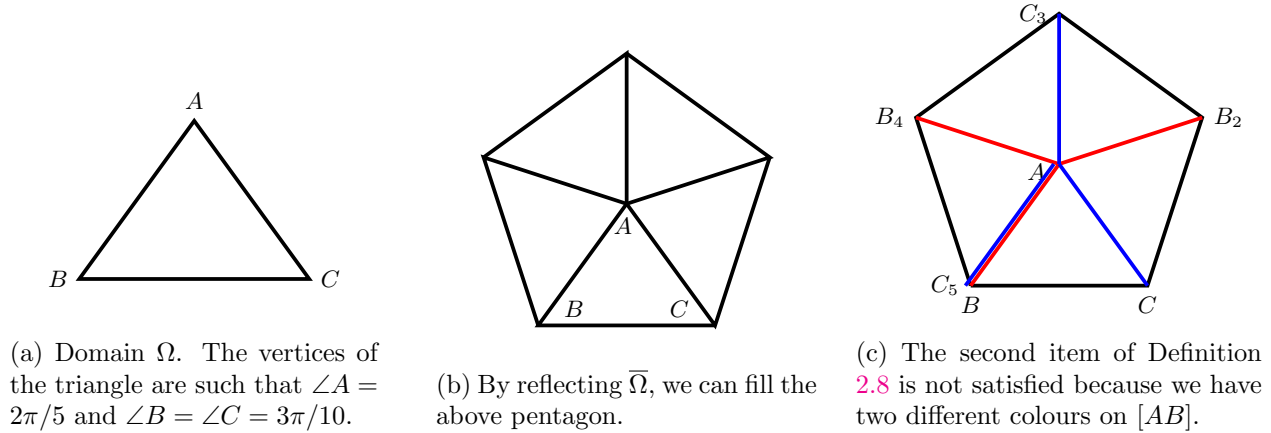
\begin{figure}[!ht]
\def\len{2}
\def\ngn{5}
\def\ang{360/\ngn}
\def\angi{-90-\ang/2}
\begin{subfigure}{.29\textwidth}
\centering
\vspace{1cm}
\begin{tikzpicture}[scale=1.3,font=\footnotesize]
\path (\angi:\len) coordinate (B) (\angi+\ang:\len) coordinate (C) (0:0) coordinate (A);
\draw[line width=0.5mm] (A)
			-- (B) node [at start, above]{$A$}
			-- (C) node [at start, left]{$B$} 
			-- (A) node [at start, right]{$C$}-- cycle;
\end{tikzpicture}\vspace{0.35cm}
\caption{Domain $\Omega$. The vertices of the triangle are such that $\angle A=2\pi/5$ and $\angle B=\angle C=3\pi/10$.}
	\end{subfigure}\qquad
	\begin{subfigure}{.3\textwidth}
	\centering

	\begin{tikzpicture}[font=\footnotesize]
		\foreach \n in {1,...,\ngn}
		{
			\path (\angi+\n*\ang-\ang:\len) coordinate (B) (\angi+\n*\ang:\len) coordinate (C) (0:0) coordinate (A);
		\draw[line width=0.5mm] (A) -- (B) -- (C);}
	\node at (-.05,-.3) {$A$};
	\node at (-.7,-1.4) {$B$};
	\node at (.7,-1.4) {$C$};
	\end{tikzpicture}\vspace{0.35cm}
	\caption{By reflecting $\overline{\Omega}$, we can fill the above pentagon.}
\end{subfigure}
\quad
\begin{subfigure}{.32\textwidth}
	\centering
	\begin{tikzpicture}[font=\footnotesize]
		\foreach \n in {1,...,\ngn}{
			\path (\angi+\n*\ang-\ang:\len) coordinate (V-\n) (0:0) coordinate (A);}
		\node at (-.3,-.1) {$A$};
		\draw [black,line width=0.5mm] (V-1)
		-- (V-2) node [at start, below, black]{$B$}
		-- (V-3)  node [at start, below, black]{$C$}
		-- (V-4)  node [at start, right, black]{$B_2$}
		-- (V-5)  node [at start, left, black]{$C_3$}
		-- (V-1)  node [at start, left, black]{$B_4$} node [at end, left, black]{$C_5$};
		\foreach \n in {1,3,5}{
			\draw [red,line width=0.5mm] (A)
			-- (V-\n);}
		\foreach \n in {2,4}{
			\draw [blue,line width=0.5mm] (A)
			-- (V-\n);}
\draw[blue,line width=0.5mm,transform canvas={xshift=-2pt}] (A) -- (V-1);
	\end{tikzpicture}
	\caption{The second item of Definition \ref{Defproper paving unit} is not satisfied because we have two different colours on $[AB]$.}
\end{subfigure}
\caption{A domain $\Omega$ that is not a proper paving unit.}
\label{fig:reflec}
\end{figure}

\begin{proposition}
Assume that $\Omega\subset\RR^d$, \Blu{$d\ge 2$}, is a proper paving unit. Then when $A$, $q$ satisfy (\ref{eq:Aqconst}), we have $\mathscr{S}_{\mrm{NS}}=\mathscr{S}_{\mrm{TE}}=\mathscr{S}_\mrm{D}\cup\mathscr{S}_\mrm{N}$.
\end{proposition}
\begin{proof}
Let $\Omega$ be a proper paving unit. Consider one particular edge or face of $\partial\Om$ and introduce some system of coordinates $x',x_d$ with $x'\in\RR^{d-1}$, such that this edge or face lies in the hypersurface $\{x_d=0\}$.
Let $\Omega_1$ denote the reflection of $\Omega$ with respect to the line or plane $\{x_d=0\}$. Consider $v$ an eigenfunction of the Dirichlet Laplacian in $\Om$. Classically we can extend it to $\Omega_1$ by defining $v(x',x_d)=-v(x',-x_d)$ for all $(x',x_d)\in\Omega_1$.
	Similarly, Neumann eigenfunctions can be extended in $\Omega_1$ by defining $v(x',x_d)=v(x',-x_d)$. It can be verified straightforwardly that such extended function $v$ satisfies $\Delta v+k^2v=0$ in the whole interior of $\overline{\Omega\cup\Omega_1}$. Repeating this reflection argument successively, we can extend Dirichlet/Neumann eigenfunctions to \Blu{an open neighbourhood of $\Omega$} as solutions to the homogeneous Helmholtz equation. The proof is completed.
\end{proof}
\noindent From classical results of interior regularity, this shows that in proper paving unit, Dirichlet and Neumann eigenfunctions are smooth up to the boundary because their extensions are real-analytic in \Blu{an open neighbourhood of $\Omega$}.\\
\newline
One can check that triangles that are equilateral (60\textdegree-60\textdegree-60\textdegree), hemiequilateral (30\textdegree-60\textdegree-90\textdegree) or isosceles right (45\textdegree-45\textdegree-90\textdegree) are all proper paving unit of $\RR^2$. In fact, for these three types of triangles the Dirichlet/Neumann eigenfunctions can be expressed with trigonometric functions as discovered by G. Lam\'{e} (see \cite{McCartin11}). Besides, observe that if $\om$ is a proper paving unit of \Blu{$\RR^d$, $d\ge 2$, then $\Omega=I\times \om$ is a proper paving unit of $\RR^{d+1}$} for any bounded interval $I$.\\
\newline
Finally, let us mention that for domains \Blu{$\om$ obtained} by ``properly'' reflecting a given proper paving unit $\Om$ a finite number of times (see an illustration in 2D with Figure \ref{fig:reflecok}), $\mathscr{S}_{\mrm{NS}}$ contains an unbounded sequence. Indeed, any eigenfunction of the Dirichlet/Neumann Laplacian in $\Om$ extended to \Blu{an open neighbourhood of $\omega$} according to the process described above is clearly an eigenfunction of the Dirichlet/Neumann Laplacian in $\om$.

\begin{figure}[!ht]
	\centering
	\begin{subfigure}{.45\textwidth}
		\centering
		\includegraphics[scale=.45]{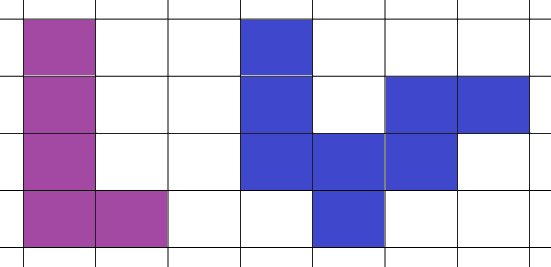}
	\end{subfigure}
	\begin{subfigure}{.5\textwidth}
		\centering
		\includegraphics[scale=.5]{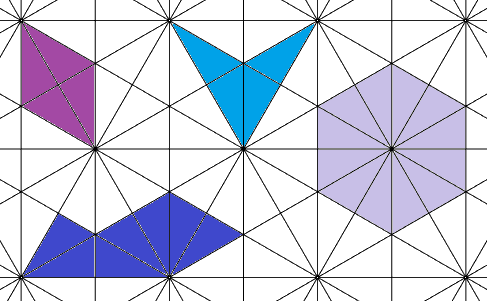}
	\end{subfigure}
	\caption{Examples of domains for which $\mathscr{S}_{\mrm{NS}}$ contains an unbounded sequence.\label{fig:reflecok}}
\end{figure}

\subsection{Combining solutions to the homogeneous Helmholtz equation}

In the previous subsections, we presented canonical examples of geometries $\Om$ where eigenfunctions of the Dirichlet/Neumann Laplacians can be extended as smooth solutions of the homogeneous Helmholtz equation in a neighbourhood of $\Om$. Here, by working directly with solutions to the homogeneous Helmholtz equation, we provide examples of more general domains such that non-scattering frequencies exist. To keep notations simple, we work in $\RR^2$ only. The ideas can be applied to generate examples in $\RR^3$.

\subsubsection{Extendable Dirichlet eigenfunctions}\label{DirichletEigenVec}

Consider the function $v_0$ such that
\[
v_0(x,y)=\sin x\sin y.
\]
It is a particular solution of the equation 
\begin{equation}\label{HHR2}
\Delta v+2v=0\quad\mbox{ in }\RR^2. 
\end{equation}
Now let us, roughly speaking, rotate this function by an angle $\alpha\in(0,2\pi)$ by defining $v_\alpha$ such that 
\[
v_\alpha(x,y)=v_0(\cos(\alpha) x-\sin(\alpha) y,\sin(\alpha) x+\cos(\alpha) y).
\]
Note that $v_\alpha$ also solves (\ref{HHR2}). Then for $\lambda\in\RR$, set
\[
v^\lambda_\alpha=v_0+\lambda\,v_\alpha
\]
and define the nodal set $\mathscr{N}^\lambda_\alpha:=\{(x,y)\in\RR^2\,|\,v^\lambda_\alpha(x,y)=0\}$. As classical in literature (see e.g. \cite{CoHi09}), we call nodal domains of $v^\lambda_\alpha$ the maximally connected subsets of $\RR^2$ for which $v^\lambda_\alpha$ does not change sign. From Proposition \ref{EquivDefNS} and Lemma \ref{lem:DN}, we obtain the following statement. 
\begin{proposition}
Any bounded nodal domain of $v^\lambda_\alpha$ is a domain where $2\in\mathscr{S}_{\mrm{NS}}$. 
\end{proposition}
\noindent Let us make a few comments concerning this approach. First the choice $k^2=2$ is arbitrary here and we could consider any other $k^2>0$. On the other hand, we combined only two particular solutions of (\ref{HHR2}). We could have worked similarly with any other linear combinations of functions satisfying (\ref{HHR2}). To exhibit different solutions of (\ref{HHR2}), we can proceed to rotations of $v_0$ as we did. We can also translate $v_0$ or its rotated versions. This a priori offers a large variety of functions. A natural question then is \textit{``how rich is the family of corresponding bounded nodal domains?''}. Finally, observe that we prove only the existence of one element in  $\mathscr{S}_{\mrm{NS}}$ and not that $\mathscr{S}_{\mrm{NS}}$ contains an unbounded sequence as in the statements of the previous subsections.\\
\newline
In Figure \ref{fig:NodalLine}, we display the nodal sets $\mathscr{N}^\lambda_\alpha$ for $\alpha=\pi/4$ and $\lambda\in\{-1,-0.2,0.5,2\}$. In each situation we observe that there exist bounded nodal domains. The question of proving the existence of bounded nodal domains would deserve to be studied in more details.\\
\newline
\begin{figure}[!ht]
	\centering
\includegraphics[width=.23\textwidth]{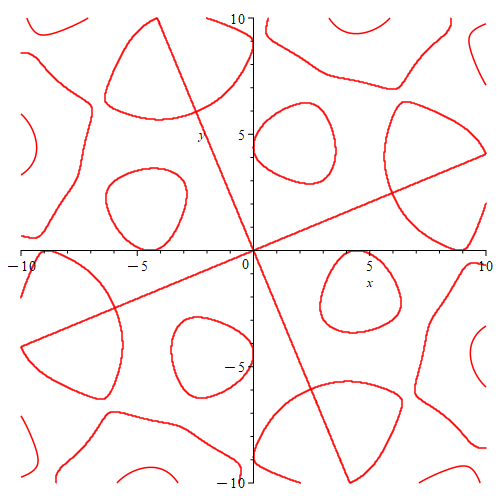}\quad \includegraphics[width=.23\textwidth]{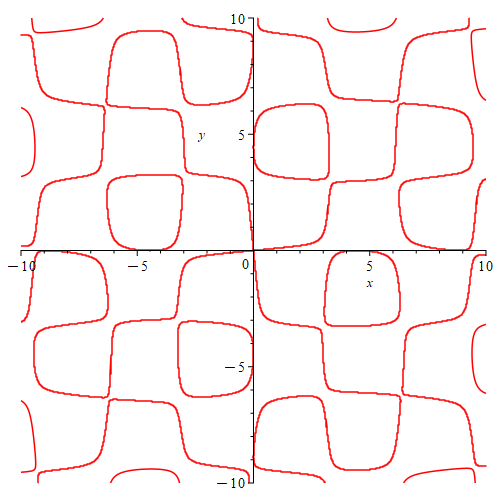}\quad \includegraphics[width=.23\textwidth]{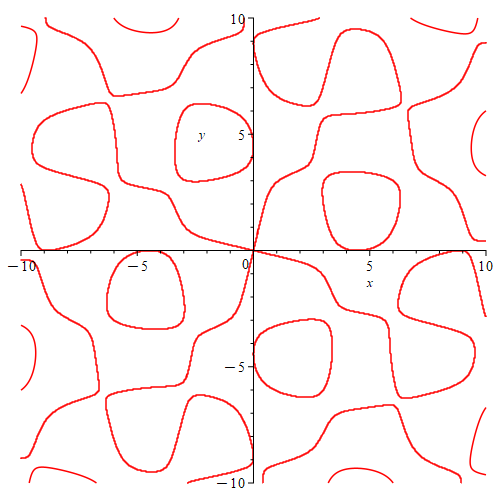}\quad \includegraphics[width=.23\textwidth]{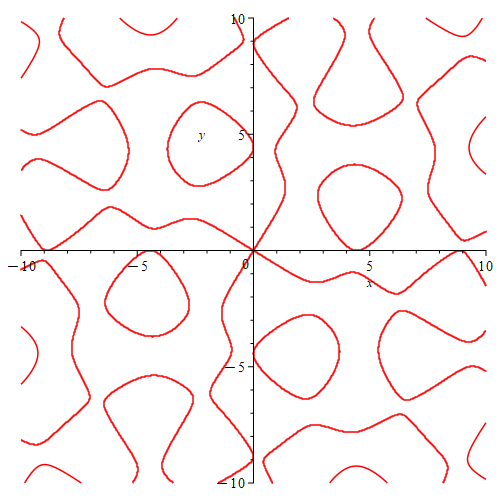}
\caption{Nodal sets $\mathscr{N}^\lambda_\alpha$ for $\alpha=\pi/4$ and, from left to right, $\lambda=-1,-0.2,0.5,2$.\label{fig:NodalLine}}
\end{figure}

\noindent Interestingly, this technique can be exploited to exhibit domains $\Om$ such that eigenfunctions of the Dirichlet Laplacian in $\Om$ can be extended only to an open neighbourhood of $\overline{\Omega}$ but not to the whole $\RR^2$. This has been found by J. Eckmann and C. Pillet in \cite{EckmannP95} and we reproduce their work here. \\
\newline
Given $L\in\mathbb{N}^\ast$ and $\mu\in\RR_+:=(0,+\infty)$, we define the function $v$ such that
	\begin{equation}\label{eq:fR}
		v(r, \theta) =\sum_{l=0}^{L-1}J_{\mu}(k\rho_l)\cos(\mu\psi_l).
	\end{equation}
Here, \Blu{$(r,\theta)$ stand for the polar coordinates} with $r>0$ and \Blu{$\theta\in(-\pi,\pi]$\footnote{\Blu{It is important here that we fix the domain of $\theta$ to be a specific $2\pi$-length interval. For example, under this setup, the singularity of the function $f(\x)=r^{1/2}\sin(\theta/2)$ is on the half axis $\{\x=(x,y)\in\RR^2;x\le 0\}$.}},} $J_{\mu}$ is again the Bessel function of the first kind with order $\mu$, \Blu{the ``translated and rotated polar coordinates'' $\rho_l>0$ and $\psi_l\in(-\pi,\pi]$ are defined via
	\[
	\rho_l\cos\psi_l=a+r\cos(\theta+2\pi l/L),\quad
	\rho_l\sin\psi_l=r\sin(\theta+2\pi l/L),\qquad l=0,\ldots,L-1,
	\]
	with some fixed $a\in\RR$.}
In particular, the term for $l=0$ in \eqref{eq:fR} corresponds to the spherical wave function $J_{\mu}(kr)\cos(\mu\theta)$ with the center translated from $(0,0)$ to $(-a,0)$. If we further rotate the initial spherical wave function $J_{\mu}(kr)\cos(\mu\theta)$ clockwise by $2\pi/L$ (resp.  $2\pi l/L$), we obtain the term in \eqref{eq:fR} for $l=1$ (resp.  $l$ in general).
\\More generally, we can consider functions $v$ of the form
\begin{equation}\label{eq:fRgen}
	v(r, \theta) =\sum_{l=0}^{L-1}b_lJ_{\mu_l}(k\rho_l)\cos\Pare{\mu_l(\psi_l-\phi_l)},
\end{equation}
with different Bessel orders $\mu_l$, rotational parameters $\phi_l$, and weights $b_l$, for each $l=0,\ldots,L-1$.
We can also impose different translational and rotational effects to each term by defining
\[\rho_l\cos\psi_l=a_l+r\cos(\theta+\theta_l)\AND
\rho_l\sin\psi_l=r\sin(\theta+\theta_l)\]
with different $a_l$ and $\theta_l$ for each $l=0,\ldots,L-1$.
\Blu{\\In the rest of this subsection, we choose the value of $k$ to be the first positive zero of $J_{\mu}$.}\\
\newline	
The following example, taken from \cite{EckmannP95}, gives a bounded and simple connected domain $\Omega$ with analytic boundary where the Dirichlet Laplacian has an eigenfunction which is extendable in an open neighbourhood of $\overline{\Omega}$ but not to the whole $\RR^2$.
	\begin{example}\label{ex:Eck3/2}
		Let us consider the function $v$ in \eqref{eq:fR} with $\mu=3/2$, $L=3$, $a=0.6$. 
		\Blu{Notice that the singularity of the function $v_0$ defined by $v_0(\x)=J_{3/2}(kr)\cos(3\theta/2)$ is located at $\theta=\pi$, since the derivative of $v_0$ with respect to $\theta$ (or $y$) has different limiting value when $\theta\to\pm \pi$. Hence the function $v$ defined in \eqref{eq:fRgen} is singular in $\RR^2$ but real-analytic in $\RR^2\setminus\Sigma$ with $\Sigma:=\{(r\cos\theta,r\sin\theta)\,|\,r\in[a,\infty),\, \theta =-\pi+2l\pi /L,\, l=0,\ldots,2\}$. In particular, $v$ is real-analytic in the open ball $B_{a}$ centred at the origin and of radius $a$. In the following, we shall show that there is a bounded and simple-connected domain $\Omega\subset B_a$ with real-analytic boundary such that $v$ is real-analytic in (an open neighbourhood of) $\Omega$ and $v=0$ on $\partial \Omega$.\\[5pt]} 
First notice that $(x,y)\mapsto J_{3/2}(kr)\cos(3\theta/2)$ is positive in $D_1=\{(r\cos\theta,r\sin\theta)\,|\,r\in(0,1), \theta\in(-\pi/3,\pi/3)\}$ and negative in $B_1\setminus\overline{D_1}$ (see Figure \ref{fig:a}). Using this property, we display in Figure \ref{fig:b} a yellow region where the three terms in the sum defining $v$ are all positive and a blue region where these three terms are all negative. We note by $\tilde{\Om}$ the green region in which the sign of $v$ is a priori not clear. Looking at every direction $\theta\in(-\pi,\pi]$, with the intermediate value theorem, we can show that there is $(x_\theta,y_\theta)$ in $\tilde{\Om}$ such that $v(x_\theta,y_\theta)=0$. Moreover, by exploiting that $v$ is real-analytic in $\tilde{\Om}$, we infer that there is a closed analytic curve $\Gamma\subset\tilde{\Om}$ such that $v=0$ on $\Gamma$ (see Figure \ref{fig:EPzero} for a plot of $\Gamma$). Denote by $\Om$ the bounded open set surrounded by $\Gamma$.\\[5pt]
Then $v$ is an eigenfunction of the Dirichlet Laplacian in $\Om$ which can be analytically extended to $\tilde{\Om}$ (actually to $\RR^2\setminus\Sigma$) as a solution to the Helmholtz equation. Because of this property, we conclude that we have $\mathscr{S}_{\mrm{NS}}\neq\emptyset$ for this $\Om$. Let us emphasize however that $v$ can not be extended in the whole $\RR^2$ analytically because its derivative with respect to $\theta$ is not continuous on $\Sigma$.
			\begin{figure}[!ht]
				\centering
				\begin{subfigure}{.29\textwidth}
					\centering
					\includegraphics[width=\textwidth]{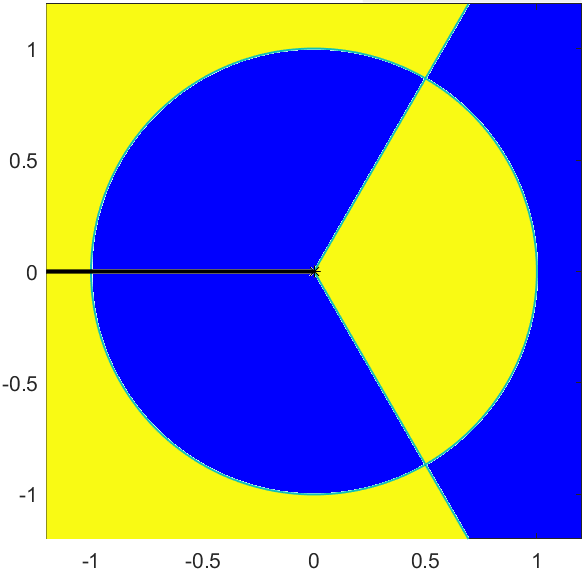}
					\caption{Properties of $(x,y)\mapsto J_{3/2}(kr)\cos(3\theta/2)$: positive in the yellow regions and negative in the blue regions.}\label{fig:a}
				\end{subfigure}
			\hfil
				\begin{subfigure}{.29\textwidth}
					\centering
						\includegraphics[width=\textwidth]{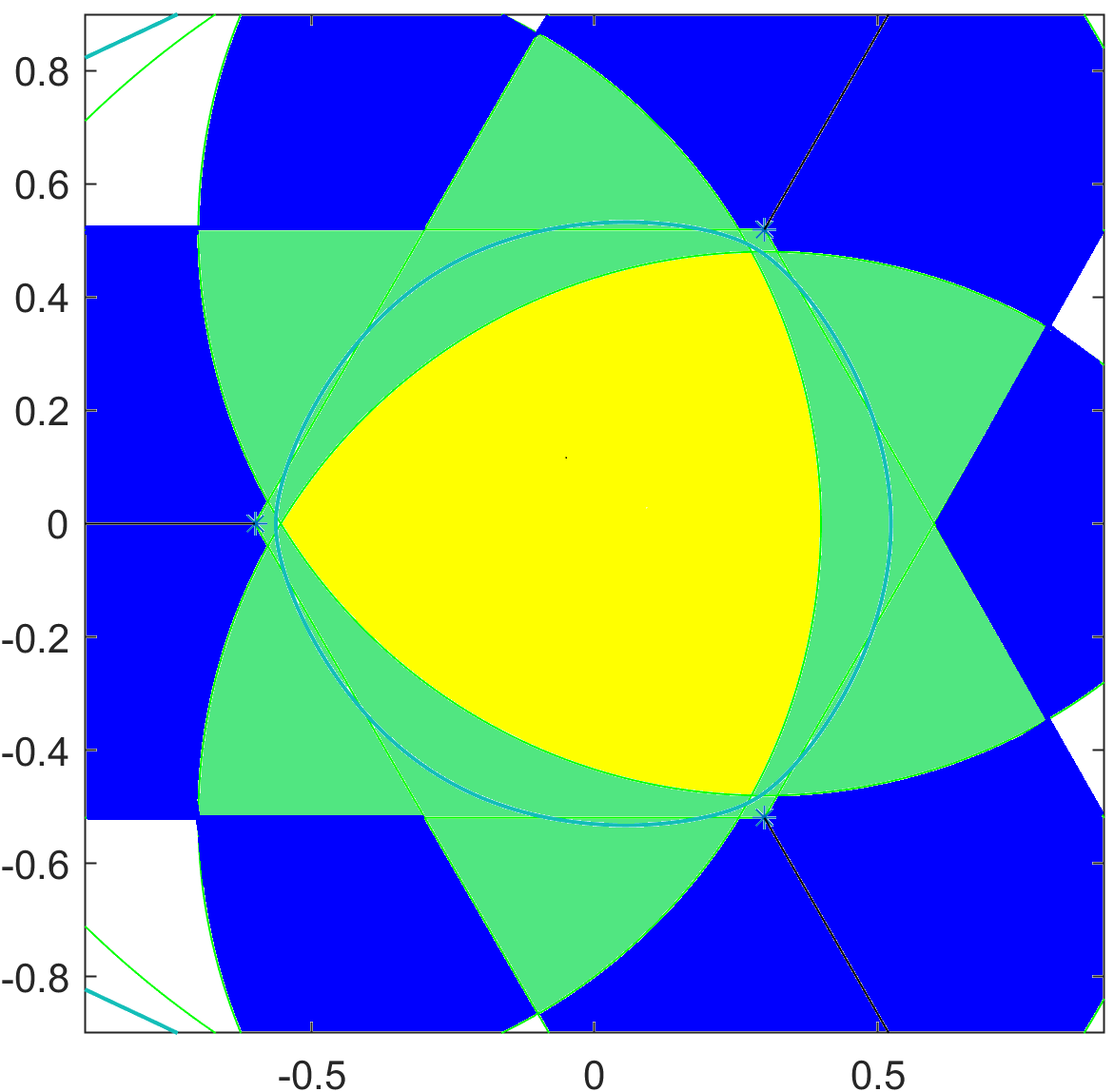}
					\caption{Properties of $v$: in the yellow (resp. blue) region, the three terms in the sum defining $v$ are positive (resp. negative).}
					\label{fig:b}
				\end{subfigure}
			\hfil
			\begin{subfigure}{.33\textwidth}
				\centering
				\includegraphics[width=\textwidth]{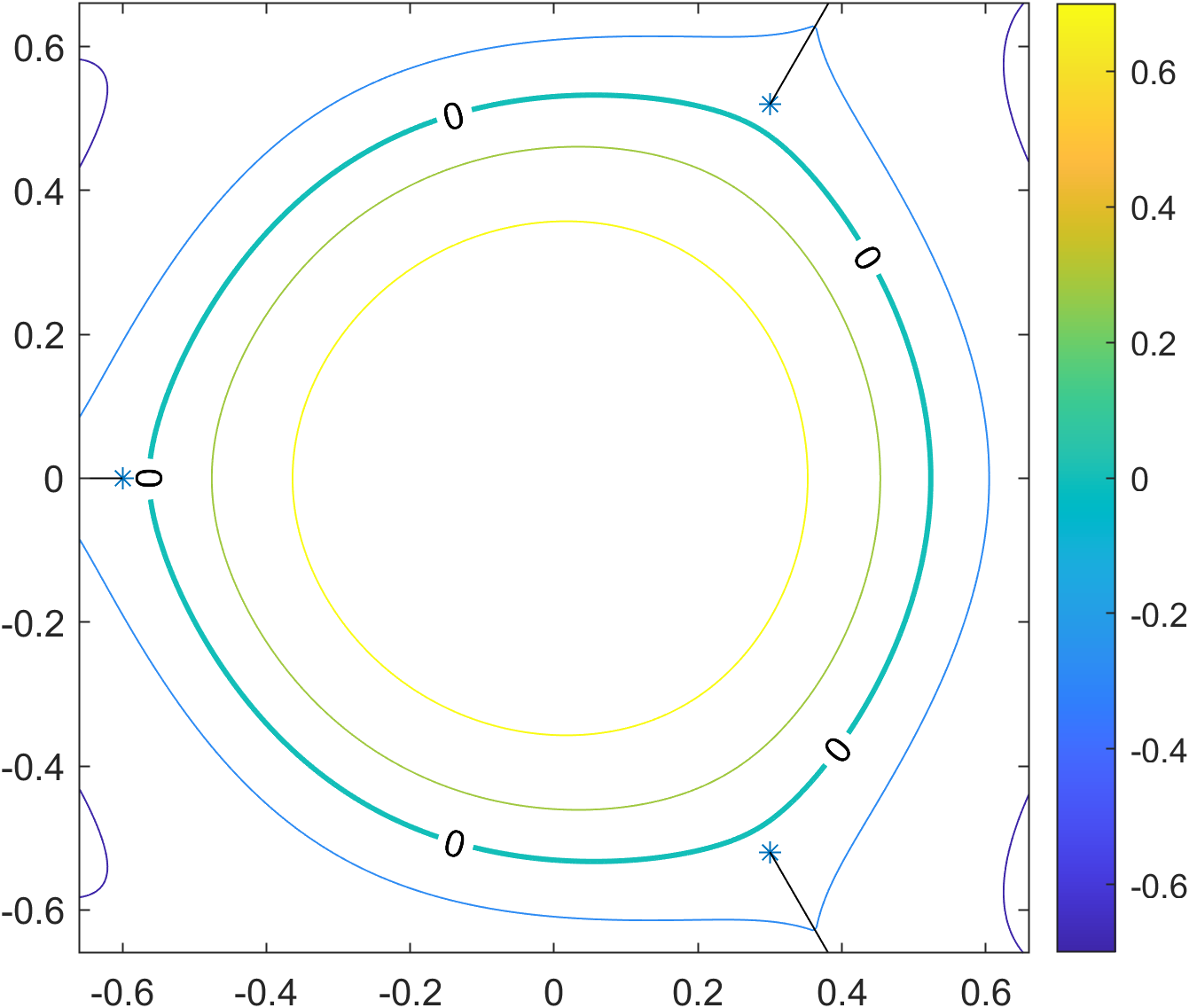}
				\caption{Curves of constant value of the function $v$.\\\phantom{l}\\\phantom{l}}
				\label{fig:EPzero}
			\end{subfigure}
			\caption{Illustration of the properties of the functions involved in Example \ref{ex:Eck3/2}. \label{fig:EPsign}}
			\end{figure}
\end{example}

\noindent We can modify Example \ref{ex:Eck3/2} by taking different parameters in the definition of the function $v$ appearing in  \eqref{eq:fR}. This allows us to exhibit other analytic domains $\Om$ where the Dirichlet Laplacian has eigenfunctions which are extendable only in a neighbourhood of $\Om$. In Figure \ref{fig:EPmore}, we display certain of these domains. 

	\begin{figure}[!ht]
		\begin{subfigure}{.3\textwidth}
			\centering
			\includegraphics[width=\textwidth]{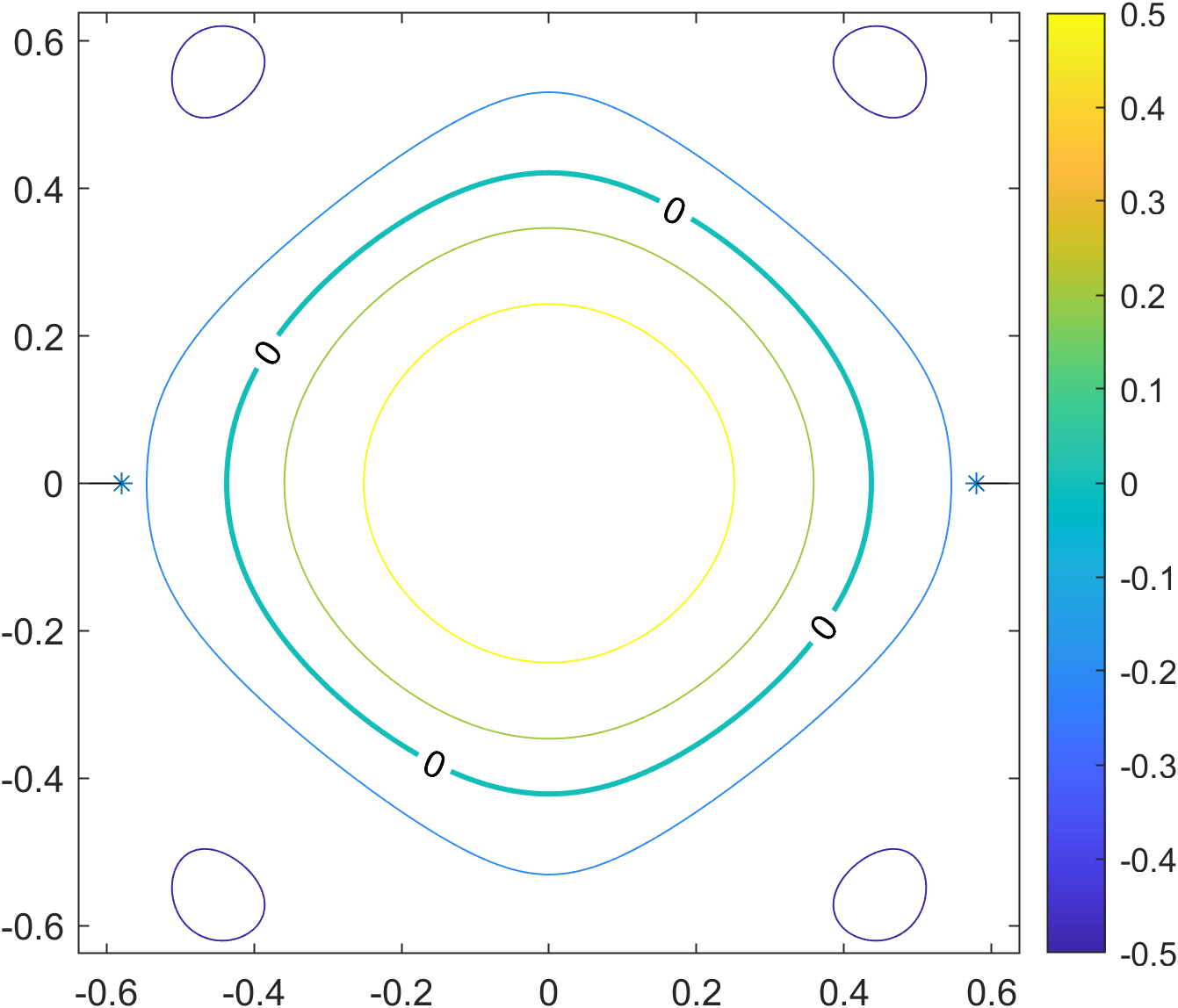}
			\caption{$L=2$, $\mu=5/2$, $a=0.58$.}
\label{fig:AnaLoc1}
		\end{subfigure}\hfil
	\begin{subfigure}{.3\textwidth}
		\centering
		\includegraphics[width=\textwidth]{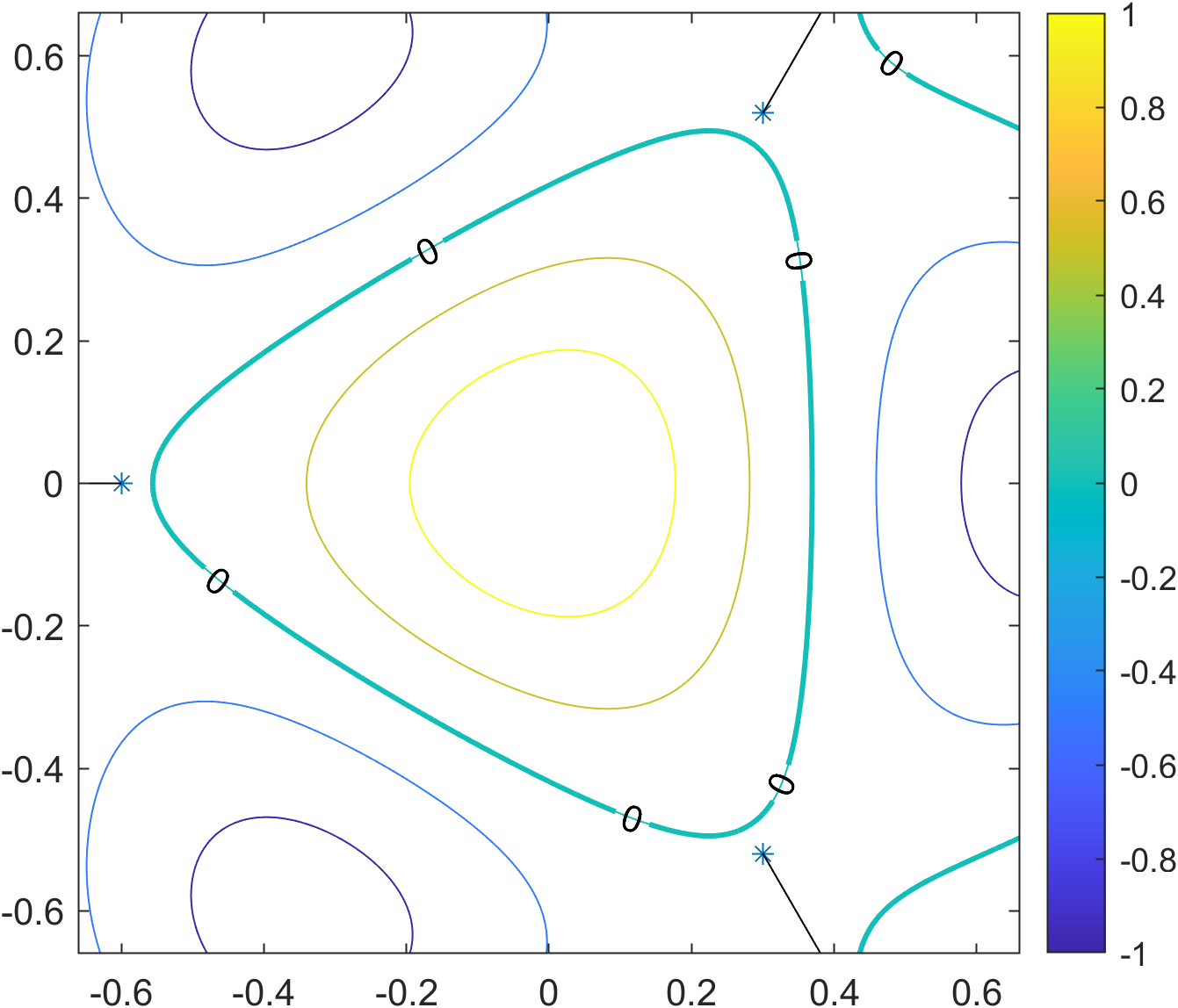}
		\caption{$L=3$, $\mu=5/2$, $a=0.6$.}
		\label{fig:AnaLoc2}
	\end{subfigure}\hfil
		\begin{subfigure}{.3\textwidth}
			\centering
			\includegraphics[width=\textwidth]{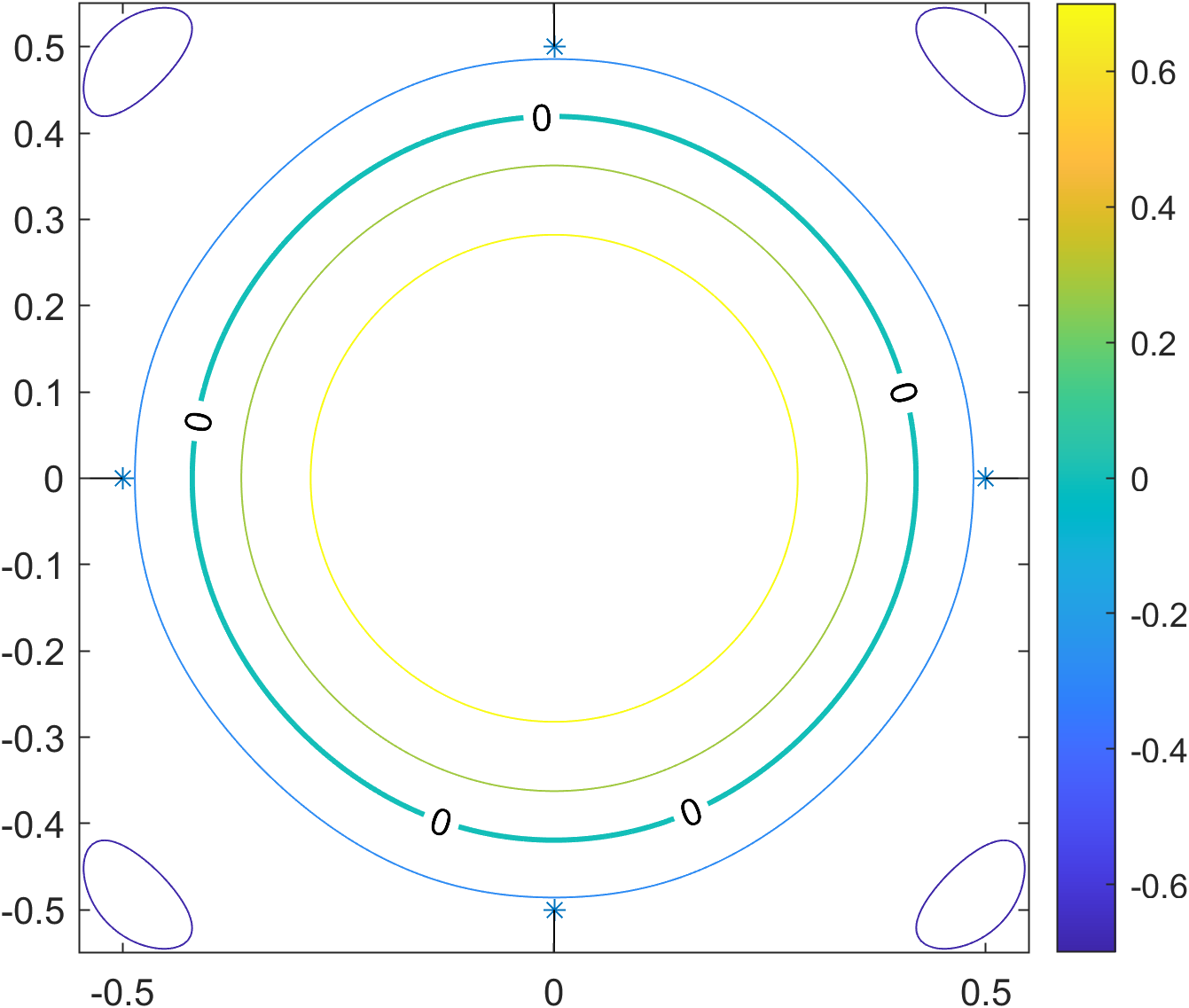}
			\caption{$L=4$, $\mu=5/2$, $a=0.5$}
\label{fig:AnaLoc3}
		\end{subfigure}
	\caption{Curves of constant value of the function $v$ for different parameters in (\ref{eq:fR}). In the analytic domains enclosed by the curve $v=0$, the Dirichlet Laplacian has eigenfunctions which are only locally extendable.}
	\label{fig:EPmore}
	\end{figure}~
	
\noindent Note that if we take the Bessel order $\mu$ in \eqref{eq:fR} in $\NN^\ast$, we can find domains $\Omega$ for which the eigenfunctions of the Dirichlet Laplacian can be extended in the whole $\RR^2$. We present such examples in Figure \ref{fig:EPint}. Let us emphasize that the domains $\Omega$ may have analytic boundaries as in Figure \ref{fig:EPintSmooth}, or Lipschitz boundaries as in Figures \ref{fig:EPintLip4}--\ref{fig:EPintLip5}. The latter situation appears when two or more nodal curves of the function $v$ intersect. For example, in Figure~\ref{fig:EPintLip4}, both $v$ and $\nabla v$ vanishes at the two points $\pm(a+1,0)$, but $\mathcal{H} v$,  the Hessian of $v$, is different from zero at $\pm(a+1,0)$. Hence there are two nodal curves intersecting at $\pm(a+1,0)$.
In Figure~\ref{fig:EPintLip5}, at $(0,-\sqrt{2}/2)$ we have 
$\partial^\alpha v=0$ for all multi-indices $\alpha\in\mathbb{N}^2$ with $0\le |\alpha|\le 2$ but $\partial_{x}^3v\neq 0$. This explains why there are three nodal curves  passing through $(0,-\sqrt{2}/2)$.
\begin{figure}[!ht]
	\begin{subfigure}{.3\textwidth}
		\centering
		\includegraphics[width=\textwidth]{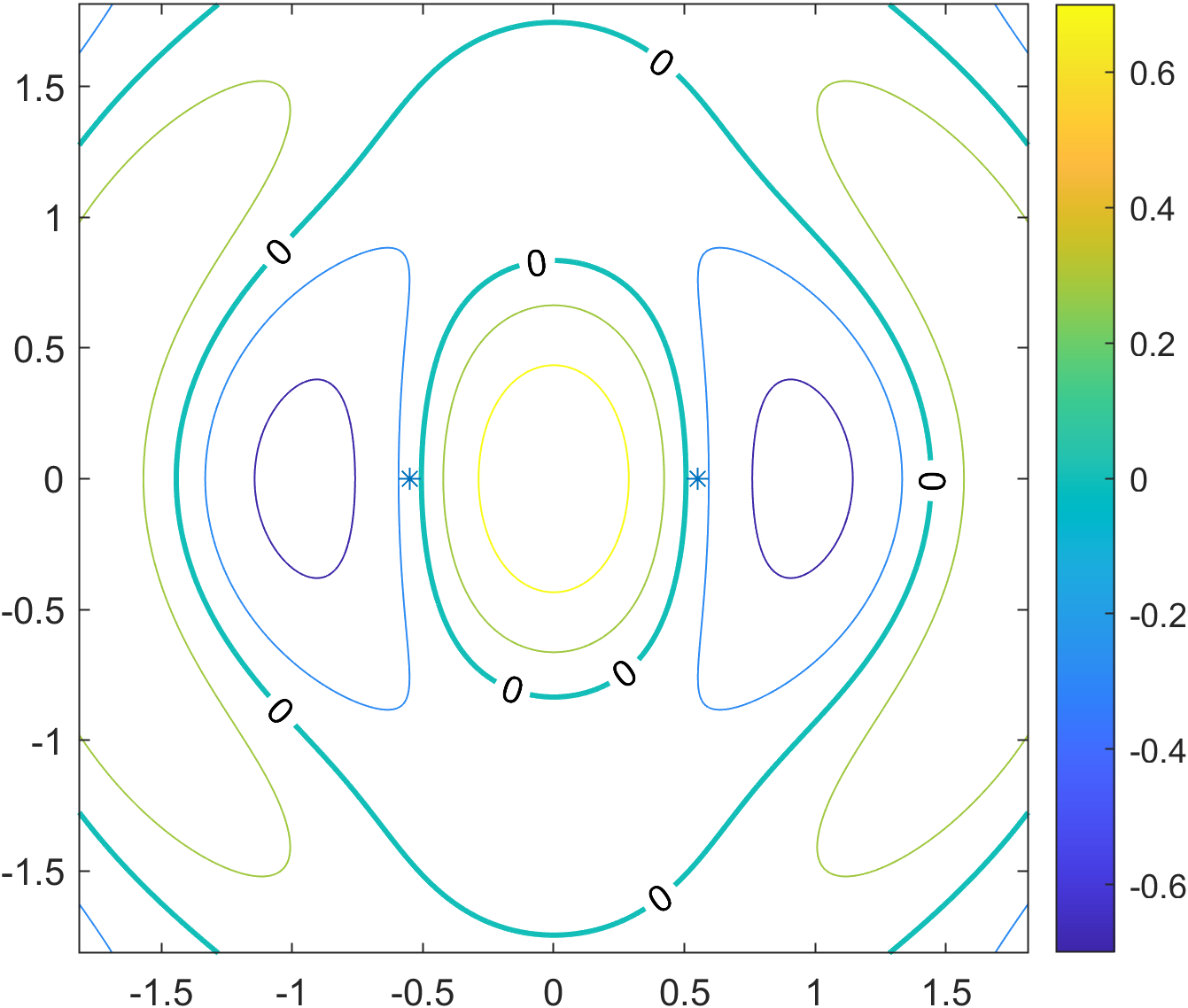}
		\caption{$L=2$, $\mu=1$, $a=0.55$.\\\phantom{$\sqrt{2}$}}
		\label{fig:EPintSmooth}
	\end{subfigure}
	\hfil
	\begin{subfigure}{.3\textwidth}
		\centering
		\includegraphics[width=\textwidth]{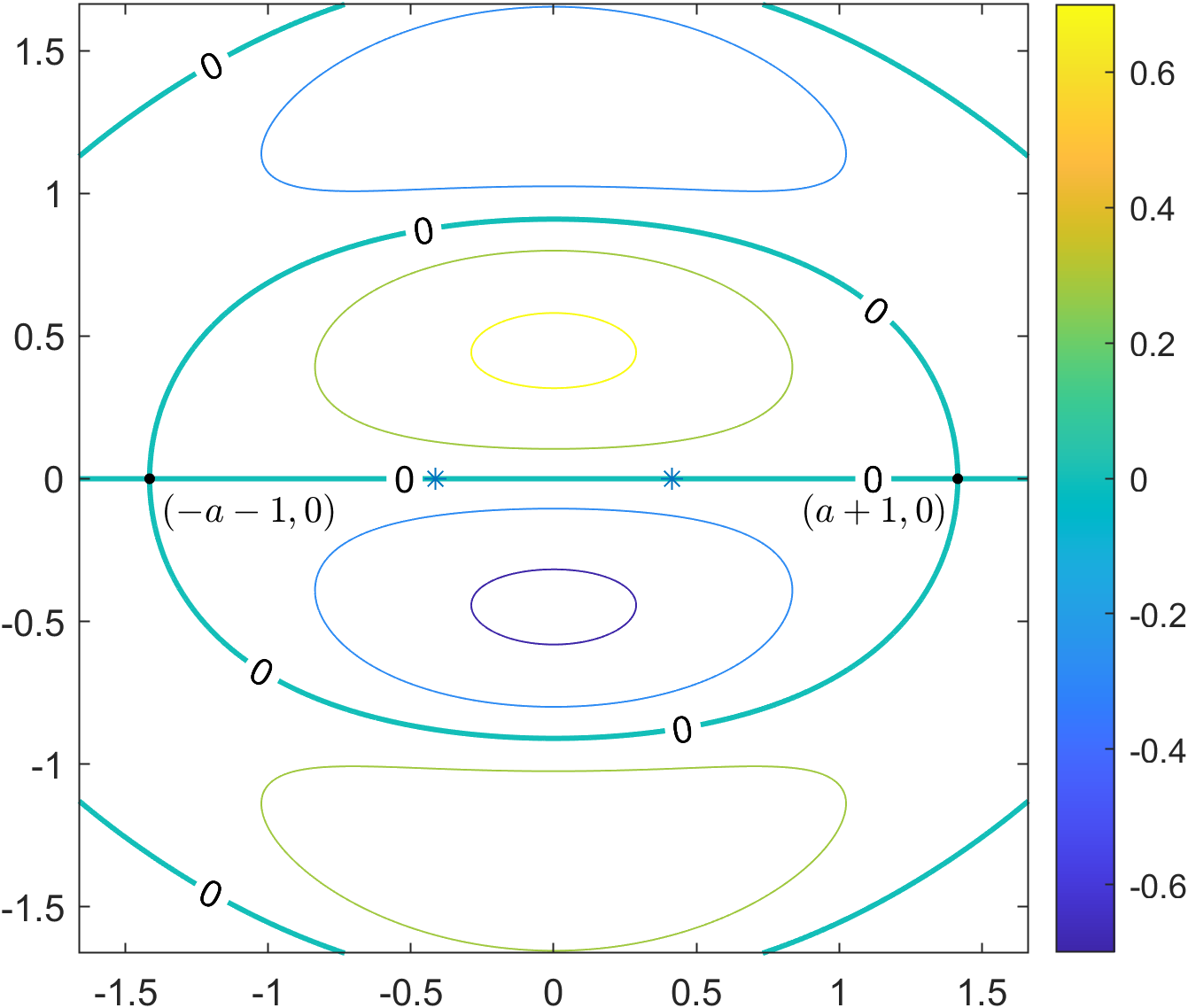}
		\caption{$L=2$, $\mu=1$, \\$a=\frac{{k_2}-{k}}{2k}$\,\footnotemark, $\phi_0=\pi/2=-\phi_1$.}
		\label{fig:EPintLip4}
	\end{subfigure}\hfil
	\begin{subfigure}{.3\textwidth}
		\centering
		\includegraphics[width=\textwidth]{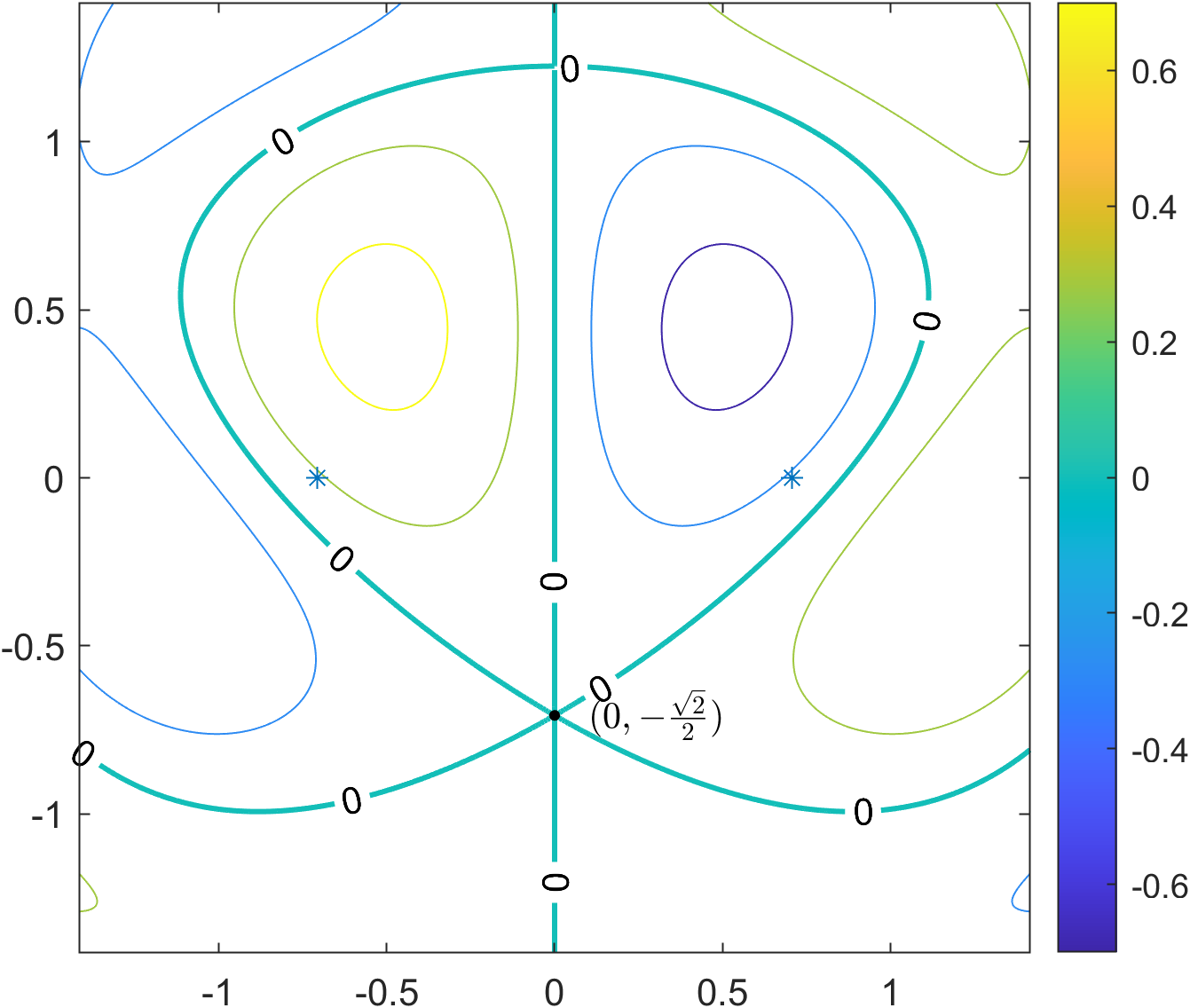}
		\caption{$L=2$, $\mu=1$, \\$a=\sqrt{2}/2$, $\phi_0=\pi/4$, $\phi_1=3\pi/4$.}
		\label{fig:EPintLip5}
	\end{subfigure}
	\caption{Curves of constant value of the function $v$ for different parameters in (\ref{eq:fR}). Here we get analytic and Lipschitz domains where the Dirichlet Laplacian has eigenfunctions which extend to solutions of the Helmholtz equation in $\RR^2$.}
	\label{fig:EPint}
\end{figure}
\footnotetext{Here $k_2$ denotes the second positive zero of $J_{\mu}$.}~

\noindent Similarly, we can find Lipschitz domains $\Omega$ for which there exist eigenfunctions are the Dirichlet Laplacian which can be extended in a neighbourhood of $\overline{\Omega}$ but not to the whole $\RR^2$ (see Figure \ref{fig:EPLip0}). We refer the reader to the Appendix for the justification of existence of such nodal curves as shown in Figures \ref{fig:EPmore}--\ref{fig:EPLip0}.
\begin{figure}[!ht]
	\begin{subfigure}{.3\textwidth}
		\centering
		\includegraphics[width=\textwidth]{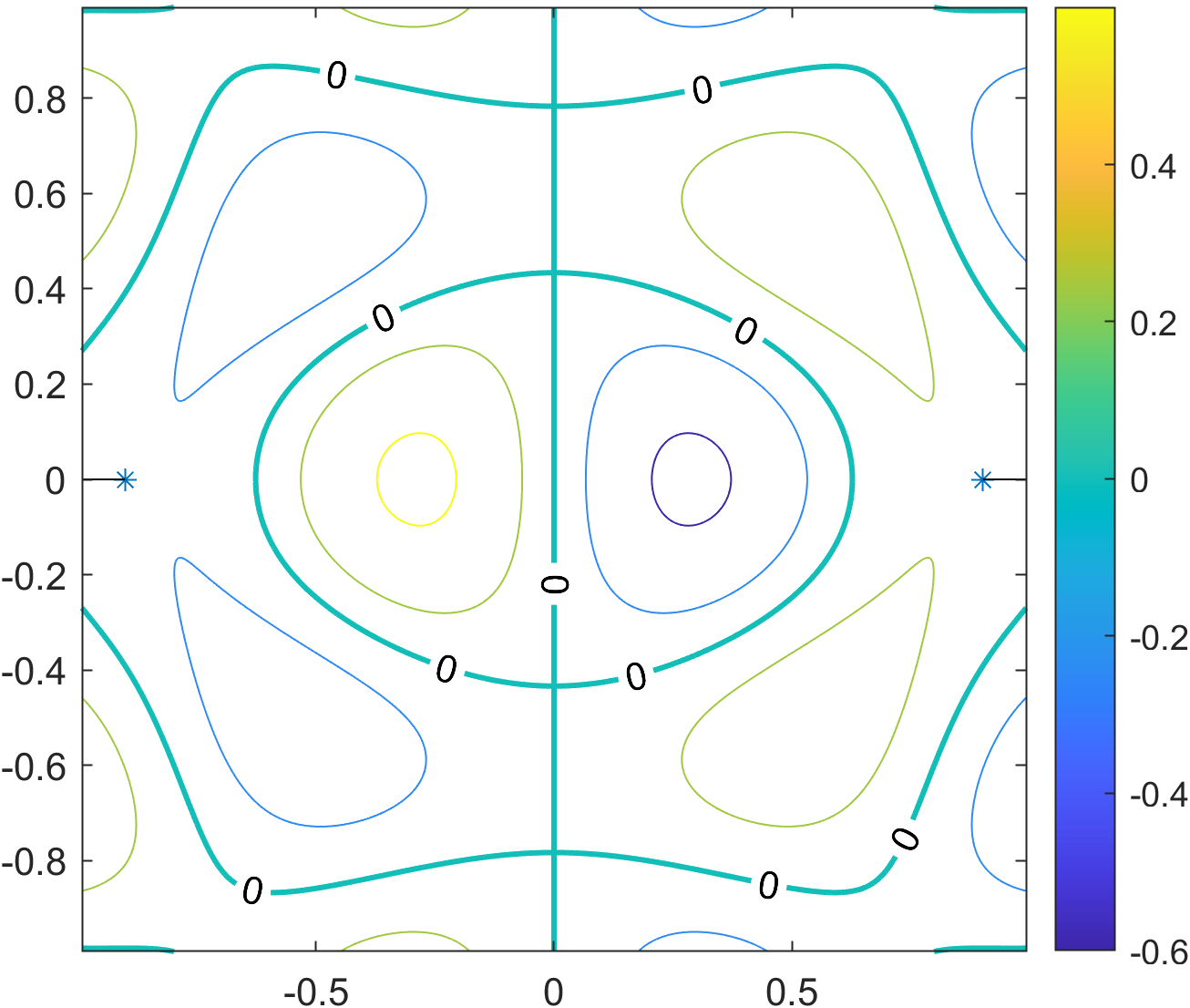}
		\caption{$L=2$, $\mu=7/2$, $a=\cos(\pi/7)$, $b_0=1=-b_1$.}
		\label{fig:EPLip}
	\end{subfigure}\hfil
\begin{subfigure}{.3\textwidth}
\centering
\includegraphics[width=\textwidth]{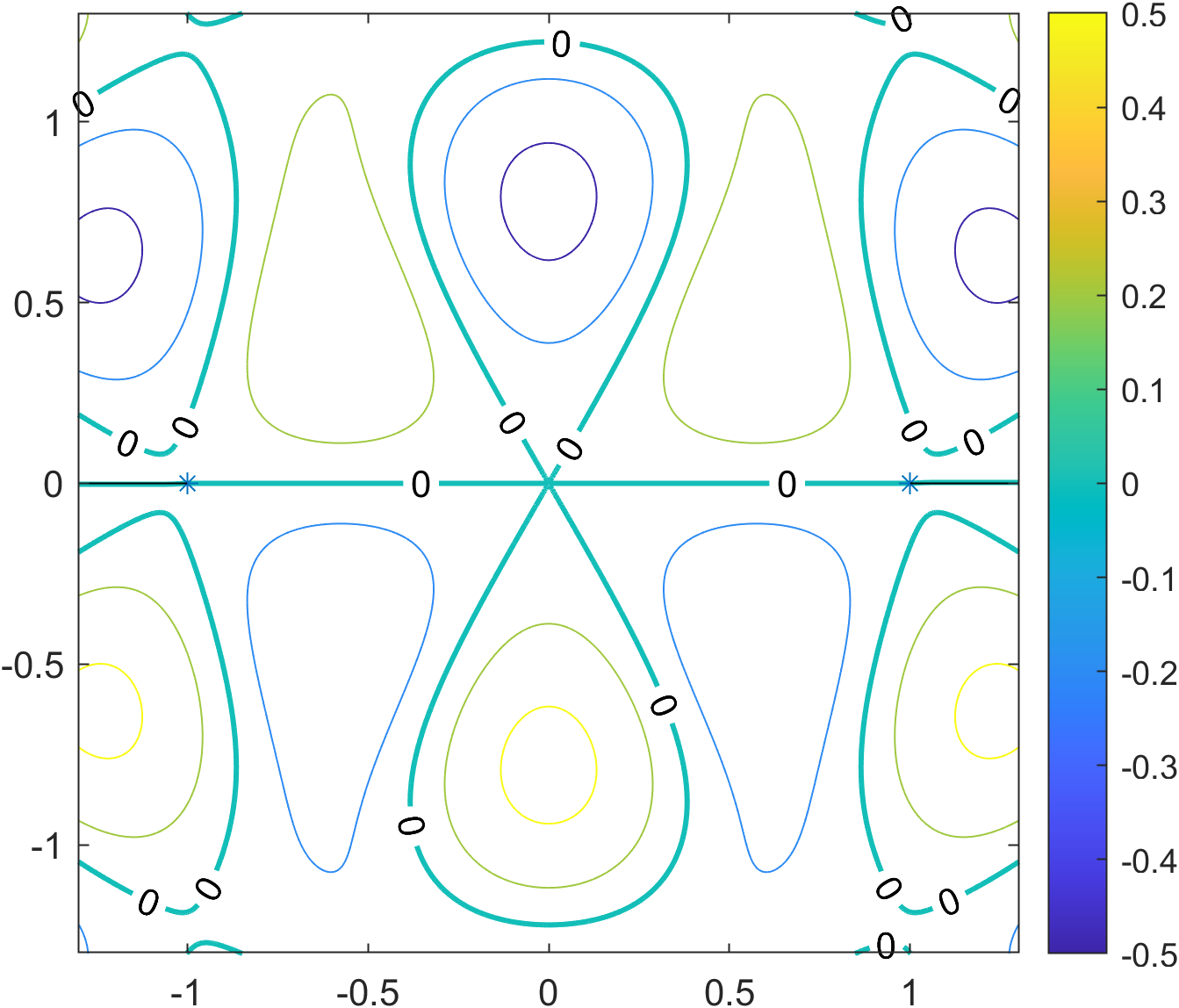}
\caption{$L=2$, $\mu=5/2$, $a=1$, $\phi_0=\pi/5=-\phi_1$.}
\label{fig:EPLip2}
\end{subfigure}\hfil
\begin{subfigure}{.3\textwidth}
	\centering
	\includegraphics[width=\textwidth]{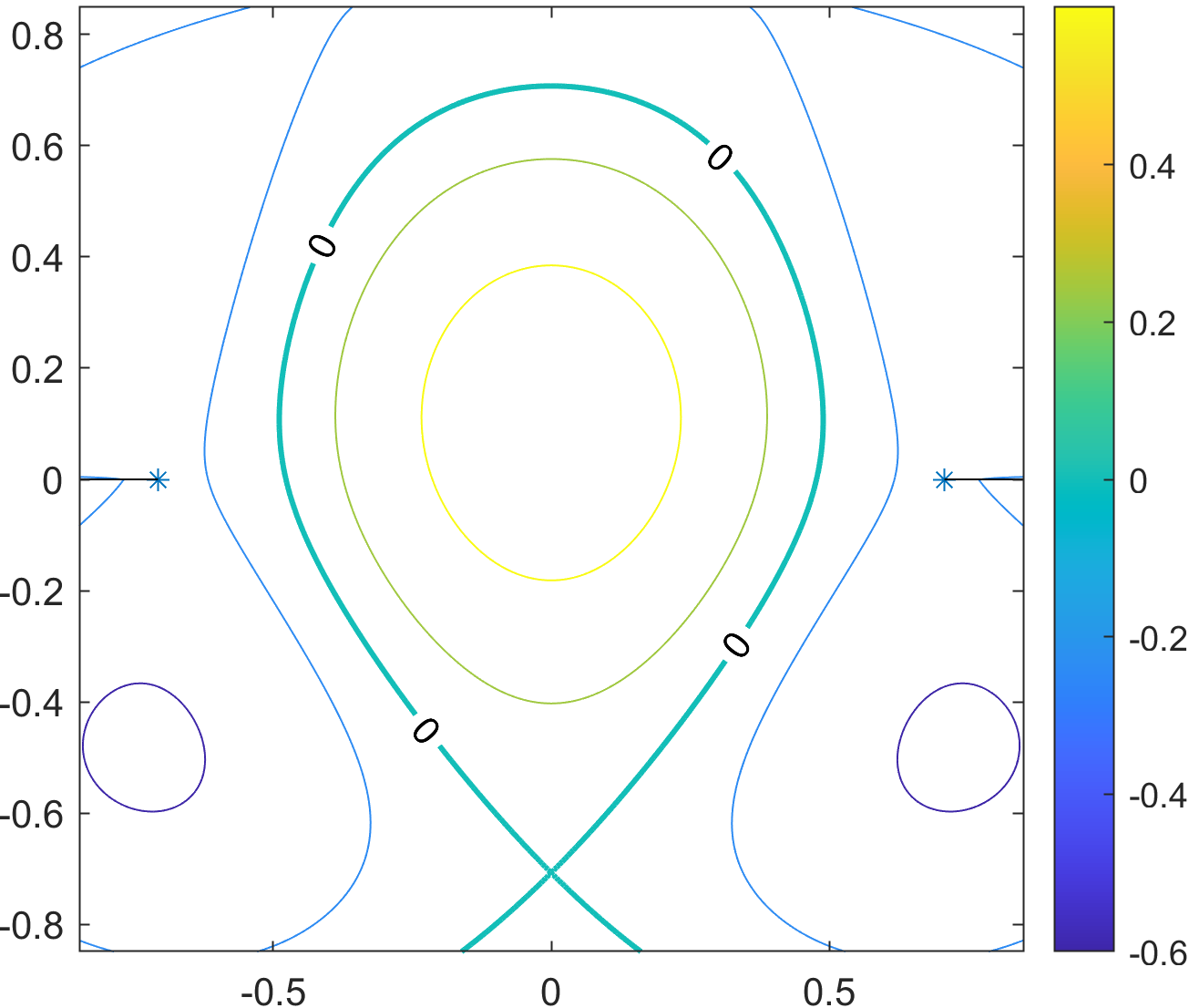}
	\caption{$L=2$, $\mu=3/2$, $a=\sqrt{2}/2$, $\phi_0=\pi/{12}=-\phi_1$.}
	\label{fig:EPLip3}
\end{subfigure}
\caption{Curves of constant value of the function $v$ for different parameters in (\ref{eq:fR}). Here we get Lipschitz domains where the Dirichlet Laplacian has eigenfunctions which are only locally extendable.}
\label{fig:EPLip0}
\end{figure}

\noindent We conclude this paragraph with an important result which generalizes the constraint we have in Proposition \ref{Propo2} concerning the angle of a sector so that non-scattering frequencies can exist. 
\begin{equation}\label{DefDomD}
\begin{array}{|l}
\mbox{Let $\Om$ be a domain such that the Dirichlet Laplacian has an eigenfunction $v$ that}\\
\mbox{can be extended into a function $\tilde{v}$ satisfying $\Delta \tilde{v}+k^2\tilde{v}=0$ in a neighbourhood of $\overline{\Om}$.}
\end{array}
\end{equation}
First, for $\Om$ as in (\ref{DefDomD}), due to the implicit functions theorem, $\partial\Om$ must be piecewise analytic. Additionally, we have the following statement (see also \cite[Chapter V, \S 16]{CoHi09}).
\begin{proposition}\label{rem:Dirdomains}
If $P\in\partial\Om$ is a corner point of a domain $\Om$ satisfying (\ref{DefDomD}) with opening angle $\alpha$, then necessarily $\alpha\in(0,2\pi)\cap\mathbb{Q}\pi$.
\end{proposition}
\begin{proof}
We start by observing that $\partial\Om$ belongs to the nodal set of $v$, an eigenfunction of the Dirichlet Laplacian in $\Om$. In an adapted system of coordinates, we can assume that $P$ coincides with the origin $O$. 
By the assumption that $v$ satisfies the Helmholtz equation in an open ball around $O$, we obtain that $v$ is of the form \begin{equation*}
v(x,y)=c_NJ_N(kr)\cos(N\theta-\theta_0)+f_{N+1}(x,y),
\end{equation*}
near $O$,
where $(r,\theta)$ are the polar coordinates of $(x,y)$. Here $c_N\in\RR\backslash\{0\}$ and $\theta_0\in[0,2\pi)$ are constants depending on $v$, $N\in\NN^\ast$ is the vanishing order of $v$ at $P$ and $f_{N+1}=O(r^{N+1})$ is a real-analytic function in $(x,y)$. As a consequence, the Taylor series of $v$ near $O$ is of the form
\begin{equation*}
v(x,y)=\tilde{c}_Nr^N\cos(N\theta-\theta_0)+\tilde{f}_{N+1}(x,y),
\end{equation*}
with $\tilde{c}_N\neq 0$ and $\tilde{f}_{N+1}$ real-analytic.
Hence, there are exactly $N$ analytic nodal curves of $v$ intersecting at $P$, with tangential directions along angles in the set $\{(\pi/2+\kappa\pi+\theta_0)/N\,|\,\kappa\in\mathbb{Z}\}$. Therefore we see that the angle between two tangential directions is necessarily equal to $\kappa'\pi/N$ for some $\kappa'\in\NN$. 
\end{proof}

\subsubsection{Extendable Neumann eigenfunctions}\label{sec:Neumann}

What we did in \S\ref{DirichletEigenVec} to find domains $\Om$ where the Dirichlet Laplacian admits extendable eigenfunctions can not be directly applied for Neumann boundary conditions. Instead of looking for nodal curves, we should consider the characteristics associated with the gradient of solutions to the Helmholtz equation. 
Let us explain in the following.\\
\newline
Start with some non-zero real valued function $v$ satisfying $\Delta v+k^2v=0$ in some domain, say $\RR^2$ for simplicity, for a given $k>0$. Pick some point $(x_0,y_0)$ such that $\nabla v(x_0,y_0)\ne0$. Then consider the gradient system
\begin{equation}\label{eq:ODENeu}
\begin{array}{|rcl}
X'(t)&=&\nabla v(X(t))\\[3pt]
X(0)&=& (x_0,y_0).
\end{array}
\end{equation}
The Cauchy-Lipschitz theorem ensures that (\ref{eq:ODENeu}) admits a unique maximal solution $X$ defined on an interval $I\subset\RR$ containing $0$. Set $\Gamma:=\{X(t)\,|\,t\in I\}\subset\RR^2$ and introduce $\nu$ a unit normal vector to the smooth curve $\Gamma$ whose orientation is arbitrarily fixed. For $t\in I$, we have
\[
\partial_\nu v(X(t))=\nu(X(t))\cdot \nabla v(X(t))=\nu(X(t))\cdot X'(t)=0
\]
because $X'(t)$ and $\nu(X(t))$ are respectively tangent and normal to the orbit $t\mapsto X(t)$. Thus $v$ satisfies an homogeneous Neumann boundary condition on $\Gamma$. Therefore if we are able to find a bounded domain $\Om$ whose boundary coincides with the union of orbits of Problem (\ref{eq:ODENeu}) for different initial conditions $(x_0,y_0)$, then this shows that $k$ belongs to the set $\mathscr{S}_{\mrm{NS}}$ associated with $\Om$. Let us illustrate the approach with some specific examples. 

\begin{example}\label{ex:cosxy}
Consider the function 
\begin{equation}\label{DefFuncv}
v(x,y)=-(\cos x+\cos y)    
\end{equation}
which satisfies $\Delta v+2v=0$ in $\RR^2$. For this $v$, \eqref{eq:ODENeu} reads
\begin{equation}\label{ODEex1}
\begin{array}{|rcl}
x'(t)&=&\sin x(t)\\[3pt]
y'(t)&=&\sin y(t)\\[3pt]
(x(0),y(0))&=& (x_0,y_0).
\end{array}
\end{equation}
The stationary points of this system of ordinary differential equations are the $(m\pi,n\pi)$, $m,n\in\mathbb{Z}$. These are constant orbits. Moreover, one can check that
\[
\{0\}\times(0,\pi),\quad\{\pi\}\times(0,\pi),\quad(0,\pi)\times\{0\},\quad(0,\pi)\times\{\pi\}
\]
are also orbits of (\ref{ODEex1}). Now pick some $(x_0,y_0)\in (0,\pi)^2$. Since different orbits cannot cross, we know that the one passing through $(x_0,y_0)$ stays inside $(0,\pi)^2$ and so the associated trajectory is global, i.e. defined for all $t\in\RR$. It turns out that we can compute it explicitly and we find
\[
x(t)=2\arctan(e^t\tan\frac{x_0}{2}),\qquad\quad y(t)=2\arctan(e^t\tan\frac{y_0}{2}).
\]
After a few operations, we obtain that the corresponding orbit coincides with the curve
\[
\Gamma(\delta):=\{(x,y(x))=2\arctan (\delta\tan (x/2)))\,|\,x\in(0,\pi)\}
\]
where $\delta:=\tan(y_0/2)\tan(x_0/2)>0$ (see Figure \ref{fig:NeumannTrig} left for representations of different $\Gamma(\delta)$). For any $0<\delta_1<\delta_2$, denote $\Om(\delta_1,\delta_2)$ the domain located between $\Gamma(\delta_1)$ and $\Gamma(\delta_2)$. Then $v$ is an eigenfunction of the Neumann Laplacian in $\Om(\delta_1,\delta_2)$ (see Figure \ref{fig:NeumannTrig} right for pictures of some $\Om(\delta_1,\delta_2)$). Since $v$ satisfies $\Delta v+2v=0$ in $\RR^2$, we deduce that \Blu{$\sqrt{2}$} belongs to $\mathscr{S}_{\mrm{NS}}$ for this domain. \\
Note that we also have \Blu{$\sqrt{2}\in \mathscr{S}_{\mrm{NS}}$} for the domain delimited for example by the curves 
\[
\{0\}\times(0,\pi),\qquad\{\pi\}\times(0,\pi),\qquad \Gamma(\delta),\qquad \mbox{for any $\delta>0$}.
\]

\begin{figure}[!ht]
	\centering
\includegraphics[width=7.9cm]{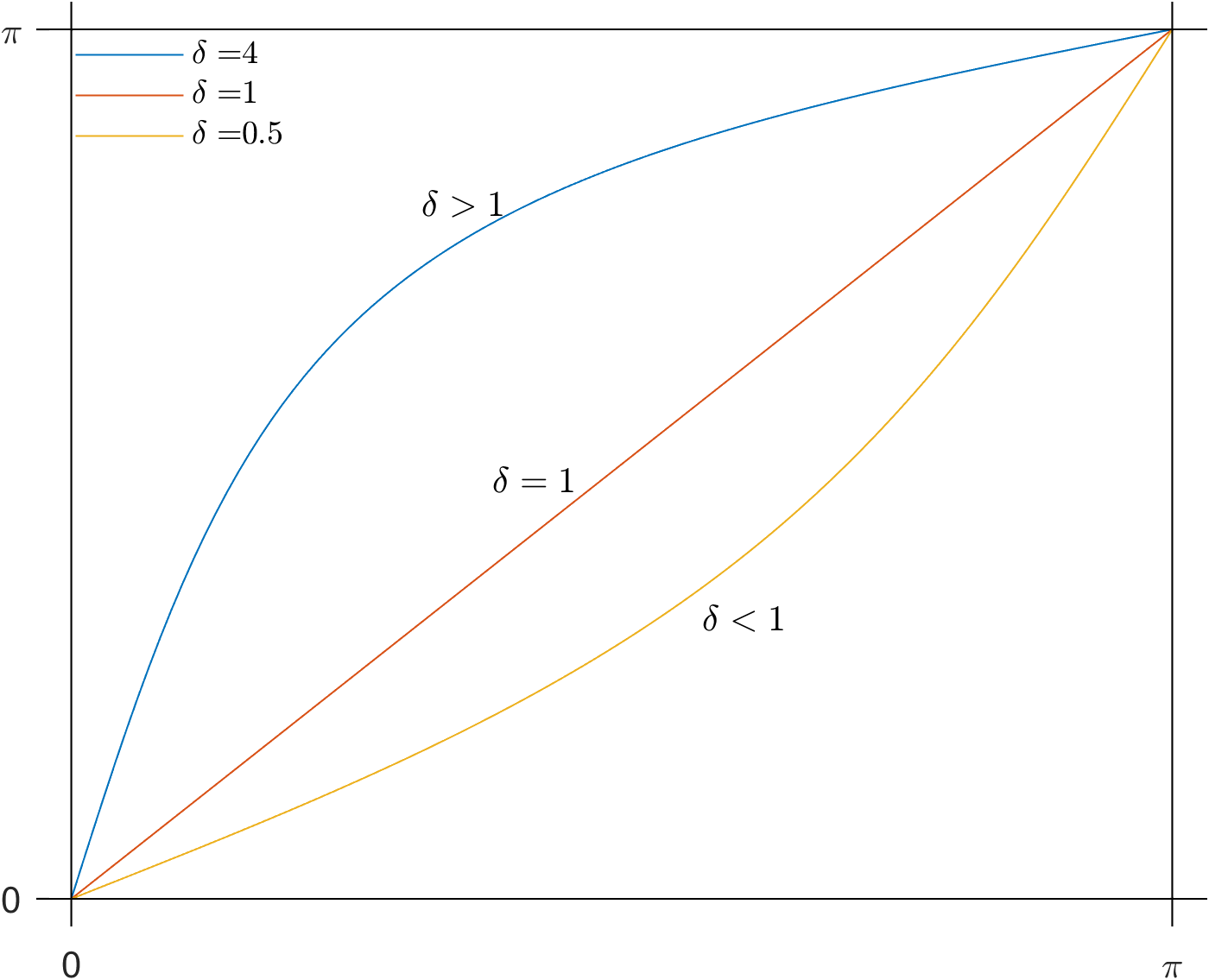}\qquad
\raisebox{3.2cm}{\begin{tabular}{cc}
\includegraphics[width=3.2cm]{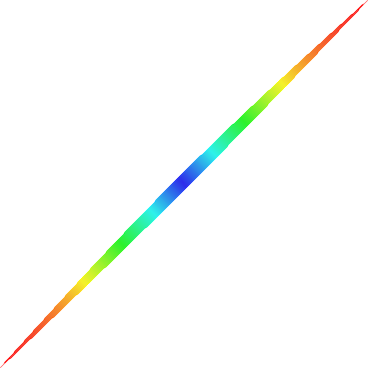} & \includegraphics[width=3.2cm]{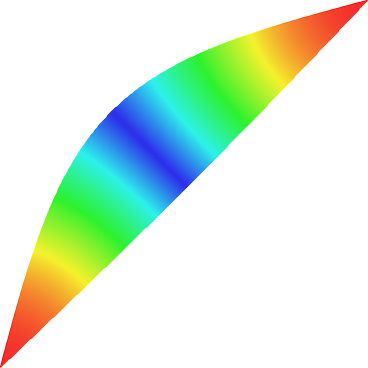} \\
\includegraphics[width=3.2cm]{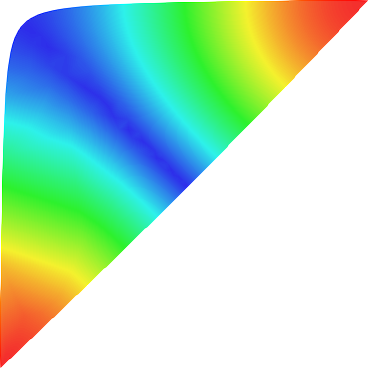} & \includegraphics[width=3.2cm]{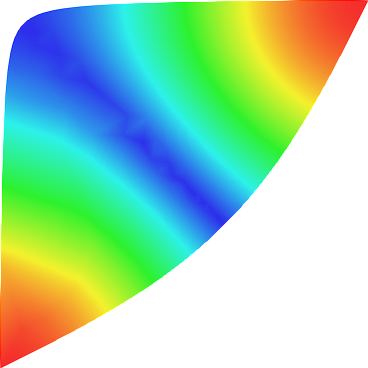} 
\end{tabular}}
	\caption{Left: curves $\Gamma(\delta)$ for $\delta=0.5, 1, 4$. Right: function $v$ defined in (\ref{DefFuncv}) in a few domains where it is an eigenfunction of the Neumann Laplacian associated with the eigenvalue $2$.}
	\label{fig:NeumannTrig}
\end{figure}
\end{example}

\noindent At this stage, let us make a few comments. First, when one enlarges a bounded domain, due to the \textit{min-max} principle, the eigenvalues of the Dirichlet Laplacian decrease. This is not true for the eigenvalues of the Neumann Laplacian, roughly speaking because one cannot extend them by zero. This is well illustrated by the above example. We have a family of \Blu{continuously} deformed shapes which all have $2$ among the eigenvalues of the Neumann Laplacian, \Blu{that is, $\sqrt{2}\in \mathscr{S}_{\mrm{N}}$}.\\
\newline
Second, this example shows that one can find domains where the Neumann Laplacian has eigenfunctions which extend as solutions to the Helmholtz solution in the whole space and which admit corners with arbitrary apertures in $(0,2\pi)$. This is very different from the Dirichlet case where in order to be able to extend the eigenfunctions of the Laplace operator, the aperture of the corners must be in $(0,2\pi)\cap\mathbb{Q}\pi$ (see Proposition \ref{rem:Dirdomains}).\\
\newline
The reason we can have such domains $\Om$ with corners of arbitrary aperture in Example~\ref{ex:cosxy} is that $\mathcal{H} v(O)$, the Hessian of $v$ at the stationary point $O=(0,0)$, is equal to the identity matrix. As a consequence, the linearized problem of (\ref{ODEex1}) at $O$ simply writes $U'(t)=U(t)$ with $U=(U_x,U_y)^\top$. Solving it, we get $U_x(t)=a_x\,e^{t}$, $U_y(t)=a_y\,e^{t}$ for some constants $a_x$, $a_y$. Thus we see that the curve $\{(U_x(t),U_y(t)),\,t\in(-\infty;0)\}$ is a line which can take any direction by choosing properly $a_x$, $a_y$.  This explains why the orbits of (\ref{ODEex1}) can leave from $O$ asymptotically with any angle.\\
\newline
Let us consider a second example.
\begin{example}
\Blu{In this example we consider a particular case when $k^2=5$.} Start from the function $v$ such that $v(x,y)=\cos x\cos 2y$. It satisfies $\Delta v+5v=0$ in $\RR^2$. The set of stationary points associated with \eqref{eq:ODENeu} is $\{(m\pi,n\pi/2), (\pi/2+m\pi,\pi/4+n\pi/2)\,|\,m,n\in\mathbb{Z}\}$. By computing the Hessian $\mathcal{H}v$, one finds that the $(\pi/2+m\pi,\pi/4+n\pi/2)$ are saddle points. At saddle points, orbits of (\ref{eq:ODENeu}) can meet the equilibrium point at only two particular directions (see Figure \ref{fig:NeumannTrigCuspR}). On the other hand, for $m,n\in\mathbb{Z}$, we have $\mathcal{H}v(m\pi,n\pi/2)=(-1)^{m+n-1}\mathrm{diag}(1,4)$. Hence the orbits of \eqref{eq:ODENeu} are all tangential to the horizontal line at $(0,0)$ except one that is tangential to the vertical line. Indeed, for example for $m=n=0$, the linearized problem associated with \eqref{eq:ODENeu} writes $U'(t)=-\mathrm{diag}(1,4)\cdot U(t)$ so that its solution is given by $U_x(t)=a_x\,e^{-t}$, $U_y(t)=a_y\,e^{-4t}$ for some constants $a_x$, $a_y$. Thus we see that except when $a_x=0$, the curve $\{(U_x(t),U_y(t)),\,t\in(0;+\infty)\}$ is tangential to the horizontal line when $t\to+\infty$.\\
We can make this more explicitly by computing the orbits. First we find that 
\[
(0,0),\quad (\pi,0),\quad (\pi/2,0),\quad (\pi/2,\pi),\quad (0,\pi)\times\{0\},\quad \{\pi\}\times (0,\pi/2),\quad (0,\pi)\times\{\pi/2\},\quad \{0\}\times (0,\pi/2)
\]
are particular orbits of (\ref{eq:ODENeu}). Then we obtain that the (global) trajectory $(x(t),y(t))$ passing through $(x_0,y_0)\in(0,\pi)\times(0,\pi/2)$ is such that
\begin{equation*}
		\sin (2y(t))=\delta \sin^4x(t)
	\end{equation*}
for $\delta>0$.	We show in Figure \ref{fig:NeumannTrigCusp} the normalized gradient flow and some of the ``Neumann curves''. The latter enclose domains $\Om$ for which $v$ is a Neumann eigenfunction so that \Blu{the frequency $k=\sqrt{5}\in\mathscr{S}_{\mrm{NS}}$}. We get that $\Om$ can be of class $\mathscr{C}^1$, \Blu{Lipschitz,} but may also possess cusps.

\begin{figure}[!ht]
		\begin{subfigure}{.5\linewidth}
			\centering\includegraphics[width=\textwidth]{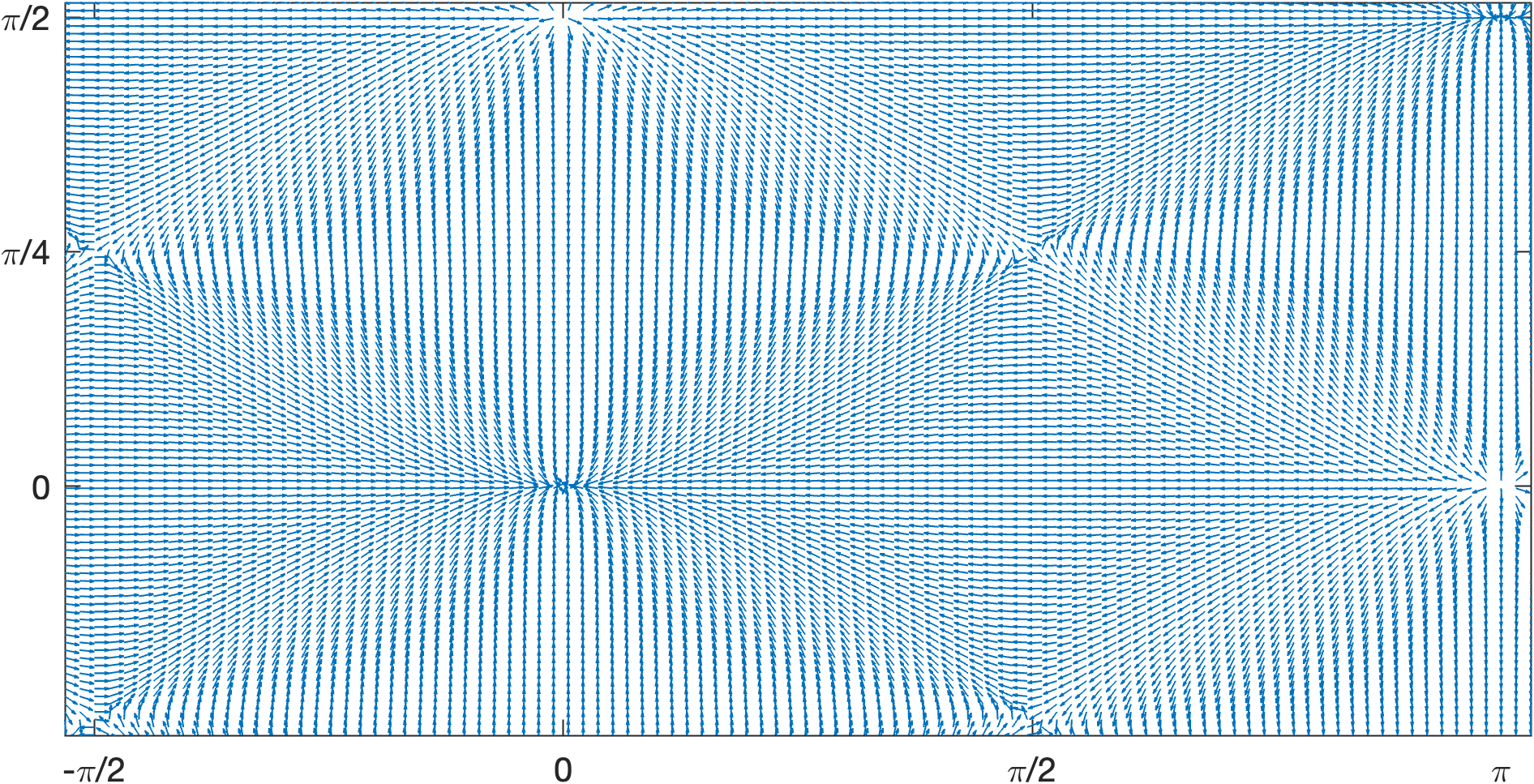}
			\caption{Normalized gradient $\nabla v/|\nabla v|$.}
		\end{subfigure}
	\begin{subfigure}{.5\linewidth}
		\centering\raisebox{0.8cm}{\includegraphics[width=0.9\textwidth]{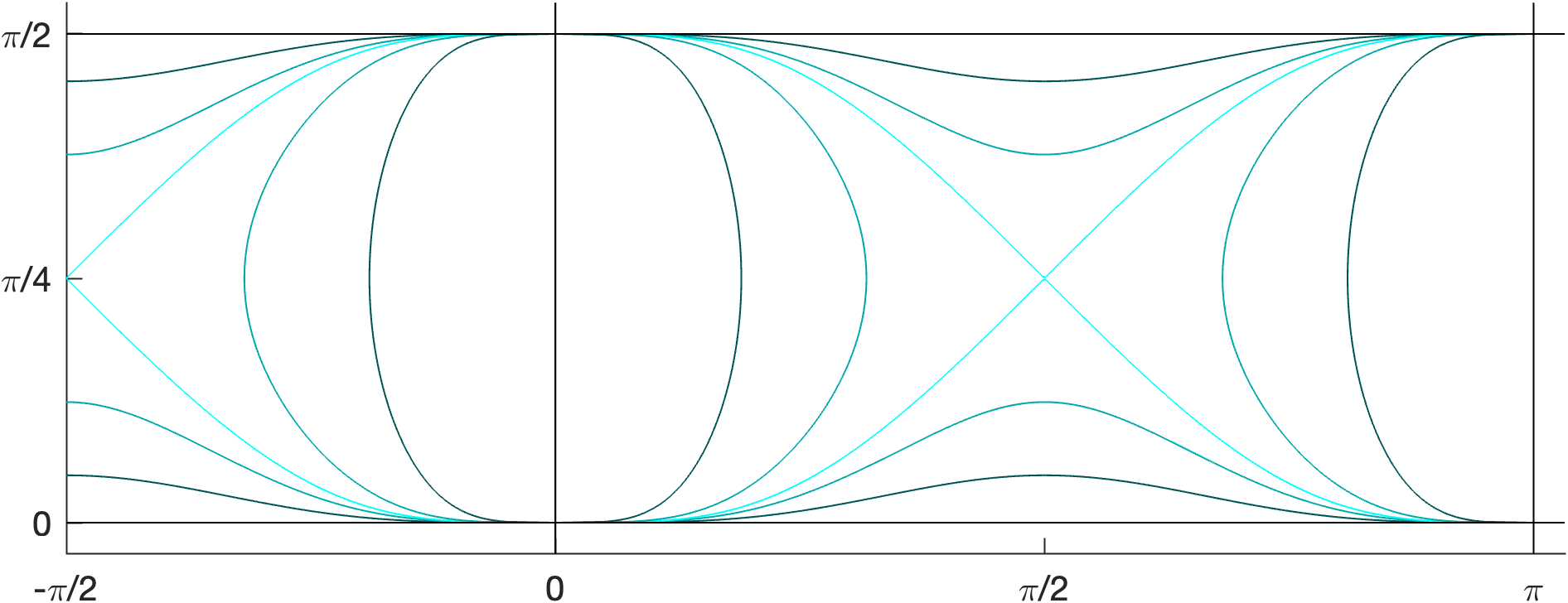}}
		\caption{Neumann curves.\label{fig:NeumannTrigCuspR}}
	\end{subfigure}
	\caption{Case $v(x,y)=\cos x\cos 2y$.}
	\label{fig:NeumannTrigCusp}
	\end{figure}
\end{example}

\begin{example}[Neumann eigenfunctions extendable in a neighbourhood of $\overline{\Omega}$ but not to the whole $\RR^d$]
Let us \Blu{define} $v$ as in \eqref{eq:fRgen} with $L=2$, $\mu=7/2$, $a=\cos(\pi/7)$ and $b_0=1=-b_1$, with $k$ being the first positive zero of $J_{\mu}$ (see Figure \ref{fig:EPLip}).
We can calculate the gradient $\nabla v$ explicitly, say, when $|x_1|<a$, as
\[
\nabla v(x,y)=\sum_{l=0}^{1}(-1)^l\PAre{kJ'_\mu(k\rho_l)\cos(\mu\psi_l)\nabla \rho_l-\mu J_\mu(k\rho_l)\sin(\mu\psi_l)\nabla\psi_l},
\]
with
\[
\rho_l=\sqrt{\pare{a+(-1)^lx_1}^2+x_2^2},\AND
				\psi_l=(-1)^l\arcsin\frac{x_2}{\rho_l},\qquad l=0,1.
\]
It can be shown that $\nabla v=0$ at $(0,\pm\sin(\pi/7))$ and $(\pm t_0,0)$, where $t_0$ is the (only) solution to $J_{\mu}'\pare{k(a+t_0)}+J_{\mu}'\pare{k(a-t_0)}=0$ in the interval $(0,a)$.
	The normalized gradient flow is shown in Figure~\ref{fig:grad}. 
	Moreover, by (numerically) solving the system of ODEs \eqref{eq:ODENeu}, we can construct curves on which $\partial_{\nu}v=0$, and some of those curves enclose bounded $\mathscr{C}^1$ or Lipschitz domains (with corners of opening angle $\pi/2$), or domains with cusps.
	We display some of these domains in Figure \ref{fig:NeuDom}, where the starting points $X(0)=(x_0,y_0)$ satisfy $x_0=0$ and $y_0\in\{\alpha\sin(\pi/7)\,|\, |\alpha|<1\}$. In the numerics, we take $\pm\alpha\in\{0, 0.1, 0.5, 0.8, 0.9, 0.9999\}$.
	The curves on the left side of the starting points are obtained by solving \eqref{eq:ODENeu} and letting $t\to+\infty$. We get $\lim_{t\to+\infty}X(t)=(-t_0,0)$. Those on the right side are obtained by solving the ``backward'' system $X'(t)=-\nabla v(X(t))$, in which case $\lim_{t\to+\infty}X(t)=(t_0,0)$. An explicit computation gives $\mathcal{H}v(\pm t_0,0)=\pm\text{diag}\,(\lambda_1,\lambda_2)$ with $0<\lambda_2<\lambda_1$. As a consequence, though this does not appear clearly in Figure \ref{fig:NeumanC} because here we do not zoom at $(\pm t_0,0)$, all the orbits  except the horizontal one, are tangential to the vertical axis at $(\pm t_0,0)$. 
	
\begin{figure}[!ht]
\centering
		\begin{subfigure}{.4\textwidth}
			\includegraphics[scale=.8]{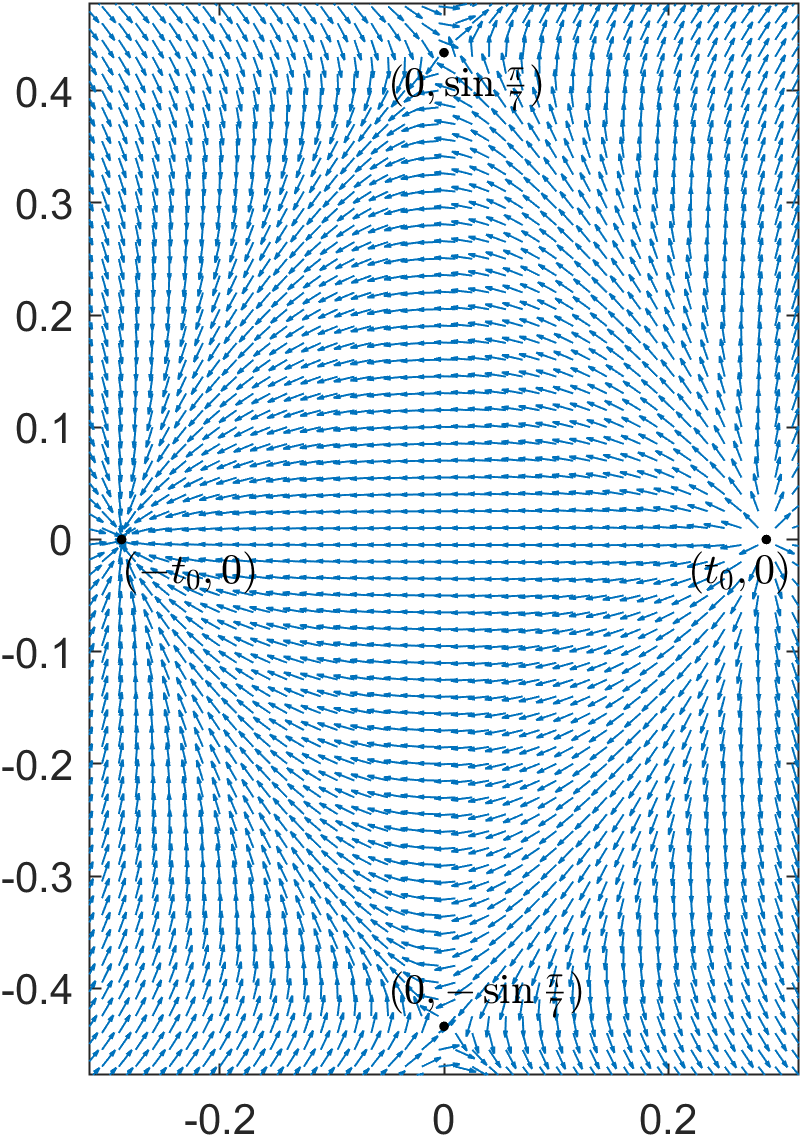}
			\caption{Normalized gradient $\nabla v/|\nabla v|$.}
			\label{fig:grad}
		\end{subfigure}\hfil
		\begin{subfigure}{.4\textwidth}
			\includegraphics[scale=.8]{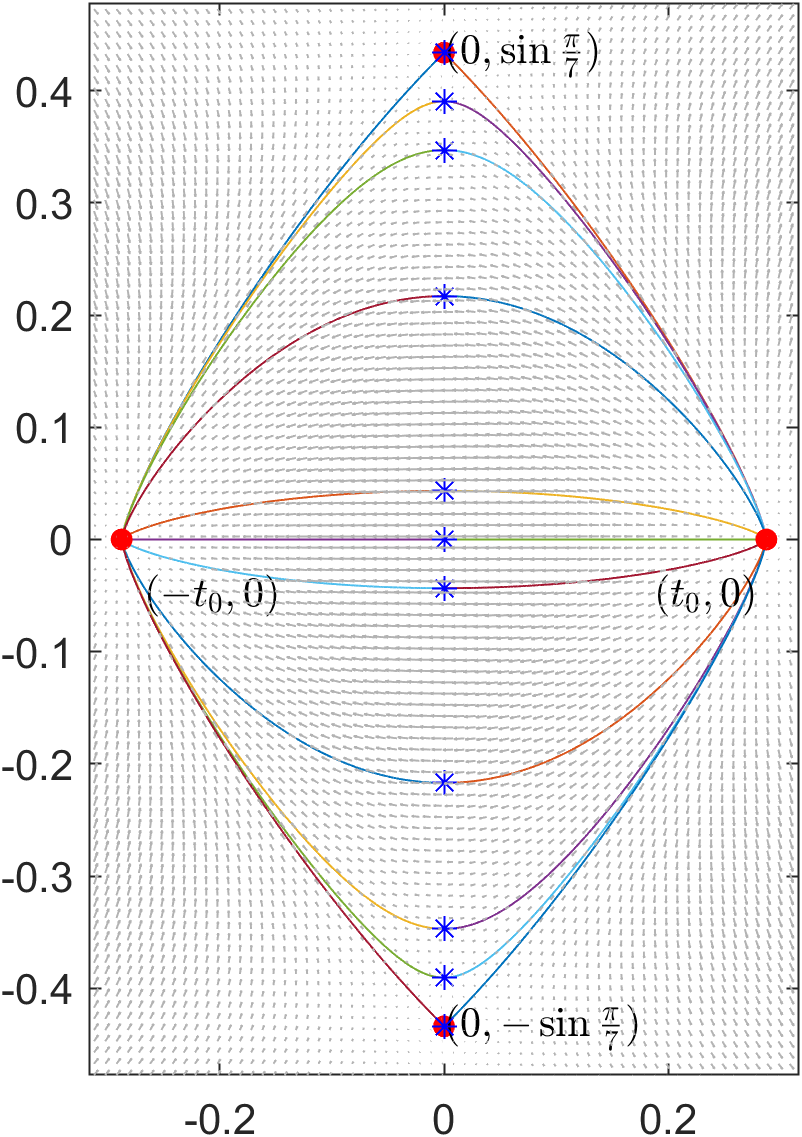}
			\caption{Neumann curves.\label{fig:NeumanC}}
		\end{subfigure}
		\caption{Case $v$ as in \eqref{eq:fRgen} with $L=2$, $\mu=7/2$, $a=\cos(\pi/7)$, $b_0=1=-b_1$.}
		\label{fig:NeuDom}
	\end{figure}
\end{example}

\section{Non-scattering with anisotropic materials}\label{section3}

In this section, we provide examples of anisotropic materials, more precisely materials for which $A$ in (\ref{CaseAq}) is a matrix valued function, that support non-scattering frequencies. For some of them, non-scattering occurs at all $k>0$ so that in particular the set of transmission eigenvalues $\mathscr{S}_{\mrm{TE}}$ contains $(0,+\infty)$. 

\subsection{Non-scattering via diffeomorphism}
We first look for non-scattering waves and media by using diffeomorphisms following for example \cite{KohnV84} (see also \cite{CakoniVX23}). Let $v$ be an entire solution to the Helmholtz equation, i.e. such that
\[
\Delta v+k^2 v=0 \qquad\mbox{in $\RR^d$}.
\]
Given a bounded Lipschitz domain  $\Omega$ in $\RR^d$, let $\Psi$ be a $\mathscr{C}^m(\overline{\Omega})$, $m\ge 2$, diffeomorphism 
satisfying $\Psi =\Id $ on $\partial\Omega$.
Define
\begin{equation}\label{eq:traq1}
{A}=\frac{D\Psi D^T\Psi}{\abs{\det D\Psi}}\circ\Psi^{-1}   \AnD {q}=\frac{1}{\abs{\det D\Psi}}\circ\Psi^{-1}  \qquad\mbox{in $\Omega$},
\end{equation}
where $D\Psi$ is the Jacobian matrix of $\Psi$. Note in particular that $A$ is symmetric. Then the function $u:=v\circ\Psi^{-1}$ satisfies $u=v$ on $\partial\Omega$ and
\begin{equation}\label{eq:aftr}
	\begin{split}
		\div(A \nabla u )+ k^2 q u=0 \quad\mbox{in $\Omega$},&\qquad\partial_{\nu}^{{A}}u=\partial_{\nu} v\quad\mbox{on $\partial \Omega$,}
	\end{split}
\end{equation}
with $\partial_{\nu}^A\cdot:=\nu\cdot A\nabla\cdot$. 
As a consequence, we have the following lemma. 
\begin{lemma}\label{lem:diffeoNonscat}
	Let $\Omega$  be a bounded Lipschitz domain in $\RR^d$ and let $\Psi$ be a $\mathscr{C}^2(\overline{\Omega})$ diffeomorphism on $\overline{\Omega}$ satisfying $\Psi =\Id $ on $\partial\Omega$. Define $A$, $q$ by \eqref{eq:traq1}. Then $A$, $q$ satisfy assumptions (\ref{AssumCoef}) and $(A,q;\Omega)$ is non-scattering for any incident waves at any frequencies.
\end{lemma}
\begin{remark}
Notice for $A$ defined in \eqref{eq:traq1} we always have $\det A= \abs{\det D\Psi}^{2-d}\circ\Psi^{-1}$. In particular $\det A =1$ in dimension two and  $\det A = q$ in dimension three. This may shed light to the verification of optimality for results on the discreteness of transmission eigenvalues; See also \cite{LaVaSu,NgNg17}.
\end{remark}
\noindent As observed in \cite{CakoniVX23}, given any bounded Lipschitz domain $\Omega\subset\RR^d$, one can construct infinitely many diffeomorphisms that satisfy the conditions in Lemma~\ref{lem:diffeoNonscat}. For instance, let $\Phi: \RR^d\to \RR^d$ be a $\mathscr{C}^{2}$ map such that $\Phi=0$ in $\RR^d\setminus\Omega$. Then the mapping $\Psi$ defined as
	\begin{equation}\label{eq:Psi}
		\Psi= \mrm{Id} + \eps \Phi.
	\end{equation}
	 is a $\mathscr{C}^2(\overline{\Omega})$ diffeomorphism  for $\eps>0$ sufficiently small that satisfies $\Psi=\Id $ on $\partial\Omega$.\\
\newline
Next, we give two explicit non-scattering examples (where $\eps$ in \eqref{eq:Psi} does not have to be small) with $\Omega$ being a square and a disk, respectively.
\begin{example}[Geometry with corners]
	Let $\Omega=(-1,1)^2\subset\RR^2$. For a fixed constant $\alpha\in(-1/2,1/2)$, define
	\begin{equation*}
		\Psi(x,y)=\pare{x+\alpha(1-x^2)(1-y^2)\,,\, y}.
	\end{equation*}
	It can be verified that $\Psi:\overline{\Omega}\to\overline{\Omega}$ is a $\mathscr{C}^{\infty}$ diffeomorphism with $\Psi|_{\partial\Omega}=\Id $.
	Then the medium $(A,q;(-1,1)^2)$ as in \eqref{eq:traq1}
	is non-scattering for all incident waves at all frequencies.
	We can express $A$ and $q$ in a more explicit form as
\[
		A\circ\Psi(x,y)=\frac{1}{1-2\alpha x(1-y^2)}\left(\begin{array}{cc}
			\pare{1-2\alpha x(1-y^2)}^2+4\alpha^2 y^2(1-x^2)^2&-2\alpha y(1-x^2)
			\\-2\alpha y(1-x^2)&1
		\end{array}\right)
\]
	and
	\begin{equation*}
		q\circ\Psi=\frac{1}{1-2\alpha x(1-y^2)}.
	\end{equation*}
\end{example}
\begin{example}[Non-spherically stratified disk]
	Let $f\in \mathscr{C}^{2}_0(\RR^+)$ satisfy $f(1)=0$.
	Define $\Psi$, in polar coordinates, as
	\begin{equation*}
		\Psi(r,\theta)=(r,\theta+f(r)).
	\end{equation*}
	Then $\Psi: \overline{B_1}\to \overline{B_1}$ is a $\mathscr{C}^2$ diffeomorphism with $\Psi|_{\partial\Omega}=\Id $ (this latter property comes from the constraint $f(1)=0$). Moreover, $\Psi$ is orientation- and area-preserving with $\det D\Psi =1$.
As a consequence, the inhomogeneity $(A,1;B_1)$ with $A$ as in (\ref{eq:traq1}) is non-scattering for all incident waves at all frequencies.  Note that we find
	\begin{equation*}
		A(x,y)=\mrm{Id}-\frac{f'(r)}{r}\left(\begin{array}{cc}
			2xy&y^2-x^2\\y^2-x^2&-2xy
		\end{array}\right)
	+\pare{f'(r)}^2\left(\begin{array}{cc}
		y^2&-xy\\-xy&x^2
	\end{array}\right),
\quad r=\pare{x^2+y^2}^{1/2}.
	\end{equation*}
\end{example}
\paragraph{On the structure of the anisotropy induced by \Blu{diffeomorphism} transforms.}  The following two results reveal an interesting geometry property of diffeomorphisms satisfying the conditions in Lemma~\ref{lem:diffeoNonscat}. 
\begin{lemma}\label{lem:diffeoGeo}
	Let $\Omega$ and $\Psi$ be as in Lemma~\ref{lem:diffeoNonscat}. Then we must have
\[
		\PAre{\frac{D^T\Psi}{\abs{\det D\Psi}}-\mrm{Id}}\nu=0\qquad\mbox{on $\partial\Omega$}.
\]
\end{lemma}
\begin{proof}
	Let $j\in\{1,\ldots,d\}$ be fixed. Let $v(\x)=x_j$ and let $u=v\circ\Psi^{-1}$.
	Then $u$, $v$ satisfy
\[
		\div(A \nabla u)=0\AnD \Delta v=0\quad\mbox{in $\Omega$},
		\qquad
		u=v\AnD \partial_{\nu}^Au=\partial_{\nu} v\quad\mbox{on $\partial \Omega$},
\]
	with $A$ defined in \eqref{eq:traq1}.
	Furthermore, by straightforward calculations we have
	\[\partial_{\nu} v=\nu^T\vec{e}_j\AND
	\partial_{\nu}^Au=\nu^T\frac{D\Psi D^T\Psi}{\abs{\det D\Psi}}(D^T\Psi^{-1})\,\vec{e}_j
	=\nu^T\frac{D\Psi}{\abs{\det D\Psi}}\, \vec{e}_j
	\qquad\mbox{on $\partial\Omega$}.\]
	The proof is complete by taking all $j\in\{1,\ldots,d\}$.
\end{proof}
\noindent Combining this result with \cite[Theorem 2.1]{CakoniVX23}, we obtain the following statement when $\Omega$ is not smooth.
\begin{proposition}
	Let $\Omega$ be a $\mathscr{C}^{1,\alpha}$ bounded domain in $\RR^d$, and let $\Psi$ be a $\mathscr{C}^{m+2}$ (\resp  $\mathscr{C}^{\infty}$ or (real) analytic) diffeomorphism on $\overline{\Omega}$ with $m\ge 1$ satisfying $\Psi|_{\partial\Omega}=\Id $.
	Then
	\begin{equation}
		\pare{D\Psi-\mrm{Id}}\nu=0\qquad \mbox{for\quad $\x_0\in\partial\Omega$},
	\end{equation}
where $\x_0\in\partial\Omega$ is such that $\partial\Omega$ is not $\mathscr{C}^{m,\alpha}$  (\resp  $\mathscr{C}^{\infty}$ or (real) analytic) in any neighbourhood of $\x_0$.
\end{proposition}
\begin{proof}
	Let $\x_0\in\partial\Omega$ be such that $\partial\Omega$ is not $\mathscr{C}^{m,\alpha}$ (\resp  $\mathscr{C}^{\infty}$ or (real) analytic) in any neighbourhood of $\x_0$. Assume by contradiction that $(D\Psi-\mrm{Id})\nu\neq0$ at $\x_0$. Let 
\[
A=\frac{D\Psi D^T\Psi}{\abs{\det D\Psi}}\quad \mbox{ in $\overline{\Omega}$}.
\]	
Then by Lemma~\ref{lem:diffeoGeo} we have
\[
(A-\mrm{Id})\,\nu = (D\Psi-\mrm{Id})\,\nu\qquad \mbox{on $\partial\Omega$},
\]
and so $(A-\mrm{Id})\nu\ne0$ at $\x_0\in\partial\Omega$. Exploiting this, we can find a solution $v$ to the homogeneous Helmholtz equation in $\RR^d$ such that $\nu^T(A-\mrm{Id})\nabla v\neq0$ at $\x_0$. But from Lemma~\ref{lem:diffeoNonscat} we know  that $v$ is non-scattering for $(A,q;\Omega)$ with $A$, $q$ given in \eqref{eq:traq1}.
	Hence by \cite[Theorem 2.1]{CakoniVX23} we must have that $\partial\Omega$ is $\mathscr{C}^{m,\alpha}$ (\resp  $\mathscr{C}^{\infty}$ or (real) analytic) near $\x_0$, which leads to a contradiction.
\end{proof}
\begin{remark}
    In particular, we show in the proof that the matrix $A$ defined in \eqref{eq:traq1} must satisfy
$$
A\,\nu = \nu
$$
at the considered ``non-smooth'' point $\x_0$ of $\partial \Omega$
\end{remark}
\subsection{Other explicit non-scattering examples}
In this section, we are focused on non-scattering examples where the total and incident fields are identical in the whole $\mathbb{R}^d$.
\begin{example}\label{ex:Adiag}
	Let $\Omega=(0,\pi)^2$ and
	\begin{equation}\label{eq:Aqaniso}
		A=\pare{\begin{matrix}
				a_1&0\\0&a_2
		\end{matrix}},\qquad\quad
		q=q_0\qquad\mbox{in $\Omega$,}
	\end{equation}
with $a_1,a_2,q_0\in\RR_+$. Assume that these coefficients satisfy one of the following assumptions:\\[4pt]
$\bullet$ $a_1\ne1$, $a_2\ne1$ and there exist $m,n\in\NN$, with $(m,n)\ne(0,0$), such that $m^2(q_0-a_1)=n^2(a_2-q_0)$;\\[2pt]
$\bullet$ $a_1=1$, $a_2\ne1$ and $(q_0-1)(a_2-1)> 0$;\\[2pt]
$\bullet$ $a_1\ne1$, $a_2=1$ and $(q_0-1)(a_1-1)> 0$.\\[2pt]
Then $\mathscr{S}_{\mrm{NS}}$ contains an unbounded sequence which accumulates at $+\infty$.
\end{example}
\begin{remark}
    It is known for the last two cases in Example~\ref{ex:Adiag} that $\mathscr{S}_{\mrm{TE}}$ is at most countable, and hence such $(A,q;\Om)$ cannot be non-scattering for all frequencies. The same is true for the first case provided $a_1q_0\neq 1$ or $a_2q_0\neq 1$ (see \cite{LaVaSu,NgNg17}); Otherwise if $a_1q_0=1$ or $a_2q_0=1$ but $q_0\neq 1$, the discreteness of $\mathscr{S}_{\mrm{TE}}$ 
    is to our best knowledge still an open question. 
\end{remark}

\begin{proof}
Suppose that $a_1,a_2,q_0$ satisfy the first set of assumptions. 
Then it can be verified straightforwardly that for any $\kappa\in\NN^\ast$, the functions $u$, $v$ such that
\[
u(\x)=v(\x)=\cos(\kappa mx)\cos(\kappa ny)\quad\mbox{in $\Omega$}
\]
solve the interior transmission eigenvalue problem (ITEP) \eqref{eq:ITEP} with $k=\kappa\sqrt{m^2+n^2}$.\\
Now assume that $a_1,a_2,q_0$ satisfy the second set of assumptions (the third one can be dealt with similarly). Then one can verify that for any $m\in\NN^\ast$, the functions $u$, $v$ such that
\[
u(\x)=v(\x)=\Pare{c_1\cos(bx)+c_2\sin(bx)}\cos(my)\quad\mbox{in $\Omega$},
\]
with $b=\sqrt{k^2-m^2}\in\RR_+\cup i\RR_+$ and where $c_1$, $c_2$ are non zero constants, solve \eqref{eq:ITEP} for $k=m\sqrt{(a_2-1)/(q_0-1)}$.
\end{proof}
\noindent	By scaling and/or rigid change of coordinates, we can adapt Example~\ref{ex:Adiag} to the case of any rectangle, with properly modified form of $A(\x)$ (not necessarily diagonal) and conditions between $A$ and $q$.
Analogous arguments and results apply when $\Omega$ is a cuboid in $\RR^3$.\\[3pt]
\noindent The case in Example~\ref{ex:Adiag} but for $q=a_j=1$ for $j=1$ or $2$ can be generalized to the following situation where $q\equiv 1$ and $A-\mrm{Id}$ is rank deficient. In this situation we lost the discreteness of the sets $\mathscr{S}_{\mrm{NS}}$ and $\mathscr{S}_{\mrm{TE}}$. A particular case of this example appears in \cite[p.1170]{LaVaSu}. 
\begin{example}\label{ex:3.11}
	Given any bounded Lipschitz domain $\Omega$ in $\RR^d$, suppose that $A=A(\x)$ satisfies
	\begin{equation*}
		A(\x)=U\left(\begin{array}{cc}
			1&0\\0&A_{d-1}(\x)
		\end{array}\right)U^T,\qquad \x\in\overline{\Omega},
	\end{equation*}
	with some $(d-1)\times(d-1)$ symmetric positive definite matrix $A_{d-1}$ and a constant orthogonal matrix $U$. Importantly, note that in this case $1$ is an eigenvalue of $A(\x)$ for all $\x\in\overline{\Omega}$. Then any $k>0$ is a non-scattering frequency for $(A,1;\Omega)$.
\end{example}
\begin{proof}
For any non-zero constants $c_1$, $c_2$, define $v_0(\x)=v_0(x_1,x')=c_1\cos(kx_1)+c_2\sin(kx_1)$. It can be verified straightforwardly that for any bounded Lipschitz domain $\tilde\Om$ and any fixed $k>0$, $(u_0,v_0)=(v_0,v_0)$ is a pair of eigenfunctions to the ITEP \eqref{eq:ITEP} for $(\tilde{A},1;\tilde\Omega)$ with $\tilde{A}=\mathrm{diag}(1,A_{d-1})$ (note in particular that $(\tilde{A}-I)\nabla v_0\equiv 0$). Then consider the change of coordinates $\mrm{y}=U\x$ for $\x\in\RR^d$ and define $v$ such that $v(\mrm{y}):=v_0\Pare{ U^{-1}\mrm{y}}$. The pair $(u,v)=(v,v)$ satisfies \eqref{eq:ITEP} for $(A,1;\Omega)$ with $\Om=\{\mrm{y}=U\x\,|\,\x\in\tilde\Om\}$. Moreover, $v$ is an entire solution to the Helmholtz equation.	The proof is complete. 
\end{proof}
\noindent The following can be viewed as a generalization of Example~\ref{ex:Adiag} when $q=a_j\neq1$ for $j=1$ or $2$.
\begin{example}\label{ex:3.13}
	Let $\Omega=(b_1,b_2)\times D\subset \RR^d$ with $D$ an open interval if $d=2$ or a bounded Lipschitz domain in $\RR^2$ if $d=3$.
	Then there are infinitely many non-scattering frequencies for $(A,q;\Omega)$ with
	\begin{equation*}
		A(\x)=\left(\begin{array}{cc}
			a_0&0\\0&A_{d-1}(\x)
		\end{array}\right)\AND
 q(\x)=a_0,\qquad \x\in\Omega,
	\end{equation*}
	for some positive constant $a_0$ and some $(d-1)\times (d-1)$ symmetric positive definite matrix $A_{d-1}$.
\end{example}
\begin{proof}
	It can be verified straightforwardly that $v(\x)=\cos(k(x_1-b_1))$ is a non-scattering incident field with $u=v$ being the total field for $k=m\pi/|b_2-b_1|$, $m\in\mathbb{N}^\ast$.
\end{proof}
\begin{remark}
    We notice from Example~\ref{ex:3.11} that if $a_0=1$ in Example~\ref{ex:3.13}, the sets $\mathscr{S}_{\mrm{NS}}$ and hence $\mathscr{S}_{\mrm{TE}}$ coincide with $\RR^+$. However, in some other cases, for example when both $a_0$ and the smallest eigenvalue of $A_{d-1}(\x)$ are larger than $1$ for all $\x$ in a small neighbourhood of $\partial\Om$, the sets $\mathscr{S}_{\mrm{TE}}$ and hence $\mathscr{S}_{\mrm{NS}}$ are discrete \cite{BoCH11,NgNg17}. The discreteness of $\mathscr{S}_{\mrm{TE}}$ or $\mathscr{S}_{\mrm{NS}}$ in general is unknown for Example~\ref{ex:3.13} especially if $\det A=1$ on $\partial\Omega$.
\end{remark}
\noindent Finally, we note for all the examples in this section that $\nu\cdot(A-\mrm{Id})\nabla v=0$ on $\partial\Omega$, which allows us to construct non-scattering media with Lipschitz (and even less regular) boundaries. Otherwise if $\nu\cdot(A-\mrm{Id})\nabla v\neq 0$ on $\partial\Omega$, we know \Blu{from \cite{CakoniVX23,KSS24} that} $\partial\Omega$ must be sufficiently regular if $A$ and $q$ are regular in $\overline{\Omega}$.

\section*{Acknowledgments}
The work of J. Xiao was supported in part by the National Science Foundation under
Grant DMS-2307737. 

\section*{Data availability}
No data sets were generated during this study.
\section*{Conflict of interest}
The authors declare no conflict of interest in this paper.
\section*{Appendix}
The following figures illustrate how to justify the existence of analytic nodal curves as shown in Figures~\ref{fig:EPmore}--\ref{fig:EPLip0}. Below the corresponding functions are positive in the yellow regions and negative in the blue regions.\\
\newline
As an example of such justification, let us look at the nodal set of Figure~\ref{fig:AnaLoc1}. That is, when $L=2$, $\mu=5/2$ and $a=0.58$ in \eqref{eq:fR}.
Based on the ``generating function'' $J_{\mu}(kr)\cos(\mu\theta)$, we can identify the positive and negative regions of each terms $l=0,\ldots,L-1$ in \eqref{eq:fR}. See Figure~\ref{fig:AnaLocIll1} for an example for the term $l=0$ in \eqref{eq:fR}, where the dashed lines are the nodal curves of the terms $l=1,\ldots,L-1$ in \eqref{eq:fR}. Then we can find certain regions where $v$ in \eqref{eq:fR} has to be either positive/negative, say, the domains where every term is of the same sign (see Figure~\ref{fig:AnaLocIll2}). Therefore, there must be a nodal set in between the strictly positive and the strictly negative regions.
The justification for Figures~\ref{fig:AnaLoc2},~\ref{fig:AnaLoc3},~and~\ref{fig:EPintSmooth} is similar.
\begin{figure}[!ht]
	\begin{subfigure}{.29\textwidth}
		\centering
		\includegraphics[width=\textwidth]{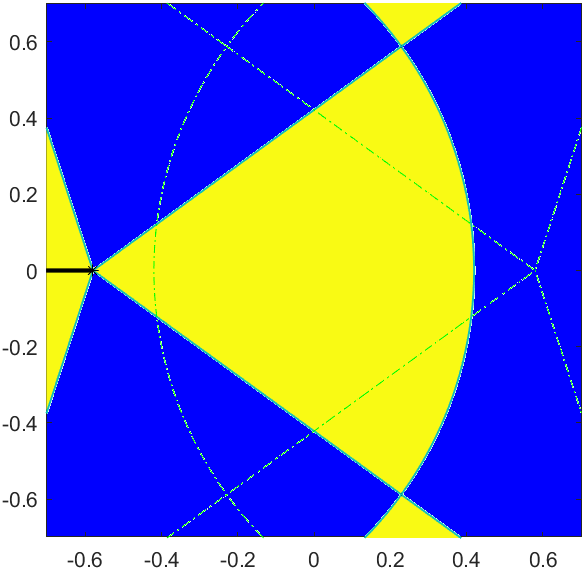}
		\caption{The term $l=0$ in \eqref{eq:fR}.}
		\label{fig:AnaLocIll1}
	\end{subfigure}\hfil
	\begin{subfigure}{.29\textwidth}
		\centering
		\includegraphics[width=\textwidth]{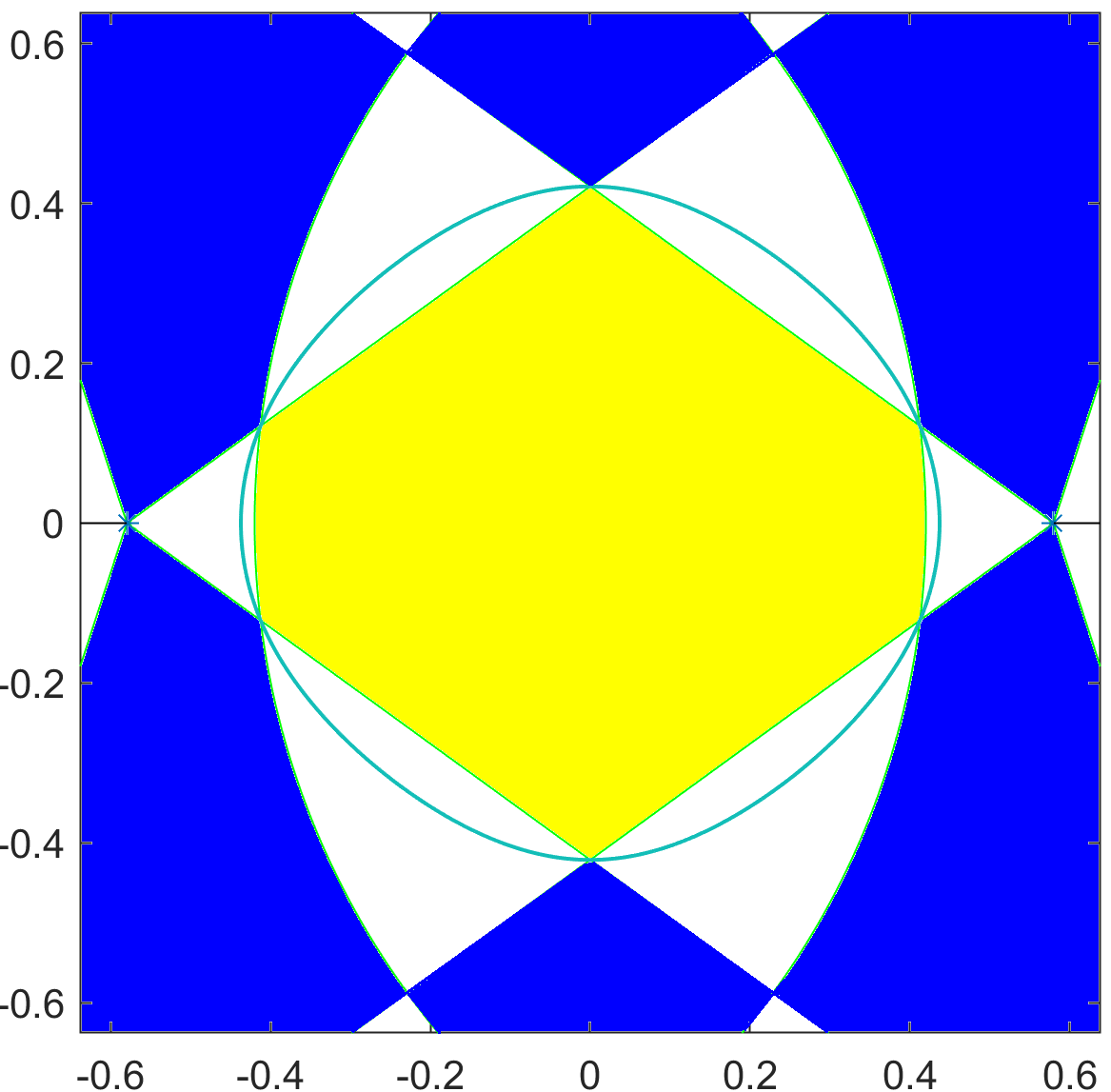}
		\caption{$v$ in \eqref{eq:fR}.}
		\label{fig:AnaLocIll2}
	\end{subfigure}\hfil
	\begin{subfigure}{.33\textwidth}
		\centering
		\includegraphics[width=\textwidth]{n=2.5L=2a=0.58.png}
		\caption{The nodal curves of $v$.}
	\end{subfigure}
	\caption{Justification for Figure~\ref{fig:AnaLoc1}: $L=2$, $\mu=5/2$, $a=0.58$.}
	\label{fig:AnaLoc1justi}
\end{figure}
\begin{figure}[!ht]
	\begin{subfigure}{.29\textwidth}
		\centering
		\includegraphics[width=\textwidth]{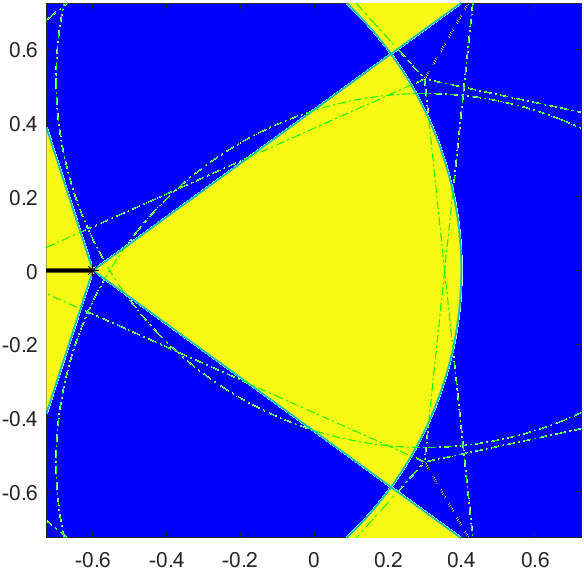}
		\caption{The term $l=0$ in \eqref{eq:fR}.}
	\end{subfigure}\hfil
	\begin{subfigure}{.29\textwidth}
		\centering
		\includegraphics[width=\textwidth]{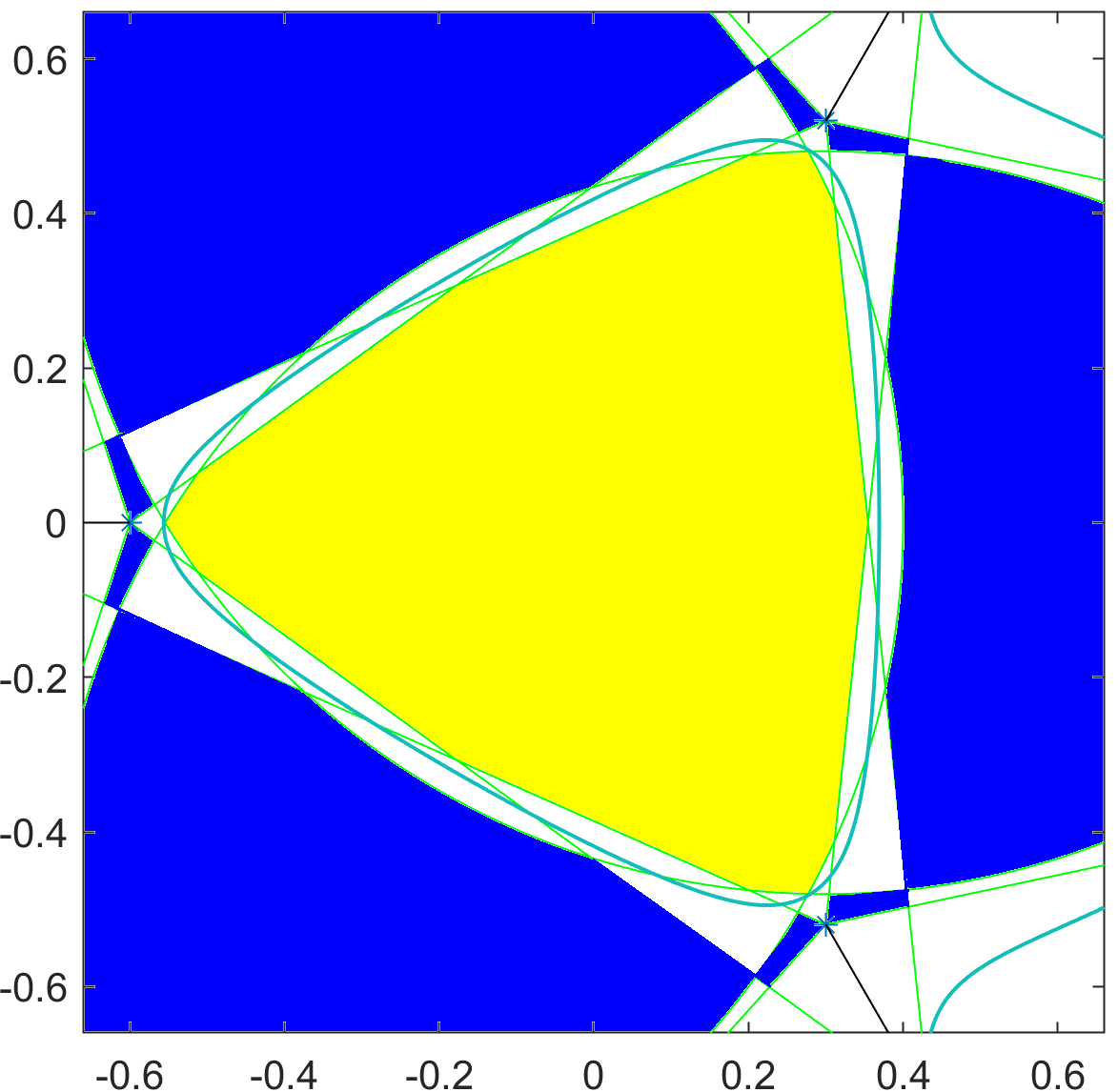}
		\caption{$v$ in \eqref{eq:fR}.}
	\end{subfigure}\hfil
	\begin{subfigure}{.33\textwidth}
		\centering
		\includegraphics[width=\textwidth]{n=2.5L=3a=0.6.png}
		\caption{The nodal curves of $v$.}
	\end{subfigure}
	\caption{Justification for Figure~\ref{fig:AnaLoc2}: $L=3$, $\mu=5/2$, $a=0.6$.}
\end{figure}
\begin{figure}[!ht]
	\begin{subfigure}{.29\textwidth}
		\centering
		\includegraphics[width=\textwidth]{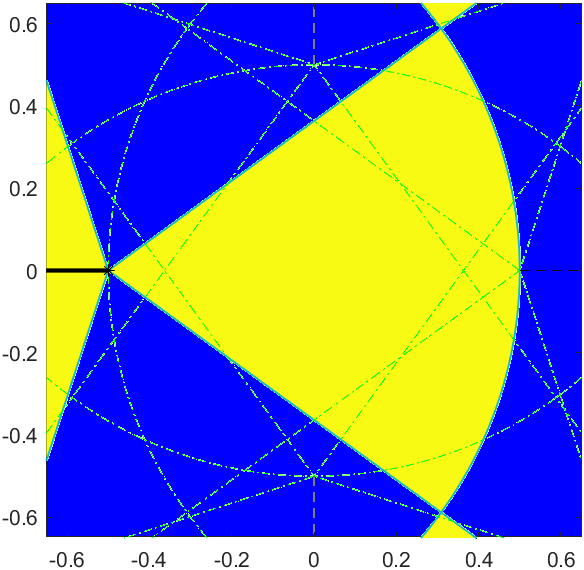}
		\caption{The term $l=0$ in \eqref{eq:fR}.}
	\end{subfigure}\hfil
	\begin{subfigure}{.29\textwidth}
		\centering
		\includegraphics[width=\textwidth]{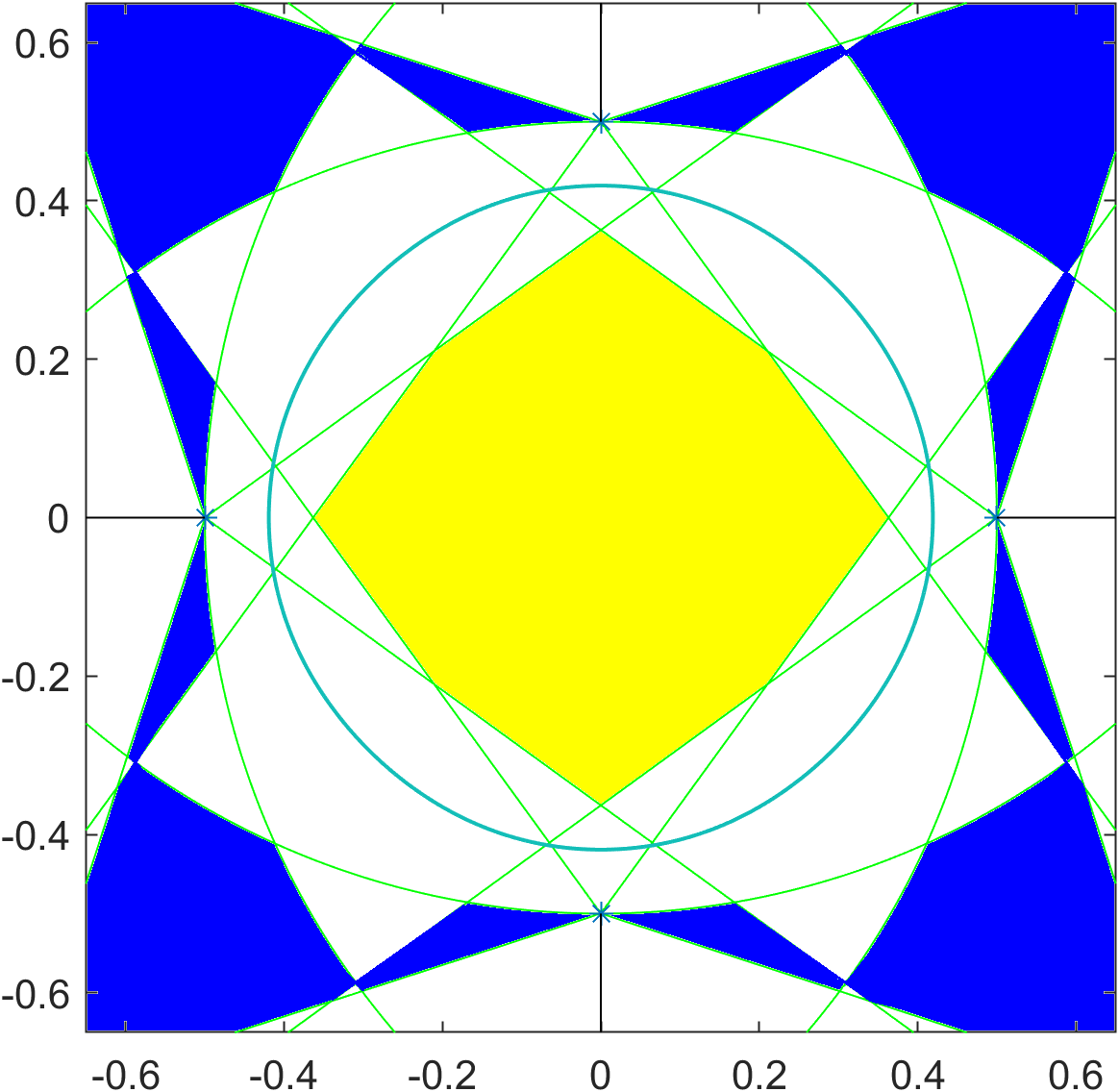}
		\caption{$v$ in \eqref{eq:fR}.}
	\end{subfigure}\hfil
	\begin{subfigure}{.33\textwidth}
		\centering
		\includegraphics[width=\textwidth]{n=2.5L=4a=0.5.png}
		\caption{The nodal curves of $v$.}
	\end{subfigure}
	\caption{Justification for Figure~\ref{fig:AnaLoc3}: $L=4$, $\mu=5/2$, $a=0.5$.}
\end{figure}
\begin{figure}[!ht]
	\begin{subfigure}{.29\textwidth}
		\centering
		\includegraphics[width=\textwidth]{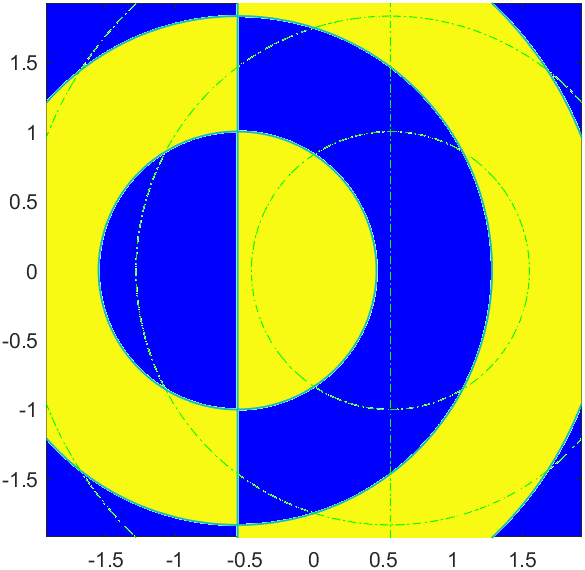}
		\caption{The term $l=0$ in \eqref{eq:fR}.}
	\end{subfigure}\hfil
	\begin{subfigure}{.29\textwidth}
		\centering
		\includegraphics[width=\textwidth]{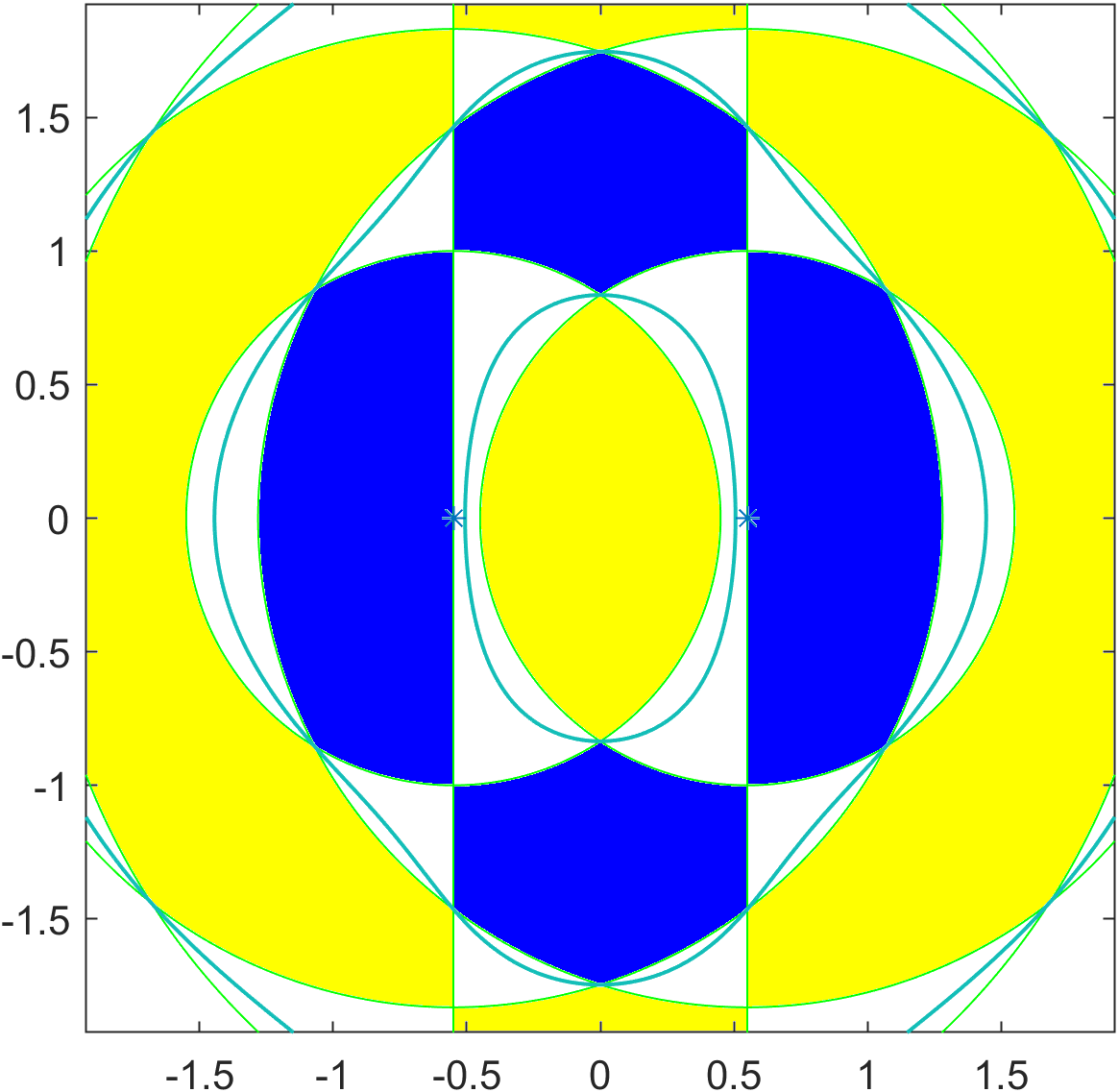}
		\caption{$v$ in \eqref{eq:fR}.}
	\end{subfigure}\hfil
	\begin{subfigure}{.33\textwidth}
		\centering
		\includegraphics[width=\textwidth]{n=1L=2a=0.55.png}
		\caption{The nodal curves of $v$.}
	\end{subfigure}
	\caption{Justification for Figure~\ref{fig:EPintSmooth}: $L=2$, $\mu=1$, $a=0.55$.}
\end{figure}

\noindent The situation for the case of Figure~\ref{fig:EPintLip4} is slightly different. First, we can identify the positive and negative regions for each term in \eqref{eq:fRgen} as before, and hence some one-sign regions for $v$ in \eqref{eq:fRgen} (see Figures~\ref{fig:EPintLip4Ill1}~and~\ref{fig:EPintLip4Ill2}).
In particular, thanks to the symmetry of the two terms in \eqref{eq:fRgen}, we have that $v=0$ on the axis $\{y=0\}$.
Additionally, we observe that $v=0$ at the two points $(0,\pm\sqrt{1-a^2})$, where in fact both terms in \eqref{eq:fRgen} are zero. Moreover, it can be verified straightforwardly that $\nabla v\neq 0$ at $(0,\pm\sqrt{1-a^2})$. So there is exactly one (real-analytic) nodal curve of $v$ passing through the point $(0,\sqrt{1-a^2})$. This curve will continue, say, to the right side of $(0,\sqrt{1-a^2})$ until it hits some point $(x_1,0)$ for some $x_1\in[1-a,1+a]$.
On the other hand, using the fact that $a=(k_2/k-1)/2$ with $k_2$ the second positive zero of $J_2$, it can be verified that $\nabla v=0$ at the two points $\pm(a+1,0)$ and $\nabla v\neq 0$  at any $(x_1,0)$ with $x_1\in(-a-1,a+1)$.
Therefore, the nodal curve of $v$ generated from $(0,\sqrt{1-a^2})$ must also pass through $(a+1,0)$. In fact, one can also verify that the Hessian of $v$ satisfies $\mathcal{H} v\neq 0$ at $\pm(a+1,0)$. Hence there are exactly two (real-analytic) nodal curves passing at $(a+1,0)$: one is the line $\{y=0\}$ and the other is the one coming from $(0,\sqrt{1-a^2})$. Applying analogous arguments for the other three quadrant we can then complete the justification of a closed (real-analytic) nodal curve of $v$.
\begin{figure}[!ht]
	\begin{subfigure}{.29\textwidth}
		\centering
		\includegraphics[width=\textwidth]{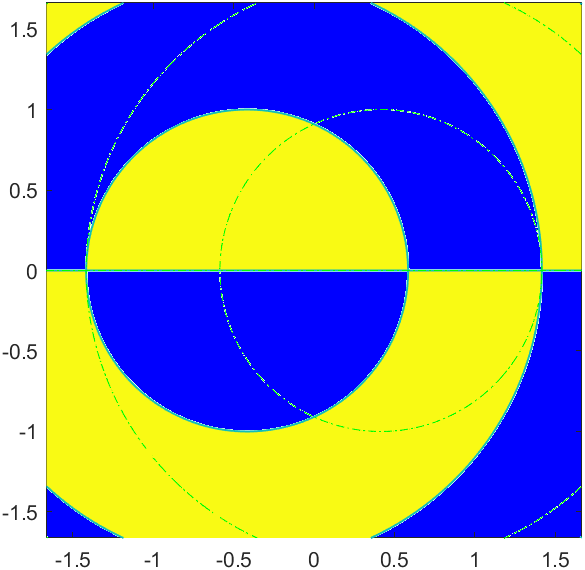}
		\caption{The term $l=0$ in \eqref{eq:fRgen}.}
		\label{fig:EPintLip4Ill1}
	\end{subfigure}\hfil
	\begin{subfigure}{.29\textwidth}
		\centering
		\includegraphics[width=\textwidth]{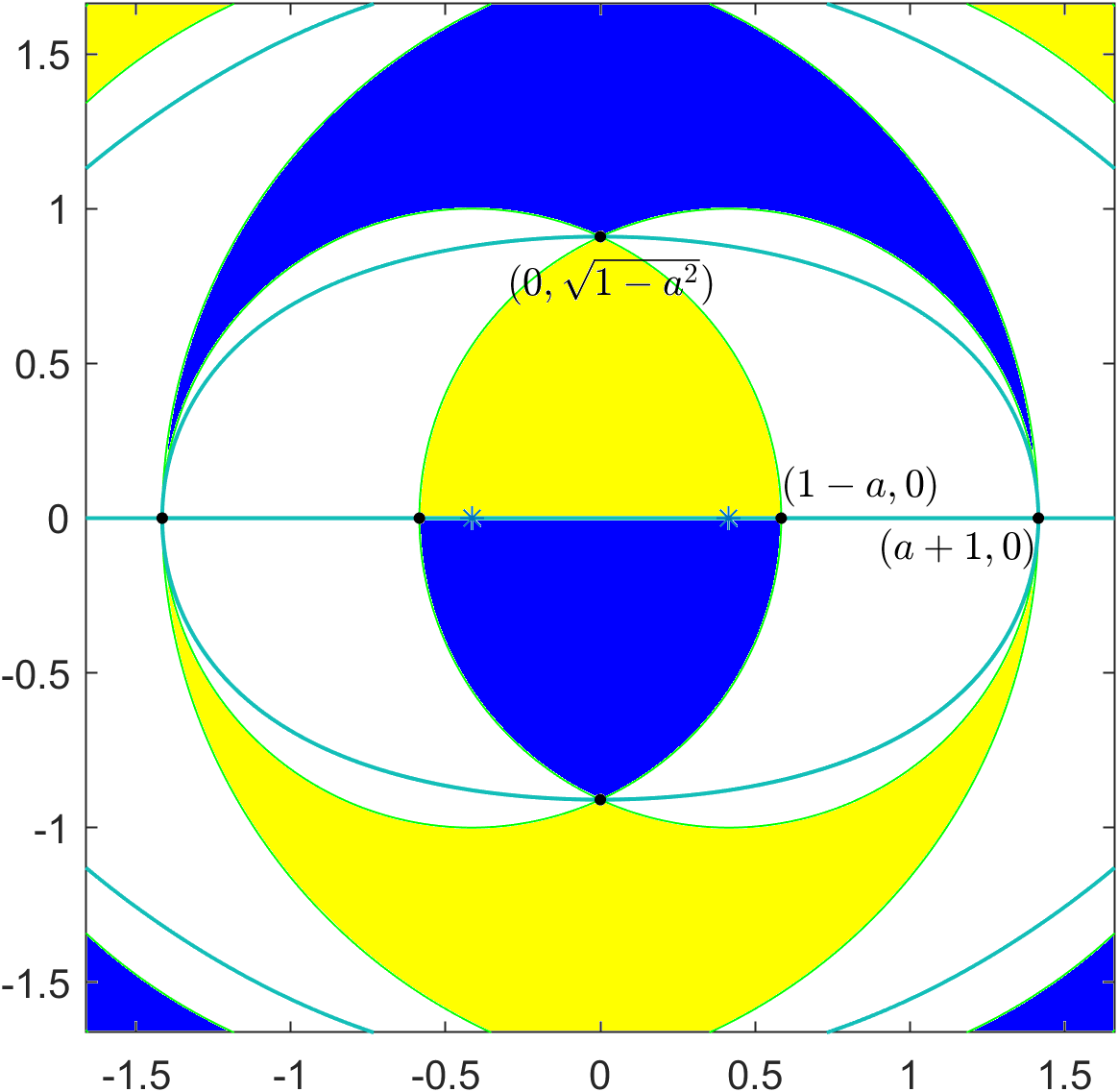}
		\caption{$v$ in \eqref{eq:fRgen}.}
		\label{fig:EPintLip4Ill2}
	\end{subfigure}\hfil
	\begin{subfigure}{.33\textwidth}
		\centering
		\includegraphics[width=\textwidth]{n=1L=2a=0.42special.png}
		\caption{The nodal curves of $v$.}
	\end{subfigure}
	\caption{Justification for Figure~\ref{fig:EPintLip4}: $L=2$, $\mu=1$, $a=\frac{{k_2}-{k}}{2k}$, $\phi_1=\pi/2=-\phi_2$.}
\end{figure}~

\noindent In Figure~\ref{fig:EPintLip5}, we make use of the function $v$ in \eqref{eq:fRgen} with $\mu=1$, $L=2$, $a=\sqrt{2}/2$, $\phi_1=\pi/4$ and $\phi_2=3\pi/4$. The choice of $\phi_1$ and $\phi_2$ ensures the symmetry of $v$ with respect to the $y$-axis. In particular, $v=0$ on $\{x=0\}$. In addition, thanks to the value of $a$, one can verify that at the point $P_0=(0,-\sqrt{2}/2)$ there holds $v=\partial_xv=\partial_yv=\partial^2_{ij}v=0$ for all $i,j\in\{x,y\}$ but $\partial_x^3v\neq 0$. Hence, there are exactly three nodal curves passing through $(0,-\sqrt{2}/2)$.
Furthermore, one can verify straightforwardly that $v=0$ while $\nabla v\neq 0$ at the points $P_j$, $j=1,\ldots,9$, where the coordinates of each point can be calculated explicitly in terms of $a$, $k$ and $k_2$.
Finally, based on the locations of strictly positive and negative regions we can conclude to the existence of real-analytic nodal curves for $v$ that enclose some bounded Lipschitz domains, as shown Figure~\ref{fig:EPintLip50}.
\begin{figure}[!ht]
	\begin{subfigure}{.29\textwidth}
		\centering
		\includegraphics[width=\textwidth]{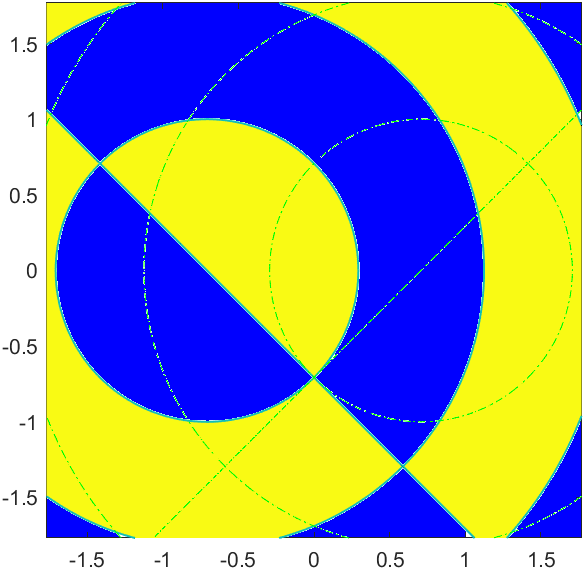}
		\caption{The term $l=0$ in \eqref{eq:fRgen}.}
	\end{subfigure}\hfil
	\begin{subfigure}{.29\textwidth}
		\centering
		\includegraphics[width=\textwidth]{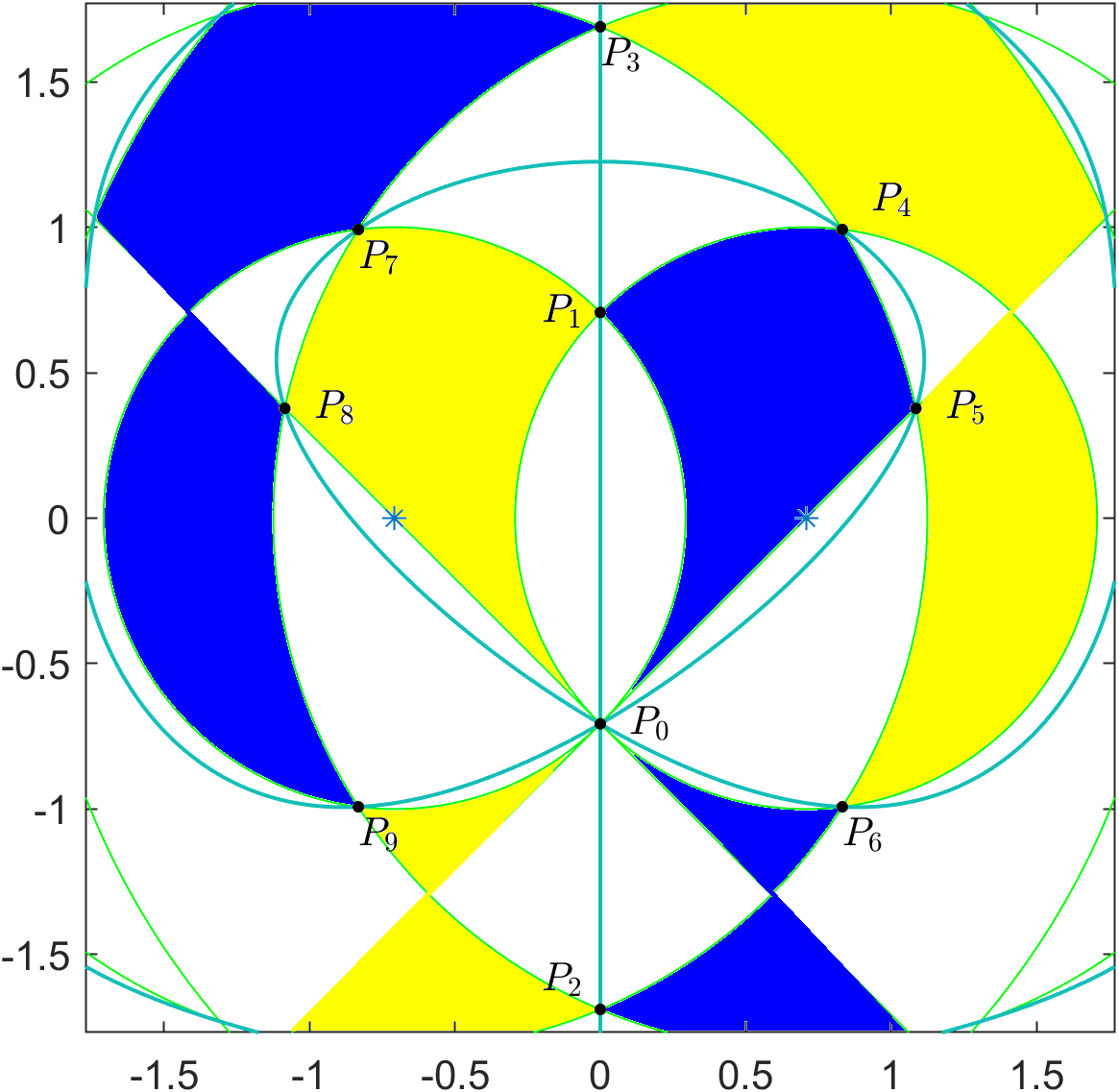}
		\caption{$v$ in \eqref{eq:fRgen}.}
	\end{subfigure}\hfil
	\begin{subfigure}{.33\textwidth}
		\centering
		\includegraphics[width=\textwidth]{n=1L=2a=0.71special.png}
		\caption{The nodal curves of $v$.}
		\label{fig:EPintLip50}
	\end{subfigure}
	\caption{Justification for Figure~\ref{fig:EPintLip5}: $L=2$, $\mu=1$, $a=\sqrt{2}/2$, $\phi_1=\pi/4$, $\phi_2=3\pi/4$.}
\end{figure}

\noindent The justification of the results appearing in Figure~\ref{fig:EPLip0} is similar. 
\begin{figure}[!ht]
	\begin{subfigure}{.29\textwidth}
		\centering
		\includegraphics[width=\textwidth]{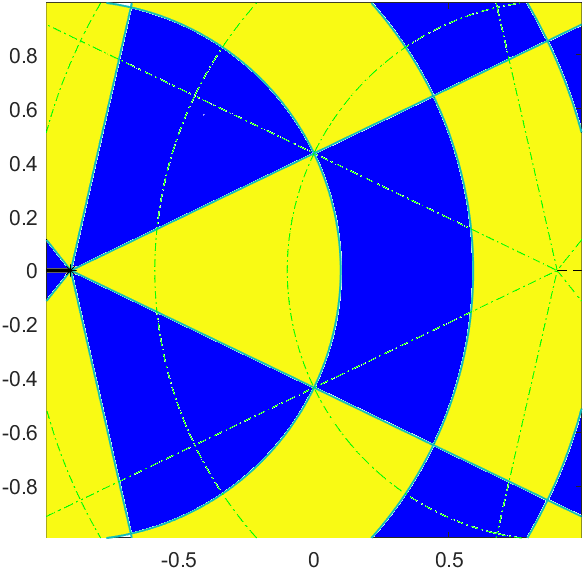}
		\caption{The term $l=0$ in \eqref{eq:fRgen}.}
	\end{subfigure}\hfil
	\begin{subfigure}{.29\textwidth}
		\centering
		\includegraphics[width=\textwidth]{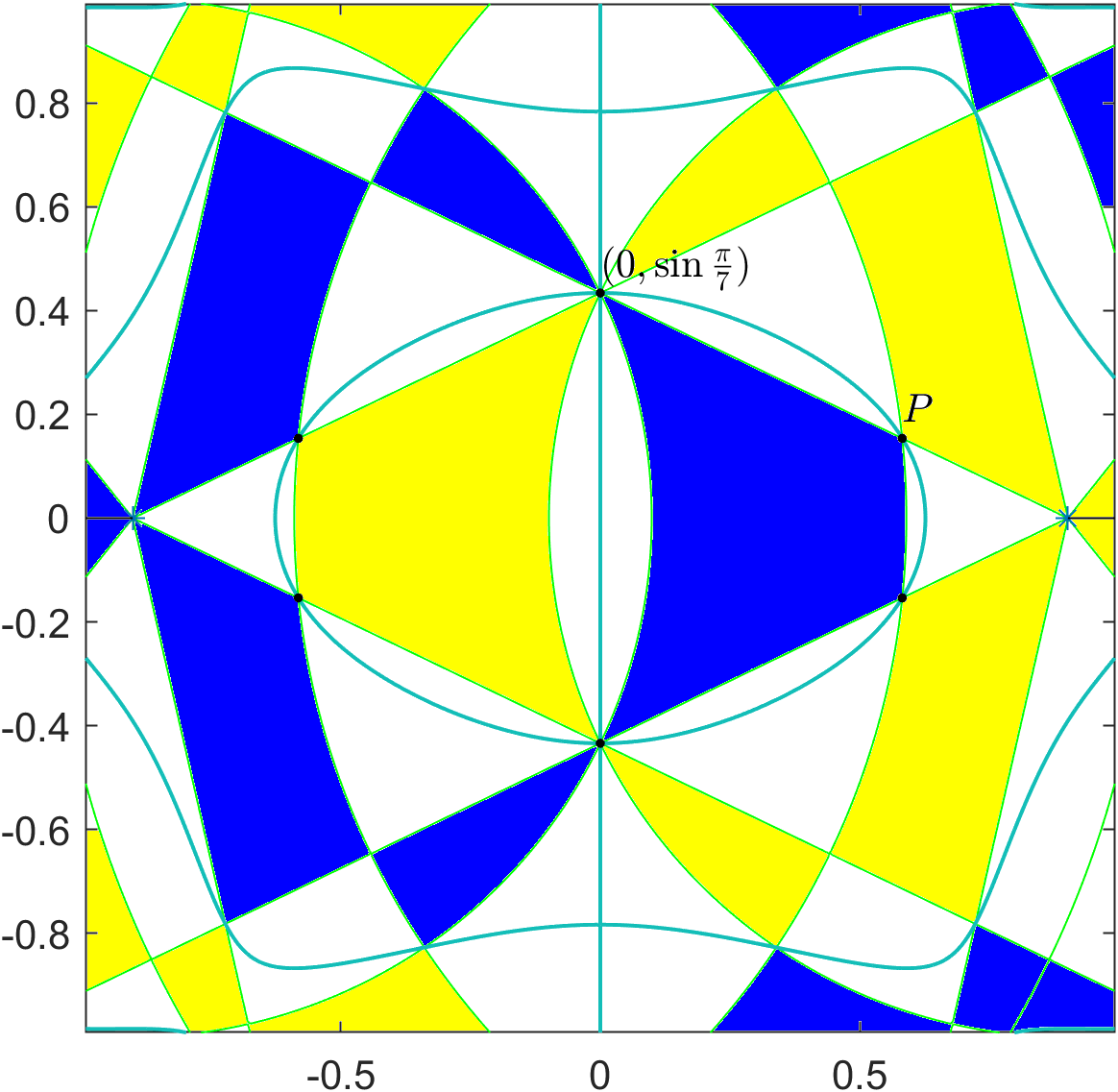}
		\caption{$v$ in \eqref{eq:fRgen}}
	\end{subfigure}\hfil
	\begin{subfigure}{.33\textwidth}
		\centering
		\includegraphics[width=\textwidth]{n=3.5L=2a=0.9special.png}
		\caption{The nodal curves of $v$.}
	\end{subfigure}
	\caption{Justification for Figure~\ref{fig:EPLip}: $L=2$, $\mu=7/2$, $a=\cos(\pi/7)$, $b_0=1=-b_1$.}
\end{figure}
\begin{figure}[!ht]
	\begin{subfigure}{.29\textwidth}
		\centering
		\includegraphics[width=\textwidth]{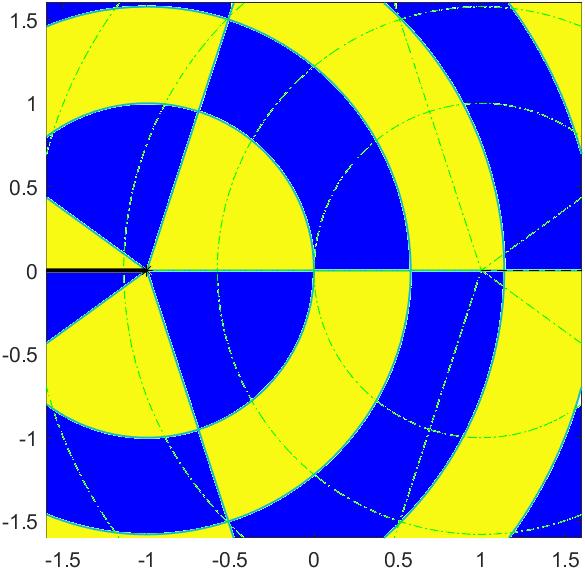}
		\caption{The term $l=0$ in \eqref{eq:fRgen}.}
	\end{subfigure}\hfil
	\begin{subfigure}{.29\textwidth}
		\centering
		\includegraphics[width=\textwidth]{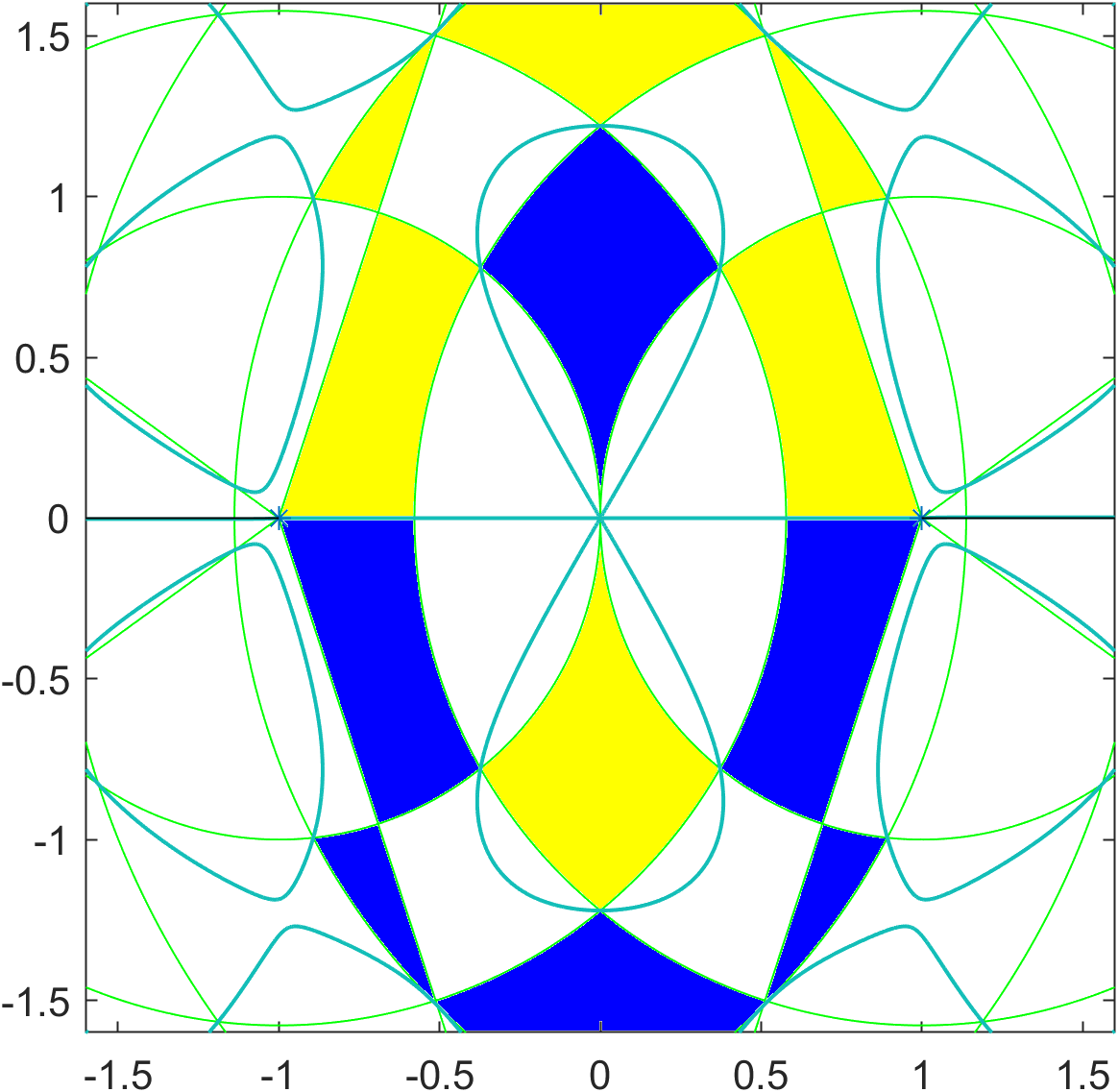}
		\caption{$v$ in \eqref{eq:fRgen}.}
	\end{subfigure}\hfil
	\begin{subfigure}{.33\textwidth}
		\centering
		\includegraphics[width=\textwidth]{n=2.5L=2a=1special.png}
		\caption{The nodal curves of $v$.}
	\end{subfigure}
	\caption{Justification for Figure~\ref{fig:EPLip2}: $L=2$, $\mu=5/2$, $a=1$, $\phi_0=\pi/5=-\phi_1$.}
\end{figure}
\begin{figure}[!ht]
	\begin{subfigure}{.29\textwidth}
		\centering
		\includegraphics[width=\textwidth]{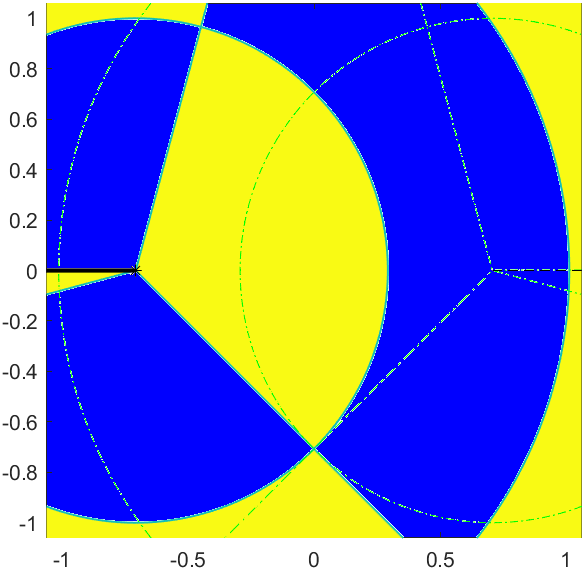}
		\caption{The term $l=0$ in \eqref{eq:fRgen}.}
	\end{subfigure}\hfil
	\begin{subfigure}{.29\textwidth}
		\centering
		\includegraphics[width=\textwidth]{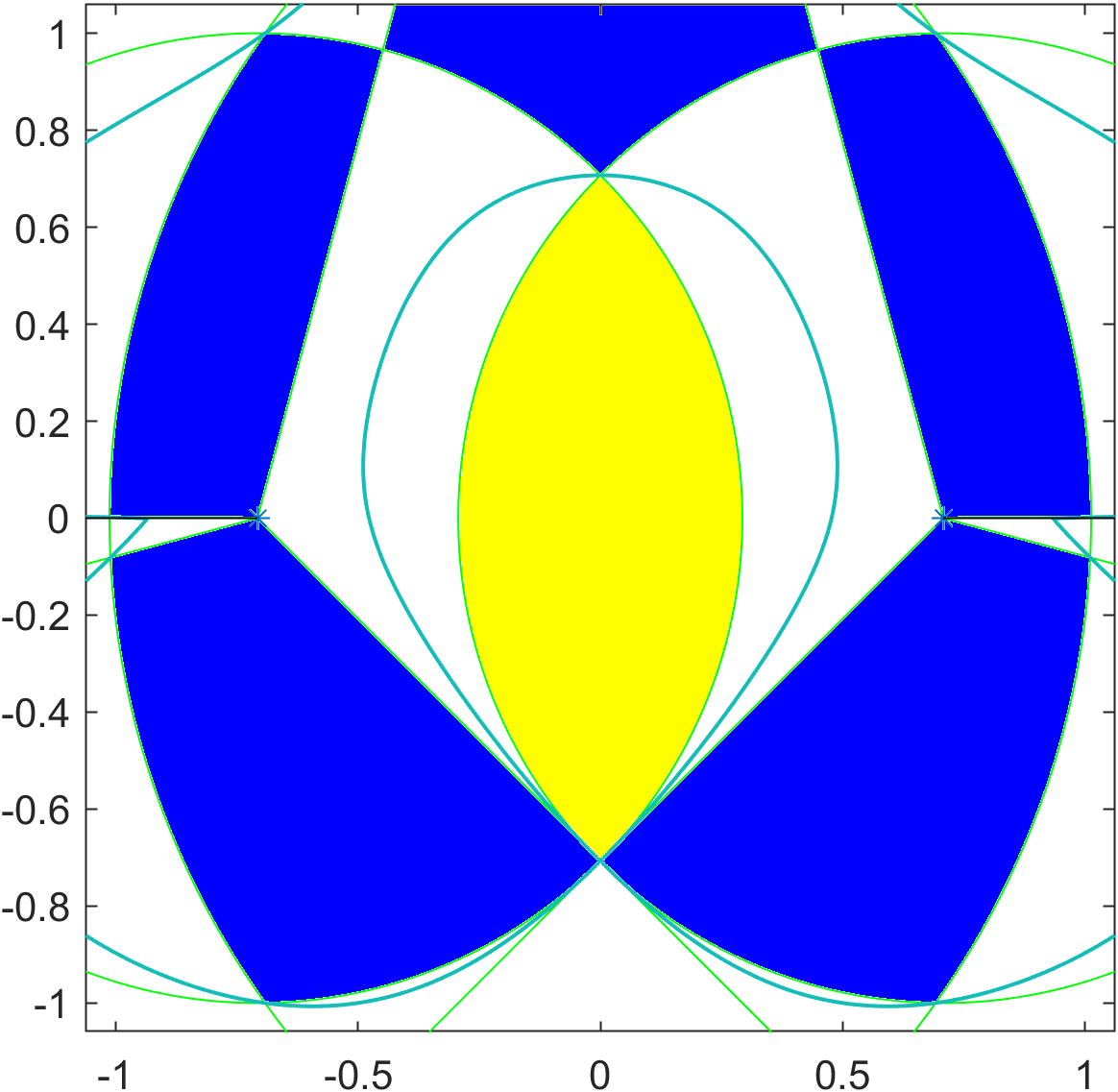}
		\caption{$v$ in \eqref{eq:fRgen}.}
	\end{subfigure}\hfil
	\begin{subfigure}{.33\textwidth}
		\centering
		\includegraphics[width=\textwidth]{n=1.5L=2a=0.71special.png}
		\caption{The nodal curves of $v$.}
	\end{subfigure}
	\caption{Justification for Figure~\ref{fig:EPLip3}: $L=2$, $\mu=3/2$, $a=\sqrt{2}/2$, $\phi_0=\pi/12=-\phi_1$.}
\end{figure}

\clearpage

	\bibliographystyle{abbrv}
	\bibliography{Biblio}
\end{document}